\documentclass[10pt,reqno,twoside]{amsart}
\usepackage{hyperref}
\usepackage[ngerman,english]{babel}
\usepackage[nameinlink,capitalise,noabbrev]{cleveref}
\hypersetup{
	colorlinks=true,
    linkcolor=blue,
    citecolor=blue,
    filecolor=black,
    urlcolor=black,
}
\usepackage{verbatim}
\usepackage{xy}
\usepackage{amsfonts}
\input xy
\xyoption{all}
\usepackage[shortlabels]{enumitem}
\usepackage{tikz-cd}
\usepackage[top=1.1in,bottom=1.1in, left=1in, right=1in]{geometry}
\usepackage{amssymb}
\usepackage{url}
\usepackage[toc,page]{appendix} 

\newtheorem{Theorem}[equation]{Theorem}
\newtheorem{Cor}[equation]{Corollary}
\newtheorem{Lemma}[equation]{Lemma}
\newtheorem{Prop}[equation]{Proposition}

\newtheorem{thmx}{Theorem}
\theoremstyle{definition}
\newtheorem{Def}[equation]{Definition}

\newtheorem{Remark}[equation]{Remark}
\newtheorem{Example}[equation]{Example}
\newtheorem{DefNot}[equation]{Definition/Notation}
\newtheorem*{Def*}{Definition}
\newtheorem*{Example*}{Example}
\newtheorem*{Theorem*}{Theorem}
\newtheorem*{Prop*}{Proposition}
\newtheorem*{Remark*}{Remark}

\newcommand{\R}{\mathbb{R}}
\newcommand{\Q}{\mathbb{Q}}
\newcommand{\Z}{\mathbb{Z}}
\newcommand{\N}{\mathbb{N}}
\newcommand{\C}{\mathbb{C}}
\newcommand{\T}{\mathbb{T}}
\newcommand{\U}{\mathcal{U}}
\newcommand{\bL}{\textbf{L}}
\newcommand{\F}{\mathbb{F}}
\newcommand{\cC}{\mathcal{C}}

\newcommand{\MUP}{\mathbf{MUP}}
\newcommand{\MOP}{\mathbf{MOP}}
\newcommand{\MU}{\mathbf{MU}}
\newcommand{\MO}{\mathbf{MO}}
\newcommand{\Ze}{\mathbb{Z}_e}
\newcommand{\Fe}{(\mathbb{F}_2)_e}
\newcommand{\CPU}{\mathbb{C}P(\U^{\C}_A)}
\newcommand{\RPU}{\mathbb{R}P(\U^{\R}_A)}
\newcommand{\Ga}{\mathbb{G}_a}
\newcommand{\Gm}{\mathbb{G}_m}

\newcommand{\upi}{\underline{\pi}}

\newcommand{\xr}{\xrightarrow}

\newcommand{\wL}{\widetilde{\bL}}

\newcommand{\hotimes}{\hat{\otimes}}
\newcommand{\el}{\text{el}_2}
\newcommand{\cD}{\mathcal{D}}
\newcommand{\wLt}{\widetilde{\bL}^{\raisebox{-3pt}{$\scriptstyle 2-\tor$}}}

\DeclareMathOperator{\map}{map}
\DeclareMathOperator{\diag}{diag}

\DeclareMathOperator{\tor}{tor}
\DeclareMathOperator{\im}{im}

\DeclareMathOperator{\Hom}{Hom}

\DeclareMathOperator{\ind}{ind}

\DeclareMathOperator{\hocolim}{hocolim}

\DeclareMathOperator{\res}{res}

\DeclareMathOperator{\Vect}{Vect}
\DeclareMathOperator{\id}{id}
\DeclareMathOperator{\pr}{pr}
\DeclareMathOperator{\toral}{toral}
\DeclareMathOperator{\reg}{reg}
\newcommand{\V}{\Vect_{\mathbb{F}_2}}

\numberwithin{equation}{section}

\title{Global group laws and equivariant bordism rings}
\author{Markus Hausmann}
\address{Matematiska institutionen, Stockholms Universitet, 106 91 Stockholm, Sweden}
\email{markus.hausmann@math.su.se}

\begin{document}
\begin{abstract}
For every abelian compact Lie group $A$, we prove that the homotopical $A$-equivariant complex bordism ring, introduced by tom Dieck (1970), is isomorphic to the $A$-equivariant Lazard ring, introduced by Cole--Greenlees--Kriz (2000). This settles a conjecture of Greenlees. We also show an analog for homotopical real bordism rings over elementary abelian $2$-groups. Our results generalize classical theorems of Quillen (1969) on the connection between non-equivariant bordism rings and formal group laws, and extend the case $A=C_2$ due to Hanke--Wiemeler (2018).

We work in the framework of global homotopy theory, which is essential for our proof. In addition to the statements for a fixed group $A$, we also prove a global algebraic universal property that characterizes the collection of all equivariant complex bordism rings simultaneously. We show that they form the universal contravariant functor from abelian compact Lie groups to commutative rings that is equipped with a coordinate; the coordinate is given by the universal Euler class at the circle group. More generally, the ring of $n$-fold cooperations of equivariant complex bordism is shown to be universal among functors equipped with a strict $n$-tuple of coordinates.
\end{abstract}	
	
\maketitle

\vspace{-0.5cm}

\section{Introduction}
In \cite{Qui69}, Quillen proved that the complex bordism ring $MU_*$ is isomorphic to the Lazard ring, and that the formal group law associated to the complex orientation of $MU$ is the universal one. This relationship between the algebra of formal group laws and the structure of the stable homotopy category has proved to be one of the main organizational principles for understanding the latter.
 For example, Quillen's theorem has been central for the construction of many new cohomology theories, the classification of thick subcategories of finite spectra \cite{HS98}, computations of homotopy groups of spheres via the Adams--Novikov spectral sequence \cite{Nov67} and the solution of the Kervaire invariant 1 problem \cite{HHR16}.

In view of the importance of Quillen's theorem, much work has been put into obtaining a similar understanding of equivariant bordism rings and their characterization in terms of formal group law data (see, e.g., \cite{tD70,BH72,Lof73,Com96,Kri99,CGK00, Gre01, Str01, Sin01, Han05,Str11, AK15,HW18,Uri18}); motivated for example by geometric questions about almost complex $G$-manifolds, the construction of equivariant higher chromatic $K$- and $E$-theory spectra or equivariant elliptic cohomology \cite{Lur18}, computations of $MU_*BG$ via the localization theorems of Greenlees--May \cite{GM97}, or for studying the global structure of the category of finite $G$-spectra \cite{BS17b,BHN^+17,BGH20}.

In the case where $G=A$ is an abelian compact Lie group, the notion of an $A$-equivariant formal group law was introduced in \cite{CGK00} and further studied in \cite{Gre01,Str11}. As in the non-equivariant case there exists a universal $A$-equivariant formal group law defined over an $A$-equivariant Lazard ring $L_A$; and a map
\[ L_A\to  \pi^A_*(MU_A)\]
classifying the Euler class of the tensor product of $A$-equivariant complex line bundles. Here, $MU_A$ is tom Dieck's homotopical $A$-equivariant complex bordism spectrum \cite{tD70}. In \cite{Gre,Gre01}, Greenlees showed that when $A$ is finite abelian, the map $L_A\to  \pi^A_*(MU_A)$ is surjective with Euler-torsion and infinitely Euler-divisible kernel, and conjectured that it is an isomorphism for all abelian compact Lie groups $A$. Despite this strong evidence, the only non-trivial group for which the conjecture has been proved to date is the cyclic group of order $2$, by work of Hanke--Wiemeler \cite{HW18} who made use of an explicit presentation of $\pi^{C_2}_*(MU_{C_2})$ due to Strickland \cite{Str01}. The main result of this paper is a proof of the conjecture in full~generality:
\begin{thmx} \label{thm:A} The map $L_A\to \pi^A_*(MU_A)$ is an isomorphism for every abelian compact Lie group $A$.
\end{thmx}
We obtain a similar result for homotopical real bordism $MO_A$ introduced in \cite{BH72}. Here the analog of the family of abelian compact Lie groups is that of elementary abelian $2$-groups, since those are the groups whose irreducible real representations are all $1$-dimensional (cf. \Cref{rem:whyabelian} below). Let $L^{2-\tor}_A$ be the ring carrying the universal $A$-equivariant $2$-torsion formal group law, i.e., the universal $A$-equivariant formal group law with trivial $2$-series. The universal real orientation of $MO_A$ induces an $A$-equivariant $2$-torsion formal group law over $\pi_*^A(MO_A)$, giving rise to a map $L^{2-\tor}_A\to \pi_*^A(MO_A)$.
\begin{thmx} \label{thm:B}
The map $L^{2-\tor}_A\to \pi_*^A(MO_A)$ is an isomorphism for every elementary abelian $2$-group $A$.	
\end{thmx}
Again, the underlying non-equivariant theorem that $MO_*$ carries the universal $2$-torsion formal group law is due to Quillen \cite{Qui69}.

\subsection{Global formulation} Our main tool in proving Theorems \ref{thm:A} and \ref{thm:B} is the global structure of equivariant bordism, i.e.,  the fact that the collections of the various $MU_A$ and $MO_A$ assemble to global equivariant spectra $\MU$ and $\MO$ in the sense of Schwede \cite{Sch18}.  In addition to its essential role in the proofs of Theorems \ref{thm:A} and \ref{thm:B}, the global framework also gives rise to a universal property of the collection of equivariant bordism rings that is simpler to state than the one for a fixed group $A$.

To describe this universal property, we recall from \cite[Section 4.2]{Sch18} that the existence of a global complex bordism theory $\MU$ implies that the collection of homotopy groups $\{\pi_*^A(MU_A)\}_{A \text{ abelian}}$ has restriction maps along all group homomorphisms; it assembles to a contravariant functor from the category of abelian compact Lie groups to commutative rings. We call this structure an $Ab$-algebra, and write $\upi_*(\MU)$ for the graded $Ab$-algebra given by the coefficients of the various $MU_A$. Moreover, given an $Ab$-algebra $X$, we call an element $e\in X(\T)$ at the circle group $\T$ a \emph{coordinate} if for every torus $A$ and split surjective character $V\colon A\to \T$ the sequence
\[ 0\to X(A)\xr{V^*(e)\cdot} X(A) \xr{\res_{\ker(V)}^A} X(\ker(V)) \to 0 \]
is exact. Here, the first map denotes multiplication with the pullback of $e$ along $V$, and the second map is the restriction map to the kernel of $V$.

A complex orientation of a global ring spectrum $E$ gives rise to an Euler class $e_{\tau}\in \pi_{-2}^{\T}(E)$ for the tautological $1$-dimensional complex $\T$-representation $\tau$. The cofiber sequence of based $\T$-spaces
\[ \T_+\to S^0 \to S^{\tau} \]
(together with its pullbacks along split surjections $V\colon A\to \T$) implies that $e_{\tau}$ defines a coordinate on $\upi_*(E)$, see \Cref{sec:complextopological} for details. In particular, this applies to the universal global complex oriented theory $\MU$, and hence $\upi_*(\MU)$ carries a canonical coordinate $e_{\tau}$.

\begin{thmx} \label{thm:C} The pair $(\upi_*(\MU),e_{\tau})$ is the universal $Ab$-algebra with a coordinate, i.e., for any other $Ab$-algebra $X$ with coordinate $e\in X(\T)$ there exists a unique map of $Ab$-algebras $\upi_*(\MU)\to X$ sending $e_{\tau}$ to~$e$.
\end{thmx}
We propose a pair consisting of an $Ab$-algebra $X$ and a coordinate $e\in X(\T)$ as a definition for a \emph{global group law}, where we leave out the word `formal' since the rings involved are not necessarily complete. Hence with this definition, $(\upi_*(\MU),e_\tau)$ is the universal global group law. To be precise, $(\upi_*(\MU),e_\tau)$ is more naturally the universal \emph{graded} global group law, while the $0$-th homotopy ring $\upi_0(\MUP)$ of periodic global complex bordism with $e_{\tau}$ shifted to degree $0$ is the universal ungraded global group law, see Sections \ref{sec:aeqgrading} and \ref{sec:gradings} for a discussion of gradings. Every global group law $(X,e)$ has an associated $A$-equivariant formal group law $X^{\wedge}_A$ over $X(A)$ for all abelian compact Lie groups $A$, and these are suitably compatible under group homomorphisms (\Cref{sec:complexadjunction}). Global group laws can be interpreted as an uncompleted version of ordinary formal group laws and are related to $1$-dimensional linear algebraic groups and (not necessarily connected) $p$-divisible groups. We refer to \Cref{sec:examples} for a list of examples.

In fact, the coordinate of $\upi_*(\MU)$ has even stronger regularity properties than those required of a global group law: The sequence
\[ 0\to \pi_*^A(MU_A)\xr{V^*(e)\cdot} \pi_*^A(MU_A) \xr{\res_{\ker(V)}^A} \pi_*^{\ker(V)}(MU_{\ker(V)}) \to 0 \]
is exact for every abelian compact Lie group $A$ and surjective character $V\colon A\to \T$, not necessarily split. In terms of global group laws, we say that $\upi_*(\MU)$ is \emph{regular}, see Sections \ref{sec:regularfgl} and \ref{sec:complextopological}. These exact sequences played a central role in Sinha's work \cite{Sin01}, where he gave a description of $\pi_*^{\T}(MU_\T)$ in terms of chosen sections of the restriction maps $\pi_*^{\T}(MU_\T)\to \pi_*^{C_n}(MU_{C_n})$. \Cref{thm:C} shows that when interpreted in the global equivariant setting, these exact sequences characterize the homotopy groups of~$\MU$.

We also obtain a similar statement for $\upi_*(\MO)$, replacing the family of abelian compact Lie groups by the family of elementary abelian $2$-groups $\el$, and the circle $\T$ by the cyclic group $C_2$: Let $X$ be an $\el$-algebra, i.e., a contravariant functor from $\el$ to commutative rings. A \emph{coordinate} on $X$ is an element $e\in X(C_2)$ such that for every surjective character $V:A\to C_2$ the sequence
\[ 0\to X(A)\xr{V^*(e)\cdot} X(A) \xr{\res_{\ker(\alpha)}^A} X(\ker(\alpha)) \to 0 \]
	is exact. (Note that for elementary abelian $2$-groups there is no difference between surjective and split surjective.) Again, the real orientation of a real oriented global ring spectrum $E$ (such as $\MO$) gives rise to a universal Euler class $e_{\sigma}\in \pi_{-1}^{C_2}(E)$ for the real sign-representation $\sigma$ of $C_2$. The $C_2$-cofiber sequence
	\[ (C_2)_+\to S^0 \to S^{\sigma}\]
implies that it defines a coordinate on $\upi_*(E)$. We then have:
\begin{thmx} \label{thm:D} The pair $(\upi_*(\MO),e_{\sigma})$ is universal among $\el$-algebras equipped with a coordinate.
\end{thmx}
In light of this, we view a pair of an $\el$-algebra and a coordinate as a global version of a $2$-torsion formal group law, which is again not necessarily complete. Every global $2$-torsion group law $(X,e)$ has an associated family of compatible $A$-equivariant $2$-torsion formal group laws $X^{\wedge}_A$, see \Cref{sec:realadjunction1}.

Finally, we also give a description of the $n$-fold cooperations $\upi_*(\MU^{\wedge n})$, which are relevant for setting up an Adams-Novikov spectral sequence in this framework. We say that an $n$-tuple of coordinates $(e^{(1)},\hdots,e^{(n)})$ of an $Ab$-algebra $X$ is \emph{strict} if each $e^{(i)}$ is of the form $\lambda_i e^{(1)}$ for a unit $\lambda_i\in X(\T)$ which restricts to $1$ at the trivial group. The inclusion maps $i_1,\hdots,i_n\colon \MU\to \MU^{\wedge n}$ give rise to a strict $n$-tuple $(e_{\tau}^{(1)},\hdots,e_{\tau}^{(n)})$ of coordinates for $\upi_*(\MU)$. We show the following:
\begin{thmx} \label{thm:E}
The tuple $(\upi_*(\MU^{\wedge n});e_{\tau}^{(1)},\hdots,e_{\tau}^{(n)})$ is universal among $Ab$-algebras equipped with a strict $n$-tuple of coordinates.
\end{thmx}
There is also a similar interpretation of $\pi_*^A(MU_A^{\wedge n})$ at every fixed group $A$ in terms of strict isomorphisms of $A$-equivariant formal group laws, see \Cref{sec:cooperations}.

\begin{Remark} Due to the failure of generic equivariant transversality, the homotopical bordism ring $\pi_*^A(MU_A)$ does not agree with the geometrically defined complex bordism ring $\Omega^A_*$ for any non-trivial $A$. However, there is a Thom-Pontrjagin map $\Omega^A_*\to \pi_*^A(MU_A)$ which was shown to be injective (in the abelian case considered here) by tom Dieck \cite{tD70}, L\"{o}ffler \cite{Lof73} and Comeza\~{n}a \cite{Com96}. The ring $\pi_*^A(MU_A)$ can also be described as a stabilization of geometric $A$-equivariant bordism, see \cite{BH72}.
	
When $A$ is a torus, Hanke \cite{Han05} proved a more precise relationship between $\Omega^A_*$ and $\pi_*^A(MU_A)$ in terms of a pullback square involving certain classes $Y_{V,d}$. It is straightforward to express these classes in terms of formal group law data, see \Cref{rem:hankeclasses} for the real analog. Hence, combined with Hanke's result, the results of this paper also give an algebraic description of the geometric bordism rings $\Omega^A_*$ for tori, though this is not entirely satisfactory since $\pi_*^A(MU_A)$ has a universal property for maps out of rather than into it.

The same discussion applies to the geometric real bordism ring for elementary abelian $2$-groups, for which the analog of Hanke's result has been obtained by Firsching \cite{Fir13}.
\end{Remark}

\subsection{Outline of the proof} \label{sec:outline} We first quickly recall Greenlees' method in his proof of the surjectivity of $L_A\to \pi_*^A(MU_A)$ for finite $A$: He proceeds by induction on the order of $A$ via isotropy separation, using that both $L_A$ and $\pi_*^A(MU_A)$ simplify after inverting or completing at certain collections of Euler classes; the precise description involves local homology. The reason his method does not lead to a full proof of the isomorphism is that it is a priori unclear that $L_A$ can be recovered from the various localizations and completions considered. In particular, elements that are infinitely Euler-divisible and Euler-torsion are taken to $0$ under all these constructions. For topological reasons, one knows that such elements do not exist in $\pi_*^A(MU_A)$. This non-existence is difficult to see directly for $L_A$, for which no simple presentation is known. Note that this stands in contrast to the non-equivariant Lazard ring $L$, which is a polynomial algebra over~$\Z$ by Lazard's theorem \cite{Laz55}.

Similar to Greenlees' work, our strategy is to reduce Theorems \ref{thm:A} and \ref{thm:B} to the already known statements that the maps $L_A\to \pi_*^A(MU_A)$ and $L_A^{2-\tor}\to \pi_*^A(MO_A)$ are isomorphisms after inverting all non-trivial Euler classes. In order to be able to make this reduction, we derive the regularity of Euler classes in Lazard rings (at tori in the complex case, and elementary abelian $2$-groups in the real case) essentially formally via global homotopy theory.
The key definition is that of a global group law given above, which is inspired by the cofiber sequence $\T_+\to S^0\to S^{\tau}$ from topology. It is based on the observation that all the $A$-equivariant formal group laws associated to a complex oriented global spectrum $E$ are encoded in the single Euler class $e\in \pi_{-2}^{\T}(E)$ and the global functoriality of $\upi_*E$, using only that the sequences
\[ 0\to \pi_*^A(E)\xr{V^*(e)\cdot } \pi_*^A(E) \xr{\res^A_{\ker(V)}} \pi_*^{\ker(V)}(E)\to 0 \]
are exact, for all split characters $V\in A^*$. This construction carries over to algebra, showing that every global group law $X$ has an associated $A$-equivariant formal group law $X^{\wedge}_A$ over $X(A)$, defined as a certain completion of the ring $X(A\times \T)$. The $A$-equivariant formal group laws that arise this way are special, since they inherit some of the structure of~$X$. In particular, all Euler classes $e_V$ for split characters $V$ of $A$ are regular elements. Since controlling the torsion of Euler classes in Lazard rings is our main concern, it is hence desirable to show that the collection $\{L_A\}_{A\text{ abelian}}$ assembles to such a global group law. Instead of trying to prove this directly, which would merely be a reformulation of the question we started with, we turn the problem around: We consider the universal global group law $\widetilde{\bL}$ and show that its evaluation on $A$ is isomorphic to $L_A$, or rather that the $A$-equivariant formal group law $\widetilde{\bL}^{\wedge}_A$ is the universal one. This is done by showing that the functor $(-)^{\wedge}_A$ has a right adjoint, which implies that it must send universal objects to universal objects, as desired. Hence, the collection of the Lazard rings $L_A$ even describes the universal global group law. In particular all split Euler classes are regular elements.

In the real bordism case, all non-trivial characters are split, so this argument allows us to prove regularity for all non-trivial Euler classes in $L_A^{2-\tor}$, and \Cref{thm:B} follows in a straightforward manner. In the complex case we need additional arguments to extend the regularity of split Euler classes to the regularity of all non-trivial Euler classes at tori. This is achieved by studying the effect of the $p$-th power map of $\T^r$ on a suitable localization of the Lazard ring  $L_{\T^r}$, see \Cref{sec:lazardregular}.

\begin{Remark} \label{rem:whyabelian} We say a few words on why we restrict to abelian rather than all compact Lie groups. The main reason is the splitting principle. It allows to reduce the theory of complex orientations to the behaviour of line bundles and their tensor product, giving the link to $1$-dimensional formal group laws. For non-abelian~$G$ there are irreducible $G$-representations of dimension larger than $1$, so that this reduction is no longer possible. We are unaware even of a definition of a $G$-equivariant formal group law for non-abelian~$G$, though some steps are taken in \cite[Part 3]{Gre01}. It seems likely that the global point of view can also shed light on this.
	
Another reason to restrict to abelian groups is that we rely on the theorems of \cite{Lof73,Com96} that $\pi_*^A(MU_A)$ is concentrated in even degrees. Whether this is also the case for non-abelian $G$ is an open problem. See \cite{Uri18} for a survey on the analog of this question for the geometric bordism rings $\Omega_*^ G$.
\end{Remark}

\textbf{Organization}: We first discuss the simpler case of real bordism and the proof of Theorems \ref{thm:B} and \ref{thm:D} in \Cref{sec:real}, after we have recalled the necessary background on equivariant formal group laws and global homotopy theory in \Cref{sec:equivariant} and \Cref{sec:globalhomotopy}. Finally, \Cref{sec:complex} discusses complex bordism and the proof of Theorems \ref{thm:A}, \ref{thm:C} and \ref{thm:E}.

\subsection{Acknowledgements} I thank Bernhard Hanke, Lennart Meier and Stefan Schwede for helpful comments on a preliminary version of this paper, and two anonymous referees for their careful reading and the many useful suggestions. This research was supported by the Hausdorff Center for Mathematics at the University of Bonn (DFG GZ 2047/1, project ID 390685813) and by the Danish National Research Foundation through the Centre for Symmetry and Deformation at the University of Copenhagen (DNRF92).


\section{Equivariant formal group laws} \label{sec:equivariant}

In this first section we recall the notion of an $A$-equivariant formal group law as defined in \cite{CGK00} and list some of its properties.

\subsection{Definition and basic properties}  \label{sec:definitions}
Let $A$ be an abelian compact Lie group. We denote by $A^*=\Hom(A,\T)$ its group of characters, with trivial character $\epsilon$. We often identify a character with its associated $1$-dimensional complex representation. Given a commutative ring $k$, we view its ring of functions $k^{A^*}$ as a topological ring via the product topology. In addition to its ring structure, $k^{A^*}$ carries the `dual group ring' coproduct
\[ \Delta\colon k^{A^*}\to k^{A^*\times A^*}\cong k^{A^*}\hotimes k^{A^*} \]
which sends a function $f\colon A^*\to k$ to the function $\Delta(f)\colon A^*\times A^*\to k$ mapping $(V,W)$ to $f(VW)\in k$. Here and throughout the paper, we write $VW$ for the product of $V$ and $W$ in the group $A^*$ (i.e., the tensor product of $V$ and $W$ when thought of as complex representations). Cole--Greenlees--Kriz \cite[Definition 11.1]{CGK00} define the following:

\begin{Def} An $A$-equivariant formal group law $F$ is a quintuple $(k,R, \Delta, \theta, y(\epsilon))$ of a commutative ring~$k$,  a commutative complete topological $k$-algebra $R$, a continuous comultiplication $\Delta\colon R\to R\hotimes R$, a map of topological $k$-algebras $\theta\colon R\to k^{A^*}$ and an orientation $y(\epsilon)\in R$, such that the following hold:
\begin{enumerate}
	\item The comultiplication is a map of $k$-algebras, cocommutative, coassociative and counital for the augmentation $\theta(\epsilon)\colon R\to k$.
	\item The map $\theta$ is compatible with the coproducts, and the topology on $R$ is generated by finite products of the kernels of the component functions $\theta(V)\colon R\to k$ for $V\in A^*$.
	\item The element $y(\epsilon)$ is regular and generates the kernel of $\theta(\epsilon)$.
\end{enumerate}
\end{Def}
\begin{DefNot}
Given an $A$-equivariant formal group law, we obtain an $A^*$-action $l$ on $R$ via the formula
\[ l_V\colon R\xr{\Delta} R\hotimes R \xr{\theta(V)\hotimes \id_R}R, \]
for $V\in A^*$.
\end{DefNot}
By the compatibility of $\Delta$ and $\theta$, we have $\theta(V)=\theta(\epsilon)\circ l_V$, hence the $A^*$-action together with the augmentation $\theta(\epsilon)$ determine the map $\theta$. The condition that $\theta$ is compatible with the comultiplication is then equivalent to demanding that $\Delta$ is $(A^*\times A^*)$-equivariant, where $A^*\times A^*$ acts through the tensor factors on $R\hotimes R$ and via the multiplication map $A^*\times A^*\to A^*$ on $R$.

\begin{DefNot} For $V\in A^*$, we write $y(V)$ for the element $l_V(y(\epsilon))$, which by the above discussion generates the ideal $I_{V^{-1}}=\ker(\theta(V^{-1}))$. Finally we define \emph{Euler classes} $e_V\in k$ via
\[ e_V=\theta(\epsilon)(y(V))=\theta(V)(y(\epsilon)). \]
\end{DefNot}
Note that $e_\epsilon=0$, since $y(\epsilon)$ lies in the kernel of $\theta(\epsilon)$ by definition. Furthermore, we have $\theta(V)(y(W))=e_{VW}$ for a pair $V,W\in A^*$.

\begin{Remark} We use a slighty different convention than \cite{CGK00} for the definition of $l_V$, $y(V)$ and $e_V$. What we call $l_V$ would be $l_{V^{-1}}$ in \cite{CGK00}, and consequently our $y(V)$ and $e_V$ correspond to their $y(V^{-1})$ and $e_{V^{-1}}$. We make this choice to simplify notation later on in the comparison with global group laws. 
\end{Remark}

\begin{Remark} \label{rem:antipode} As is shown in \cite[Appendix B]{CGK00}, every $A$-equivariant formal group law allows a coinverse, i.e., a continuous self-map $\chi\colon R\to R$ of $k$-algebras such that the composite
\[ R\xrightarrow{\Delta} R\hotimes R \xrightarrow{\id_R\hotimes \chi} R\hotimes R \to R \]
equals the composite $R\xrightarrow{\theta(\epsilon)} k\to R$. The coinverse turns $R$ into a complete topological Hopf algebra over $k$.
\end{Remark}

\subsection{Flags and topological bases} \label{sec:flags}
Given $V\in A^*$, every element $x\in R$ can be written as
\[ x=\theta(V^{-1})(x)+\widetilde{x}\cdot y(V) \]
for a unique $\widetilde{x}\in R$, since $y(V)$ is regular and generates the kernel of $\theta(V^{-1})$. It follows that for any sequence $V_1,\hdots,V_n \in A^*$, the quotient $R/(I_{V^{-1}_1}\cdot \hdots \cdot I_{V^{-1}_n})$ has a $k$-basis given by 
\[ 1,y(V_1),y(V_2)y(V_1),\hdots,y(V_{n-1})\cdots y(V_1). \]
Given a \emph{complete flag} $F$, i.e., a sequence $V_1,V_2,\hdots,V_n,\hdots$ of characters such that every character occurs infinitely often, an element $x\in R$ can be written uniquely as
\[ x=\sum_{i=0}^{\infty} a^ F_i \cdot y(V_i)\cdots y(V_1)\]
with $a^F_i\in k$. Unless $A$ is trivial, the choice of flag is not unique. The translation between the coefficients $(a^F_i)$ and $(a^{F'}_i)$ for different flags $F$ and $F'$ depends on the multiplicative structure of~$R$, which is in general not simply given by a power series ring, see \cite[Section 3]{Gre01} for a discussion. Writing $x$ in the above form, we see that
\[ \theta(V)(x)=\sum_{i=0}^{\infty} a^F_i\cdot e_{VV_i}\cdots e_{VV_1} \in k,\]
where the sum is finite since the product $e_{VV_i}\cdots e_{VV_1}$ becomes trivial once there is a $j\leq i$ with $V_j=V^{-1}$. Hence, for every flag $F$, the map $\theta(V)$ is a polynomial in the coefficients $a^F_i$ and the Euler classes.

For later use we note the following, where $F=(k,R,\Delta,\theta,y(\epsilon))$ is an $A$-equivariant formal group law. 
\begin{Lemma}\label{lem:unit} An element $x\in R$ is invertible if and only if all augmentations $\theta(V)(x)\in k$ are invertible.
\end{Lemma}
\begin{proof}
Since the augmentations are ring maps, it is clear that if $x$ is invertible, then so is $\theta(V)(x)$ for any $V\in A^*$. Conversely, assume all the $\theta(V)(x)$ are invertible. We need to see that $x$ is a unit in $R/(I_{V_n}\cdots I_{V_1})$ for all $n\in \N$ and characters $V_1,\hdots,V_n\in A^*$. To show this, we consider the kernel $K$ of the projection $R/(I_{V_{n}}\cdots I_{V_1})\to R/(I_{V_{n-1}}\cdots I_{V_1})$. By the preceding discussion, $K$ is a free module of rank $1$ over $k$, generated by the equivalence class of the element $y(V_{n-1}^{-1})\cdots y(V_1^{-1})$. By induction, we need to show that~$x$ acts isomorphically on $K$. As we saw above, we can write $x=\theta(V_n)(x)+x'\cdot y(V_n^{-1})$ for some $x'\in R$. Computing in $R/(I_{V_{n}}\cdots I_{V_1})$, we obtain
\begin{eqnarray*} x\cdot [y(V_{n-1}^{-1})\cdots y(V_1^{-1})]& = & (\theta(V_n)(x)+x'\cdot y(V_n^{-1}))\cdot [y(V_{n-1}^{-1})\cdots y(V_1^{-1})] \\
&  =  &\theta(V_n)(x)\cdot [y(V_{n-1}^{-1})\cdots y(V_1^{-1})] + [y(V_n^{-1})y(V_{n-1}^{-1})\cdots y(V_1^{-1})] \\
& = & \theta(V_n)(x)\cdot [y(V_{n-1}^{-1})\cdots y(V_1^{-1})],\end{eqnarray*}
since $[y(V_{n}^{-1})y(V_{n-1}^{-1})\cdots y(V_1^{-1})]=0$ in $R/(I_{V_{n}}\cdots I_{V_1})$. Hence, $x$ acts on $K$ via $\theta(V_n)(x)$, which is a unit in~$k$ by assumption. This finishes the proof.
\end{proof}

\subsection{The category of $A$-equivariant formal group laws} \label{sec:categoryoffgla}
We write $FGL_A$ for the category of $A$-equivariant formal group laws, where a morphism from $F=(k,R, \Delta, \theta, y(\epsilon))$ to $F'=(k',R', \Delta', \theta', y(\epsilon)')$ is a pair of maps $f_1\colon k\to k'$ and $f_2\colon R\to R'$ preserving all the structure at hand. In fact, such a morphism is already determined by $f_1$, since $f_2$ needs to send $y(V)$ to $y(V)'$ and the $y(V)$ generate $R$ as a $k$-algebra topologically.

Said differently, given a map $f\colon k\to k'$ of commutative rings and an $A$-equivariant formal group law $F$ over $k$, there is a push-forward $A$-equivariant formal group law $f_*F$ over $k'$ whose complete $k'$-algebra $f_*R$ is given by $\lim_{V_1,\hdots,V_n}\left( R/(I_{V_n}\cdots I_{V_1})\otimes_k k'\right)$. The induced map $R\to f_*R$ then extends $f$ to a morphism of $A$-equivariant formal group laws. Up to canonical isomorphism fixing the coefficient ring, every morphism of $A$-equivariant formal group laws arises in this way.

In \cite[Sections 13 and 14]{CGK00} it is shown that the category $FGL_A$ has an initial object (which will also follow from the adjunction in \Cref{sec:complexadjunction}), whose ground ring the authors call the \emph{$A$-equivariant Lazard ring} and denote by $L_A$. By the discussion above, this means that every $A$-equivariant formal group law over a commutative ring $k$ arises as the push-forward of a unique map $L_A\to k$. Therefore, the category of $A$-equivariant formal group laws is in fact equivalent to the category of commutative rings under $L_A$.

\subsection{The global Lazard ring} \label{sec:globallazard} It will be crucial for us to consider the functoriality of $A$-equivariant formal group laws as $A$ varies. For this let $\alpha\colon B\to A$ be a homomorphism of abelian compact Lie groups, and $F=(k,R, \Delta, \theta, y(\epsilon))$ a $B$-equivariant formal group law. Then we obtain an $A$-equivariant formal group law $\alpha_*F=(k,R_{\alpha},\Delta_{\alpha},\theta_{\alpha},y_{\alpha}(\epsilon))$ over the same ground ring $k$, with
\begin{itemize}
	\item $R_{\alpha}$ the completion of $R$ at all finite products of elements $y(V)$ with $V$ in the image of $\alpha^*\colon A^*\to B^*$,
	\item $\Delta_{\alpha}\colon R_{\alpha}\to R_{\alpha}\hotimes R_{\alpha}$ the completion of the comultiplication $\Delta \colon R\to R\hotimes R$,
	\item $\theta_{\alpha}$ the map induced by the composite $R\xr{\theta} k^{B^*}\xr{(\alpha^*)^*} k^{A^*}$ on completion, and
	\item $y_{\alpha}(\epsilon)\in R_{\alpha}$ the image of $y(\epsilon)$ under the completion map $R\to R_{\alpha}$.
\end{itemize}

This defines a functor $\alpha_*\colon FGL_B \to FGL_A$, which induces a map $\alpha^*\colon L_A\to L_B$ between the representing rings. It is a straightforward check that we obtain a functor
\[ \bL\colon \text{(abelian compact Lie groups)}^{op}\to \text{commutative rings}, \]
which we call the \emph{global Lazard ring}.

If $i\colon B\hookrightarrow A$ is the inclusion of a subgroup, then the induced map $i^*\colon A^*\to B^*$ is surjective. This implies that the pushforward $i_*F=(k,R_{i},\Delta_{i},\theta_{i},y_{i}(\epsilon))$  of a $B$-equivariant formal group law $F=(k,R,\Delta, \theta, y(\epsilon))$ has $R_i=R$, $\Delta_i=\Delta$ and $y_i(\epsilon)=y(\epsilon)$. The only difference is that $\theta$ is replaced by the composite $R\to k^{B^*}\to k^{A^*}$, or equivalently, that the $B^*$-action on $R$ is replaced by the $A^*$-action obtained via $i^*\colon A^*\to B^*$. This implies that $FGL_B$ is equivalent to the full subcategory of $FGL_A$ on those $A$-equivariant formal group laws where the $A^*$-action factors through $i^*$. This property can be characterized in terms of Euler classes as follows:

\begin{Lemma} \label{lem:fullsubcat} Let $(k,R, \Delta, \theta, y(\epsilon))$ be an $A$-equivariant formal group law, and $i\colon B\hookrightarrow A$ the inclusion of a subgroup. Then the action of $A^*$ on $R$ factors through $i^*$ if and only if $e_V=0$ for every $V\in \ker(i^*)\subseteq A^*$.
\end{Lemma}
\begin{proof} If the action $l$ factors through $i^*$, then any $V\in \ker(i^*)$ acts trivially. Hence $y(V)=y(\epsilon)$ and $e_V=e_{\epsilon}=0$. If for the other direction we assume that $e_V=0$ for all $V\in \ker(i^*)$, then for any such $V$ the map $\theta(V)\colon R\to k$ has $y(\epsilon)$ in the kernel. Since $\theta(V)$ is a $k$-algebra map and $R/y(\epsilon)\cong k$, this implies that $\theta(V)=\theta(\epsilon)$, and hence by definition $l_V=l_{\epsilon}=\id_R$. This shows that the action factors through $i^*$, as desired.
\end{proof}

\begin{Cor} \label{cor:lazardeulerquotient} Let $i\colon B\hookrightarrow A$ be a subgroup inclusion. Then the map $i^*\colon L_A\to L_B$ is surjective with kernel generated by the Euler classes $e_V$ for $V\in \ker(i^*)\subseteq A^*$.
\end{Cor}
In fact the proof of \Cref{lem:fullsubcat} shows that it always suffices to quotient out by finitely many Euler classes: When $V_1,\hdots,V_n$ are a set of generators of $\ker(i^*)\subseteq A^*$, then $e_{V_1},\hdots,e_{V_n}$ generate the kernel of $i^*\colon L_A\to L_B$.

As every abelian compact Lie group embeds into a torus, \Cref{cor:lazardeulerquotient} implies that the theory of equivariant formal group laws is controlled by its behaviour over tori. For example, there is a universal Euler class $e_{\tau}\in L_\T$ for the tautological $\T$-character $\tau\in \T^*$. The Euler class $e_V\in L_A$ of any other character $V\in A^*$ is obtained by pulling back $e_{\tau}$ along $V\colon A\to \T$. We will later simply write $e$ instead of $e_{\tau}$.

\subsection{The global $2$-torsion Lazard ring} \label{sec:2tor}There is also a $2$-torsion version of this theory which is relevant for the description of equivariant real bordism. For this we restrict to elementary abelian $2$-groups $A$, or in other words to finite dimensional $\F_2$-vector spaces. Note that in this case the dual group $A^*$ agrees with the dual $\F_2$-vector space.

We say that an $A$-equivariant formal group law $F=(k,R, \Delta, \theta, y(\epsilon))$ is \emph{$2$-torsion} if the composite
\[ [2]_F\colon R\xr{\Delta} R\hotimes R\xr{m} R \]
sends the orientation class $y(\varepsilon)$ (or phrased orientation-independently, the ideal $I_{\epsilon})$ to $0$. In other words, $F$ is $2$-torsion if the coinverse $\chi\colon R\to R$ (see \Cref{rem:antipode}) equals the identity map of $R$. Since $[2]_F(y(\varepsilon))$ can be written as $2\cdot y(\varepsilon)+x\cdot y(\varepsilon)^2$, the ground ring $k$ of every $2$-torsion $A$-equivariant formal group law is an $\F_2$-algebra. 

The subcategory of $2$-torsion $A$-equivariant formal group laws, which we denote by $FGL^{2-\tor}_A$, also has an initial object defined over a $2$-torsion Lazard ring $L^{2-\tor}_A$. For example, $L^{2-\tor}_A$ can be obtained from $L_A$ by quotienting out the coefficients of $[2]_F(y(\varepsilon))$ with respect to some choice of complete flag. Moreover, given a homomorphism of elementary abelian $2$-groups $\alpha\colon B\to A$, the functor $\alpha_*\colon FGL_B\to FGL_A$ restricts to a functor $\alpha_* \colon FGL^{2-\tor}_B\to FGL^{2-\tor}_A$. We obtain a global $2$-torsion Lazard ring
\[ \bL^{2-\tor}\colon (\el)^ {op}=\left(\text{elementary abelian $2$-groups}\right)^{op}\to \text{commutative $\F_2$-algebras}. \]
Again, there is a universal $2$-torsion Euler class, given by $e_{\sigma}\in L^{2-\tor}_{C_2}$ associated to the sign character $\sigma$ of $C_2$, and the analog of \Cref{cor:lazardeulerquotient} holds.

\subsection{Gradings} \label{sec:aeqgrading} It is sometimes convenient to view $\bL$ and $\bL^{2-\tor}$ as functors to $\Z$-graded rings. The relevant gradings are uniquely determined by the requirement that the Euler class $e_{\tau}\in L_\T$ is of degree $-2$, and the Euler class $e_{\sigma}\in L^{2-\tor}_{C_2}$ is of degree $-1$. The existence and uniqueness of these gradings follows either by inspection of the construction in \cite[Section 14]{CGK00}, or from the arguments in Sections \ref{sec:realuniversal} and \ref{sec:globalfgl}. The grading on $L_{C_2}$ is also discussed in \cite{HW18}, where it plays a central role in their proof of the isomorphism $L_{C_2}\xr{\cong}\pi_*^{C_2}(MU_{C_2})$.

Topologically, the graded version corresponds to considering $\upi_* (\MU)$, while the ungraded version corresponds to $\upi_0(\MUP)$, the $0$-th homotopy ring of periodic complex bordism, see \Cref{sec:complextopological}, and similarly for $\MO$ and $\MOP$.

\subsection{Inverting Euler classes} \label{sec:eulerinvert}
The goal of this subsection is to understand the structure of equivariant Lazard rings after inverting certain collections of Euler classes.
\begin{Prop} \label{prop:eulerinvert}
	Let
	\[ 1\to B\xr{i} A\xr{p} A/B\to 1\]
be a short exact sequence of abelian compact Lie groups. Let $V(B)\subset A^*$ be the subset of those characters whose restriction to $B^*$ is non-trivial, or in other words the complement of the image of the inclusion $p^*\colon (A/B)^*\hookrightarrow A^*$. Then the composite
\[ L_{A/B}\xr{p^*} L_A \to L_A[e_{V}^{-1}\ |\ V\in V(B)]\]
is flat.
\end{Prop}
\begin{proof}
The arguments are due to Greenlees (\cite[Section 11]{Gre}, \cite{Gre01}), but not explicitly stated in this form and generality, so we recall the proof. The ring $L_A[e_{V}^{-1}\ |\ V\in V(B)]$ classifies $A$-equivariant formal group laws for which all Euler classes $e_V$ with $V\in V(B)$ are invertible.	We claim that the datum of such an $A$-equivariant formal group law is equivalent to the datum of an $(A/B)^*$-equivariant formal group law $\widetilde{F}$ and an $(A/B)^*$-equivariant function
	\[	z\colon V(B)\to \widetilde{R}^{\times}\subset \widetilde{R}. \]
Here, $(A/B)^*$ acts on $V(B)\subset A^*$ via left translation along $p^*$. The identification sends $F$ to its pushforward $p_*F$ together with the function $z$ which takes $V\in V(B)$ to the image of $y(V^{-1})$ under the completion map $R\to R_p$ (note that $R_p$ is given by the completion of $R$ at finite products of the ideals $I_V$ with $V\in \im(p^*)$). The image of $z(V)$ under $\theta(V')$, for $V'\in (A/B)^*$, equals the Euler class $e_{V^{-1}V'}\in k$. Hence, as $V$ does not lie in $(A/B)^*$, neither does $V^{-1}V'$ and the Euler class $e_{V^{-1}V'}$ is a unit by assumption. Therefore. $z(V)$ is a unit in $R$ by \Cref{lem:unit}, as required.

Assuming the claim, the proposition follows: To give an $(A/B)^*$-equivariant function $z\colon V(B)\to \widetilde{R}$ is equivalent to giving one element $z(W)\in \widetilde{R}$ for a chosen representative $W$ of every non-trivial $(A/B)^*$-coset in $V(B)$. The elements $z(W)$ in turn are uniquely determined by their coefficients $a_i^F(z(W))\in k$ for a chosen flag $F$. Finally, the requirement that these are units is, again by \Cref{lem:unit}, equivalent to the requirement that the images under $\theta(V')\colon \widetilde{R}\to k$ are units for all $V'\in (A/B)^*$. As we saw in \Cref{sec:flags}, the functions $\theta(V')$ are polynomials in the coefficients $a_i^F$ and the Euler classes. In summary, the representing ring of pairs $(\widetilde{F}\in FGL_{A/B},z\colon V(B)\to \widetilde{R}^\times)$ is obtained from a polynomial ring over $L_{A/B}$ by inverting a collection of elements, and hence flat over~$L_{A/B}$.

To prove the claim, we explain how to construct an $A$-equivariant formal group law $(k,R,\Delta,\theta,y(\epsilon))$ out of data $\widetilde{F}$ and $z$ as above. We set
\[ R=\map_{(A/B)^*} (A^*,\widetilde{R}), \]
the coinduction of $\widetilde{R}$ to a ring with $A^*$-action, with the product topology. The map $\theta$ is given by the composite
\[ R=\map_{(A/B)^*} (A^*,\widetilde{R})\xr{\map_{(A/B)^*} (A^*,\widetilde{\theta})} \map_{(A/B)^*} (A^*,k^{(A/B)^*})\cong k^{A^*}. \]
The comultiplication
\[ \Delta\colon R=\map_{(A/B)^*} (A^*,\widetilde{R})\to \map_{(A/B)^*\times (A/B)^*} (A^*\times A^*,\widetilde{R}\hotimes \widetilde{R})\cong R\hotimes R \]
is the unique $(A^*\times A^*)$-equivariant map extending the $(A/B)^*\times (A/B)^*$-equivariant map
\[ R=\map_{(A/B)^*} (A^*,\widetilde{R})\xr{ev(\epsilon)} \widetilde{R} \xr{\widetilde{\Delta}} \widetilde{R}\hotimes \widetilde{R}. \]
Finally, to define the coordinate $y(\epsilon)$ we note that $R$ splits up as
\[ R\cong \widetilde{R}\times \map_{(A/B)^*} (V(B),\widetilde{R}),\]
where the first component of the isomorphism evaluates a function at $\epsilon$, and the second restricts to $V(B)$. We then define the first component of $y(\epsilon)$ to be $\widetilde{y}(\epsilon)$, and the second component to be the given map $z$. By assumption, $z$ is a unit in $\map_{(A/B)^*} (V(B),\widetilde{R})$ and hence $\theta(\epsilon)$ defines an isomorphism $R/y(\epsilon)\cong \widetilde{R}/\widetilde{y}(\epsilon)\cong k$, as desired. The Euler classes $e_V$ for $V\in V(B)$ are given by $\widetilde{\theta}(\epsilon)(z(V^{-1}))$ and hence units, since $z(V^{-1})$ is a unit.

It remains to show that the two assignments are inverse to one another. It is immediate that the pushforward to an $(A/B)$-equivariant formal group law of the construction just described gives back $\widetilde{F}$, since the completion amounts to projecting to the factor $\widetilde{R}$ of $R$. To see that it also gives back the same function~$z$, note that by equivariance the $V$-th component of $y(\epsilon)\in \map_{(A/B)^*} (A^*,\widetilde{R})$ equals the $\epsilon$-component of $y(V^{-1})$, or in other words the image of $y(V^{-1})$ in the completion, as desired. For the other direction, let $F=(k,R,\Delta,\theta,y(\epsilon))$ be an $A$-equivariant formal group law for which all Euler classes $e_V$ with $V\in V(B)$ are units, and let $p_*F$ be the pushforward to an $(A/B)$-equivariant formal group law. The completion map $R\to R_p$ is $(A/B)^*$-equivariant (if we let $(A/B)^*$ act on $R$ via $p^*\colon (A/B)^*\to A^*$), hence we obtain an induced $A^*$-equivariant map
	\[f\colon R\to \map_{(A/B)^*} (A^*,R_p).\]
It is straightforward to check that $f$ is compatible with the comultiplication and augmentation, and by the same argument as above it preserves the coordinates.
Since both sides are $A$-equivariant formal group laws and $f$ is a map of $k$-algebras, it follows that $f$ is an isomorphism (see the discussion in \Cref{sec:categoryoffgla}). This finishes the proof.
\end{proof}
\begin{Remark} An alternative proof that $R$ splits up into a product when the Euler classes $e_V$ for $V(B)$ are invertible is given by the Chinese Remainder theorem, see \cite[Theorem 11.2]{Gre}.
\end{Remark}

In the special case where $A=B$ and we are hence inverting all non-trivial Euler classes, both the choice of coset representatives and the choice of flag in the proof of \Cref{prop:eulerinvert} go away. This leads to the following description in terms of the ordinary Lazard ring $L$, again originally due to Greenlees (\cite[Theorem 6.3]{Gre01}). Given $V\in A^*-\{\epsilon\}$, we define elements $\gamma_i^V\in L_A$ via the property that
\[ y(\epsilon)=e_{V}+\sum^{n}_{i=1}\gamma_i^V\cdot y(V^{-1})^{i} \in R/y(V^{-1})^{n+1}\cong L_A\{1,y(V^{-1}),y(V^{-1})^2,\hdots,y(V^{-1})^n\} \]
for all $n\in \N$. In other words, $\gamma_i^V$ is the coefficient of $y(V^{-1})^{i}$ in the power series expansion of $y(\epsilon)$ with respect to $y(V^{-1})$. Equivalently, $\gamma_i^V$ can be defined as the coefficient of $y(\epsilon)^{i}$ in the power series expansion of $y(V)$ with respect to $y(\epsilon)$, which, in the notation of the proof of \Cref{prop:eulerinvert}, equals the $i$th coefficient of the power series $z(V^{-1})\in R_p=k[|y_p(\epsilon)|]$. Then we have:

\begin{Prop} \label{prop:geomlazard} The canonical map
 \[ L[e_V^{\pm 1},\gamma^V_i\ |\ V\in A^*-\{\epsilon\}, i\in \mathbb{N}_{>0}]\to L_A[e_V^{-1}\ |\ V\in A^*-\{\epsilon\}]\]
is an isomorphism for every abelian compact Lie group $A$.
\end{Prop} 
There is also a version for the $2$-torsion Lazard ring. Again, given an elementary abelian $2$-group $A$ and $V\in A^*$ we write $\gamma^V_i\in L_A^{2-\tor}$ for the coefficient of $y(V)^{i}$ in the power series expansion of $y(\epsilon)$ with respect to $y(V)$. The proof of the following is then entirely analogous to the one above:
\begin{Prop} \label{prop:realgeomlazard} The canonical map
	\[	L^{2-\tor}[e_V^{\pm 1},\gamma^V_i\ |\ V\in A^*-\{\epsilon\}, i\in \mathbb{N}_{>0}]\to L^{2-\tor}_A[e_V^{-1}\ |\ V\in A^*-\{\epsilon\}]\]
	is an isomorphism for every elementary abelian $2$-group $A$.
\end{Prop}

\section{Recollections on global spectra} \label{sec:globalhomotopy}

In this section we recall various features of global spectra. We work in the framework of \cite{Sch18}, though our arguments rely only on formal properties of the global stable homotopy category. For our purposes it will be enough to know that (1) every global spectrum gives rise to a global cohomology theory in the sense below, and (2) that there exist global spectra $\MU,\MUP,\MO$ and $\MOP$ whose associated $A$-equivariant spectra are equivalent to the ones introduced by tom Dieck \cite{tD70} and Br\"{o}cker-Hook \cite{BH72}.

\subsection{Global cohomology theories} Throughout this paper we take \emph{global spectrum} to mean an object of the global stable homotopy category with respect to the global family of abelian compact Lie groups, in the sense of \cite[Section 4.3]{Sch18}. Every global spectrum $E$ has an underlying genuine $A$-spectrum $E_A$ for every abelian compact Lie group $A$; hence it defines a cohomology theory $E_A^*(-)$ on based $A$-spaces. Since the spectra $E_A$ are genuine, these cohomology theories are graded over the real representation ring $RO(A)$, though there are some subleties in making this precise. For us it will be enough to consider gradings of the form $E_A^{*+V}(-)$ for $*$ an integer and $V$ an $A$-representation. Then there are natural isomorphisms
\[ E_A^{*+V}(X\wedge S^V)\cong E_A^*(X) \]
for any based $A$-space $X$.

The fact that the $E_A$ come from one global spectrum $E$ induces extra functoriality between these cohomology theories (see \cite[Construction 3.3.6]{DHLPS} for a construction in the more general setting of proper actions of arbitrary Lie groups): Any group homomorphism $\alpha\colon B\to A$ induces a natural restriction map
\[ \alpha^*\colon E^{*+V}_A(X)\to E^{*+\alpha^*(V)}_B(\alpha^*(X)),\]
where $\alpha^*(X)$ denotes $X$ with $B$-action pulled back along $\alpha$. The restriction maps are contravariantly functorial, i.e., we have $(\alpha \circ \beta)^*=\beta^*\circ \alpha^*$ if $\beta\colon C\to B$ is another group homomorphism.

\begin{Example} In particular, setting $X=S^0$ and $V=0$, the assignment
\[ \upi_*E\colon A\mapsto \pi^A_*(E_A)=E_A^{-*}(S^0) \]
defines a contravariant functor from abelian compact Lie groups to graded abelian groups, see \cite[Section 4.2]{Sch18} for a discussion. In addition to the restriction maps there are also transfer maps for subgroup inclusions, but these do not play a role in this paper.
\end{Example}

\subsection{Inductions} \label{sec:induction}
The restriction maps satisfy an additional property, called the \emph{induction isomorphism}: Given a group homomorphism $\alpha\colon B\to A$ and a based $B$-space $X$, we let $A_+\wedge_{\alpha}X$ denote the induction of $X$ to a based $A$-space along $\alpha$. The resulting functor $A_+\wedge_{\alpha}(-)$ is left adjoint to the restriction functor $\alpha^*$, with unit
\[ \eta_X\colon X\to \alpha^*(A_+\wedge_\alpha X) \]
sending $x$ to the class $[1,x]$. For a global spectrum $E$, we obtain an induction map
\[ \ind_{\alpha}\colon E_A^*(A_+\wedge_\alpha X) \xr{\alpha^*} E_B^*(\alpha^*(A_+\wedge_\alpha X)) \xr{(\eta_X)^*}E_B^*(X). \]
These satisfy the following (see \cite[Proposition 3.3.8]{DHLPS} for a proof):
\\
\\
\textbf{Induction isomorphism}: If $X$ is a based $B$-CW complex on which the kernel of $\alpha$ acts freely (in the based sense), then the induction map $\ind_\alpha$ is an isomorphism $E_A^*(A_+\wedge_\alpha X) \cong E_B^ *(X)$.
\\
\newline
The induction isomorphism can be viewed as a combination of the following two special cases:
\begin{enumerate}
	\item If $i\colon B\hookrightarrow A$ is the inclusion of a subgroup, there is an induction isomorphism
	\[ E_A^*(A_+\wedge_B X) \cong E_B^ *(X)\]
	for any based $B$-CW complex $X$. This is a formal consequence of the induction/restriction adjunction relating $A$-spectra and $B$-spectra, together with the fact that given a global spectrum $E$, its underlying $B$-spectrum $E_B$ is equivalent to the restriction of its underlying $A$-spectrum $E_A$ to $B$.
	\item If $1\to B\to A\to A/B \to 1$ is a short exact sequence of abelian compact Lie groups, and $X$ is a based $A$-CW complex on which $B$ acts freely, then there is an induction isomorphism $E_{A/B}^*(X/B)\cong E_A^*(X)$. This special case is closely related to the Adams isomorphism.
\end{enumerate}

We further record the following:
\begin{Lemma} \label{lem:restriction}
Let $i\colon B\hookrightarrow A$ be a subgroup inclusion, $X$ a based $A$-CW complex and consider the counit $\epsilon_X\colon A_+\wedge_B \res^A_BX\to X$. Then the composite
	\[  E_A^*(X)\xr{\epsilon_X^*} E_A^*(A_+\wedge_B \res^A_BX)\mathop{\xr{\ind_i}}_{\cong} E_B^*(\res^A_BX) \]
equals the restriction map.
\end{Lemma}
\begin{proof} We consider the diagram
\[   \xymatrix{ E_A^*(X) \ar[rr]^-{\epsilon_X^*}  \ar[d]_{\res^A_B} & & E_A^*(A_+\wedge_B \res^A_BX) \ar[d]^{\res^A_B} \ar[drr]_{\cong}^{\ind_{i}} \\
		E_B^*(\res^A_BX) \ar[rr]_-{(\res^A_B \epsilon_X)^*}	& &	E_B^*(\res^A_B(A_+\wedge_B \res^A_BX)) \ar[rr]_-{\eta_{\res_B^A(X)}^*} && E_B^*(\res^A_BX) }\]
where the square commutes by naturality and the triangle commutes by definition. By the triangle equality, the composite of the two lower horizontal map equals the identity of $E_B^*(\res^A_BX)$, which shows the claim.
\end{proof}
\begin{Remark} \label{rem:restriction}
Let $A,C$ be two abelian compact Lie groups, $X$ a based $A$-CW complex and $p\colon A\times C\to A$ the projection. Then $C=\ker(p)$ acts freely on $(A\times C)_+\wedge_A X$ and hence there is an induction isomorphism
\[ E_{A\times C}^*((A\times C)_+\wedge_A X) \mathop{\xleftarrow{\ind_p}}_{\cong} E_A^*(((A\times C)_+\wedge_A X)/C)\cong E_A^*(X). \]
We omit the straightforward proof that this ismorphism agrees with the induction isomorphism associated to the inclusion $A\hookrightarrow A\times C$.
\end{Remark}

\subsection{Homotopy orbits} \label{sec:homorbits}
The induction isomorphisms give rise to the following construction regarding homotopy orbits: Given two abelian compact Lie groups $A$ and $C$, let $E_AC$ be a universal $C$-space in the category of $A$-spaces. In other words, $E_AC$ is an $(A\times C)$-CW complex such that given a subgroup $H\subseteq A\times C$, the $H$-fixed points $(E_AC)^H$ are contractible if $H$ is the graph subgroup $\Gamma(\alpha)$ associated to a homomorphism $\alpha\colon B\to C$ for a subgroup $B$ of $A$, and empty otherwise. Given a based $(A\times C)$-space $X$, the quotient $A$-space $E_AC_+\wedge_C X$ is called the \emph{$(C,A)$-homotopy orbit space of $X$} and denoted $X_{h_AC}$. Now given a global spectrum $E$, we obtain a natural map 
\[ h_{A,C}(X)\colon E^{*}_{A\times C}(X)\to E^{*}_{A}(X_{h_AC}) \]
as the composite
\[ E^{*}_{A\times C}(X) \xr{(E_AC\to *)^*} E^{*}_{A\times C}((E_A C)_+\wedge X)\xleftarrow{\cong} E^{*}_{A}((E_A C)_+\wedge_C X), \]
where the second map is the inverse of the induction isomorphism. These maps have the following property: Let $\alpha \colon A\to C$ be a homomorphism. Choosing a point $*_{\alpha}$ in the contractible space $(E_AC)^{\Gamma(\alpha)}$ yields an $A$-equivariant map $(\id_A,\alpha)^*(X)\to X_{h_AC}$ sending $x$ to the class $[*_{\alpha},x]$.

\begin{Lemma} \label{lem:orbitrestriction} The composite
	\[ E^*_{A\times C}(X)\xr{h_{A,C}(X)} E^{*}_{A}(X_{h_AC})\to E^*_A((\id_A,\alpha)^*(X)) \]
equals the restriction map $(\id_A,\alpha)^*$.
\end{Lemma} 
\begin{proof} We consider the graph subgroup $\Gamma(\alpha)\subseteq A\times C$. Our chosen point $*_{\alpha}$ induces a $\Gamma(\alpha)$-equivariant map $[*_{\alpha},-]\colon X\to E_AC_+\wedge X$, which we can induce up to an $(A\times C)$-equivariant map 
\[ (A\times C)_+\wedge_{\Gamma(\alpha)} X\to E_AC_+\wedge X. \]
Now, since the intersection of $C$ with any conjugate of $\Gamma(\alpha)$ is empty, the $C$-action on $(A\times C)_+\wedge_{\Gamma(\alpha)} X$ is free. The quotient $A$-space $((A\times C)_+\wedge_{\Gamma(\alpha)} X)/C$ identifies with $(\id_A,\alpha)^*(X)$ via the homeomorphism sending $x$ to the class $[1,x]$. Hence, by naturality of the induction isomorphism we obtain a commutative diagram
	\[ \xymatrix{ E^*_{A\times C}(X) \ar[r]^-{(E_AC\to *)^*} \ar[dr]_-{(\epsilon_X)^*} & E_{A\times C}^*(E_AC_+\wedge X) \ar[d]^{[*_{\alpha},-]^*} &  E_{A}^*(E_AC_+\wedge_C X) \ar[d]^{[*_{\alpha},-]^*} \ar[l]_-{\ind }^-{\cong}\\
	& E_{A\times C}^ *((A\times C)_+\wedge_{\Gamma(\alpha)} X)  & E^*_{A}((\id_A,\alpha)^* X). \ar[l]_-{\ind }^-{\cong}
}\]
The composite of the upper two horizontal maps equals $h_{A,C}(X)$ by definition. Hence it remains to see that the lower composite equals the restriction map $(\id_A,\alpha)^*$. But identifying $A$ with $\Gamma(\alpha)$ by sending $a$ to $(a,\alpha(a))$, this follows from \Cref{lem:restriction} and \Cref{rem:restriction}.
\end{proof}

\subsection{Products} \label{sec:products} The functor that assigns to a global spectrum $E$ the genuine $A$-spectrum $E_A$ is strong symmetric monoidal. Hence, if $E$ is a commutative global homotopy ring spectrum (i.e., a commutative monoid in the homotopy category of global spectra with the smash product \cite[Corollary 4.3.26]{Sch18}), then each $E_A$ is a commutative homotopy $A$-ring spectrum. Hence there are product maps of the form
\[ E_A^{n+V}(X)\otimes E_A^{m+W}(Y)\to E_A^{n+m+V+W}(X\wedge Y)\]
and these are compatible with the restriction maps $\alpha^*$. In particular, if $E$ is a commutative global homotopy ring spectrum, then $\upi_*E$ is a contravariant functor from abelian compact Lie groups to graded commutative rings.

In the case $X=Y$, pulling back along the diagonal $X\to X\wedge X$ induces an internal multiplication on $E_A^{*+*}(X)$. Given $X$, we say that an element $t\in E_A^{n+V}(X\wedge S^W)$ is a unit if there exists an element $s\in E_A^{-n+W}(X\wedge S^V)$ such that $s\cdot t\in E_A^{W+V}(X\wedge S^{V+W})$ corresponds to $1$ under the isomorphism 
\[E_A^{W+V}(X\wedge S^{V+W})\cong E_A^{W+V}(X\wedge S^{W+V})\cong E_A^0(S^0)=\pi_0^A(E).\]
In terms of the $RO(A)$-grading, the class $t$ defines an element in $E_A^{n+V-W}(X)$, and being a unit with respect to the $RO(A)$-graded multiplication amounts to the condition above.

\section{Real bordism} \label{sec:real}

We now begin our study of equivariant bordism spectra and global group laws. We start with real bordism to demonstrate our general method while avoiding some of the complications that arise in the complex case.
\subsection{Global $2$-torsion group laws} We work over the family of elementary abelian $2$-groups and write `$\el$-algebra' to mean a functor
\[ \text{(elementary abelian $2$-groups)}^{op}\to \text{commutative rings}.\]
In this setup, we make the following definition:
\begin{Def}
	A \emph{global $2$-torsion group law} is an $\el$-algebra $X$ together with a class $e\in X(C_2)$, called the \emph{coordinate}, such that for every elementary abelian $2$-group $A$ and non-trivial $V\in A^*$ the sequence
	\begin{equation} \label{eq:sequence} 0\to X(A)\xr{e_V\cdot } X(A)\xr{\res_{\ker(V)}^A} X(\ker(V))\to 0 \end{equation}
	is exact. Here, $e_V$ denotes the pullback $V^*e\in X(A)$.
	
	A morphism of $2$-torsion global group laws is a map of $\el$-algebras $X\to Y$ sending the coordinate of $X$ to that of $Y$. We write $GL_{gl}^{2-\tor}$ for the category of $2$-torsion global group laws.
\end{Def}
We will usually leave the choice of coordinate implicit and denote the pair $(X,e)$ simply by $X$. We call the element $e_V$ the \emph{Euler class} of $V$. Note that the surjectivity in \eqref{eq:sequence} holds for any $\el$-algebra, since every inclusion of elementary abelian $2$-groups has a section. So the condition on $e$ is that any $e_V$ with $V\neq \epsilon$ is a regular element, and that divisibility by $e_V$ is detected by restricting to $\ker(V)$ (in particular that $\res^A_{\ker(V)}(e_V)=0$, which follows from the universal case $\res^{C_2}_1(e)=0$).
\begin{Remark} At this point the notion of a $2$-torsion global group law might seem a little mysterious. The definition is inspired by equivariant stable homotopy theory, where long exact sequences relating Euler classes and restrictions to subgroups are a standard tool. When combined with global structures and real orientations, these long exact sequences split up into short exact sequences of the form \eqref{eq:sequence}, as we will see shortly in \Cref{sec:globalrealoriented}.
\end{Remark}
In \Cref{sec:realadjunction1} we will show that a global $2$-torsion group law $X$ naturally encodes a $2$-torsion $A$-equivariant formal group over $X(A)$, for every elementary abelian $2$-group $A$. For now we list some elementary properties:

\begin{Lemma} \label{lem:elementary}
	\begin{enumerate}
		\item Let $X$ be a global $2$-torsion group law and $A$ an elementary abelian $2$-group. If $V_1,\hdots,V_n\in A^*$ are linearly independent, then the sequence $(e_{V_1},\hdots,e_{V_n})\in X(A)$ is regular and generates the kernel of the restriction map
		\[ \res^A_{\cap \ker(V_i)}\colon X(A)\to X(\bigcap_{i=1}^n(\ker(V_i))).\]
		In particular, given a basis $V_1,\hdots,V_r$ of $A^*$, the kernel of the restriction map to the trivial group is generated by the regular sequence $(e_{V_1},\hdots,e_{V_r})$.
		\item If $X$ is a global $2$-torsion group law, then $2=0$ in $X(1)$ and therefore $2=0$ in $X(A)$ for any $A\in \el$. Hence, $X$ takes image in $\F_2$-algebras.
	\end{enumerate} 
\end{Lemma}
\begin{proof} (1): The proof is via induction on $n$. We assume the statement for $n-1$ and write $K_{n-1}$ for the intersection $\cap_{i=1}^{n-1}(\ker(V_i))$. Then the quotient $X(A)/(e_{V_1},\hdots,e_{V_{n-1}})$ identifies with $X(K_{n-1})$ via the restriction map. Since the $V_i$ are linearly independent, the restriction of $V_n$ to a character of $K_{n-1}$ is non-trivial. Hence, by the axioms of a global $2$-torsion group law, the element $\res^A_{K_{n-1}} e_{V_n}$ is a regular element in $X(K_{n-1})$ generating the kernel of the restriction map to $X(K)$, where $K$ is the kernel of the restriction of $V_n$ to $K_{n-1}$. Since $K=K_{n-1}\cap \ker(V_n)$, the claim follows.
	
	(2): We let $e_{1,0},e_{0,1}, e_{1,1}\in X(C_2\times C_2)$ denote the Euler classes associated to the projection to the first component, projection to the second component, and the multiplication map, respectively. Then $e_{1,1}-e_{1,0}$ lies in the kernel of the restriction to the first component which is generated by $e_{0,1}$. Hence there exists a unique $x\in X(C_2\times C_2)$ such that $e_{1,1}=e_{1,0}+x\cdot e_{0,1}$. Similarly, let $x_{0,1}\in X(C_2)$ be the restriction of $x$ along the inclusion of $C_2$ in the second component. Then $x-\pr_2^*(x_{0,1})$ lies in the kernel of the restriction to the second component and there exists an $x'\in X(C_2\times C_2)$ such that $x=\pr_2^*(x_{0,1})+x'\cdot e_{1,0}$. Put together we have $e_{1,1}=e_{1,0}+\pr_2^*(x_{0,1})\cdot e_{0,1}+x'\cdot e_{1,0}\cdot e_{0,1}$. Now restricting to the second component, we find that $e=x_{0,1}\cdot e$, and hence by regularity $x_{0,1}$ must be $1$. So in summary we see that
	\[ e_{1,1}=e_{1,0}+e_{0,1}+ x'\cdot e_{1,0}\cdot e_{0,1}. \]
	Restricting along the diagonal $\diag \colon C_2\to C_2\times C_2$ this yields
	\[ 0= 2e + \diag^*(x')\cdot e^2. \]
Regularity of $e$ then implies that $0=2 + \diag^*(x')\cdot e$, which restricts to $0=2$ at the trivial group, as claimed.
\end{proof}

\begin{Remark} \label{rem:decomposition} Let $X$ be a global $2$-torsion group law. Then for every element $x\in X(C_2)$ there exist unique elements $x_0\in X(1)$ and $x'\in X(C_2)$ such that $x=p^ *(x_0)+x'\cdot e$, where $p\colon C_2\to 1$ is the projection. Indeed, the uniqueness follows from the fact that restricting to the trivial group shows that $x_0$ must be $\res^{C_2}_1(x)$ and so regularity implies that $x'$ is also unique. Conversely, the element $x-p^*(\res^{C_2}_1(x))$ lies in the kernel of $\res^{C_2}_1$, hence it is uniquely divisible by $e$, proving existence. Iterating this construction shows that for given $n\in \N$, every $x\in X(C_2)$ can be uniquely written as $x=\sum_{i=0}^{n-1} p^*(x_i)\cdot e^i + \widetilde{x}\cdot e^n$ for elements $x_i\in X(1)$ and $\widetilde{x}\in X(C_2)$. In particular, the completion of $X(C_2)$ at the augmentation ideal is isomorphic to the power series ring $X(1)[|e|]$.
	
Similarly, every $x\in X(C_2\times C_2)$ can be written uniquely as 
\[x=p^*(x_0)+\pr_1^*(x_{1,0})\cdot e_{1,0} + \pr_2^*(x_{0,1})\cdot e_{0,1} + x'\cdot e_{1,0} \cdot e_{0,1} \]
with $x_0\in X(1)$, $x_{1,0},x_{0,1}\in X(C_2)$ and $x'\in X(C_2\times C_2)$. The special case $x=e_{1,1}$ was shown in the proof of Part (2) of \Cref{lem:elementary}. Applying the polynomial expansions from the first paragraph of the remark to the elements $x_{1,0},x_{0,1}\in X(C_2)$, this implies that modulo the ideal $(e_{1,0}^m,e_{0,1}^n)$, every element $x\in X(C_2\times C_2)$ can be written uniquely as \[x=\sum_{0\leq i<m,0\leq j<n}p^*(x_{i,j})\cdot e_{1,0}^i\cdot e_{0,1}^j \]
with coefficients $x_{i,j}\in X(1)$. There are analogous decompositions for $x\in X(B)$ for all $B\in \el$. More generally, elements $x\in X(A\times B)$ can be expressed as polynomials with coefficients in $X(A)$ modulo certain products of Euler classes, for all pairs $A,B\in \el$. These constructions are studied systematically in \Cref{sec:realadjunction1} and give the link from global group laws to $A$-equivariant formal group laws.
\end{Remark}

The following are examples of global $2$-torsion group laws:

\begin{Example}[Additive] \label{ex:torsionadd} Let  $\Ga$ be the $\el$-algebra defined via
	\[ \Ga(A)=\F_2[e_V, V\in A^*]/(e_V+e_W+e_{VW}),\]
with restriction maps $\alpha^*\colon \Ga(A)\to \Ga(B)$ sending $e_V$ to $e_{\alpha^*(V)}$. Then the non-trivial element $\sigma \in (C_2)^*$ defines a coordinate $e=e_{\sigma}\in \Ga(C_2)$. This turns $\Ga$ into a global $2$-torsion group law, which we call the \emph{additive $2$-torsion formal group law}. Indeed, given a non-trivial character $V\in A^*$, we can extend it to a basis $V,V_2,\hdots,V_n$ of $A^*$. Then the ring $\Ga(A)$ identifies with the polynomial ring $\F_2[e_{V},e_{V_2}\hdots,e_{V_n}]$ and restriction to $\ker(V)$ amounts to the map
\[ \F_2[e_{V},e_{V_2}\hdots,e_{V_n}]\to \F_2[e_{V_2}\hdots,e_{V_n}]\]
sending $e_V$ to $0$. So the sequence $0\to \Ga(A)\xr{e_V\cdot} \Ga(A)\xr{\res^A_{\ker(V)}} \Ga(\ker(V))\to 0$ is short exact as required.

A more compact way to write down $\Ga$ is via the global presentation
\[ \Ga=\F_2[e]/(e_{1,0}+e_{0,1}+e_{1,1}), \]
where $e$ is a free variable in $\Ga(C_2)$ and $e_{1,0},e_{0,1}, e_{1,1}\in \Ga(C_2\times C_2)$ are as in the proof of \Cref{lem:elementary}.
\end{Example}

\begin{Example}[Ordinary $2$-torsion formal group laws] \label{ex:realcomplete} A global $2$-torsion group law is said to be \emph{complete} if each $X(A)$ is complete with respect to the kernel of the restriction map $X(A)\to X(1)$. When $X$ is complete, $X(C_2)$ is given by the power series ring $X(1)[|e|]$, see \Cref{rem:decomposition}. Moreover, given $A,B\in \el$, 	the map \[ X(A)\hotimes_{X(1)} X(B) \to X(A\times B)\]
is an isomorphism, where the tensor product is completed with respect to the induced topologies on $X(A)$ and $X(B)$ (see \Cref{lem:realadjunction} for a proof). Choosing a basis $V_1,\hdots,V_n$ of $A^*$, this shows by induction that $X(A)$ is a power series ring over $X(1)$ on the variables $e_{V_1},\hdots,e_{V_n}$.

The category of complete global $2$-torsion group laws is equivalent to the category of non-equivariant $2$-torsion formal group laws. The equivalence sends a complete global $2$-torsion group law $X$ to the tuple $(X(1),X(C_2),e)$ with comultiplication
\[ X(C_2)\xr{m^*} X(C_2\times C_2)\xleftarrow{\cong } X(C_2)\hotimes_{X(1)} X(C_2).\]
In other words, the non-equivariant $2$-torsion formal group law associated to $X$ is given by the Euler class $e_{1,1}\in X(C_2\times C_2)\cong X(1)[|e_{1,0},e_{0,1}|]$.
\end{Example}
\begin{Example}[More general complete Hopf algebras] \label{ex:realgencomplete} More examples of the previous sort can be constructed by starting with a triple $(k,R,y)$ where $R$ is a complete topological Hopf algebra over $k$ (with the topology not necessarily induced by the augmentation $R\to k$) and $y\in R$ a regular generator of the augmentation ideal. In particular, every $A$-equivariant $2$-torsion formal group law for some $A\in \el$ gives rise to a global $2$-torsion group law. More details are given in \Cref{sec:realadjunction2}.
\end{Example}

Many more examples follow from the adjunction defined in the following section, which shows that any $\el$-algebra $X$ with a chosen element $e\in X(C_2)$ can be turned into a global $2$-torsion group law in a universal way. For example, there is an initial global $2$-torsion group law (which we will later see is isomorphic to $\bL^{2-\tor}$, of which we don't yet know that it satisfies the regularity properties of a global $2$-torsion group law), a free global $2$-torsion group law on a class $x\in X(A)$, or a global $2$-torsion group law classifying the choice of two coordinates (see \Cref{sec:cooperations}).

\subsection{The universal global $2$-torsion group law} \label{sec:realuniversal}

In this section we study some formal properties of global $2$-torsion group laws. Since we make use of similar constructions throughout the paper, we recall the following standard categorical lemma, see for example \cite[Theorem 1.39]{AR94}:
\begin{Lemma} \label{lem:categorical}
	Let $\cC$ be a bicomplete category and $M=(f_i\colon c_i\to d_i)_{i\in I}$ a set of morphisms in $\cC$. Let $\cD\subseteq \cC$ be the subcategory of $M$-local objects, i.e., those objects $d$ for which $\Hom_{\cC}(f_i,d)$ is a bijection for all $i\in I$. Assume further that the sources $c_i$ are small in the sense that $\Hom_{\cC}(c_i,-)$ commutes with filtered colimits. Then the inclusion $\cD\to \cC$ has a left adjoint $l$ and $\cD$ has all colimits and limits.	
\end{Lemma}
In our situation, we take $\cC$ to be the following category of $\Fe$-algebras:
\begin{Def} An $\Fe$-algebra is an $\el$-algebra
\[ X\colon \text{(elementary abelian $2$-groups)}^{op}\to \text{commutative rings},\]
equipped with an element $e\in X(C_2)$ which restricts to $0$ at the trivial group.
\end{Def}
In other words, an $\Fe$-algebra is an $\el$-algebra $X$ together with a map 
\[ \Fe = \F_2[e]/(\res_1^{C_2}(e))\to X,\]
where $\F_2[e]$ is the free $\el$-algebra over $\F_2$ on a class $e$ in level $C_2$. In particular, the category of $\Fe$-algebras is bicomplete, with colimits and limits computed levelwise. For the set of morphisms $M$ in the lemma we take the collection
\[ f_{A,V}\colon \Fe[x]/(\res_{\ker(V)}^A(x)) \to \Fe[y], \]
indexed over all pairs of an elementary abelian $2$-group $A$ (or rather a set of representatives of isomorphism classes of such) and a non-trivial element $V\in A^*$. Here, $x$ and $y$ are free variables in level $A$, and $f_{A,V}$ sends $x$ to $e_V\cdot y$. Note that every global $2$-torsion group law is in particular an $\Fe$-algebra, since the short exact sequence
\[ 0\to X(C_2)\xrightarrow{e\cdot} X(C_2)\xrightarrow{\res^{C_2}_1} X(1) \to 0 \]
(associated to the tautological $C_2$-character $\sigma=\id_{C_2}\colon C_2\to C_2$) implies that $e$ restricts to $0$ at the trivial group. Moreover, the following holds:
\begin{Lemma}
	An $\Fe$-algebra $X$ is a global $2$-torsion group law if and only if it is local with respect to $M$.
\end{Lemma}
\begin{proof} An $\Fe$-algebra $X$ is local with respect to a fixed $f_{A,V}$ if and only if every element $x\in X(A)$ which restricts to $0$ at $X(\ker(V))$ is uniquely divisible by $e_V$. Since any $\Fe$-algebra satisfies $\res^A_{\ker(V)}(e_V)=0$, this is precisely the definition of a global $2$-torsion group law.
\end{proof}

\begin{Cor} The category of global $2$-torsion group laws is a reflective subcategory of the category of $\Fe$-algebras and closed under limits. In particular, it has all limits and colimits. We write $\wLt$ for the initial object.
\end{Cor}
One can make the construction of the left adjoint $l\colon (\Fe-\text{algebras})\to GL_{gl}^{2-\tor}$ slightly more explicit than just an existence statement: Given an $\Fe$-algebra $X$, we form a new $\Fe$-algebra $X_1$ by
\begin{enumerate}
	\item quotienting out by all $e_V$-torsion in $X(A)$, and
	\item adjoining an element $x'\in X(A)$ and the relation $x'\cdot e_V=x$ for every element $x$ in the kernel of the restriction $\res^A_{\ker(V)}$,
\end{enumerate}
both for all pairs $(A\in \el,V\in A^*)$ with $V$ non-trivial. Iterating this process leads to a sequence
\begin{equation} \label{eq:leftadjoint} X\to X_1\to X_2\to \hdots \to X_n\to \hdots. \end{equation}
The colimit of this sequence is a model for $l(X)$. Note that this shows that if each $X(A)$ is a countable ring, then so is each $l(X)(A)$. In particular the universal $A$-torsion formal group law $\wLt$ is countable in each degree, since this is the case for $\Fe$.

\subsubsection{Gradings} \label{sec:gradings}

We say that an $\Fe$-algebra (or global $2$-torsion group law) is graded if it is equipped with a grading where $e$ is of degree $-1$.

\begin{Lemma} If an $\Fe$-algebra $X$ is graded, there is a unique grading on $l(X)$ such that the unit $X\to l(X)$ respects the gradings. Via these gradings, the adjunction lifts to an adjunction between graded $\Fe$-algebras and graded global $2$-torsion group laws.
\end{Lemma}
\begin{proof} This follows from the iterated construction above. If an $\Fe$-algebra $X$ is graded, it suffices to check the conditions of a global $2$-torsion group law on the homogeneous elements. This means that we can alternatively form the sequence \eqref{eq:leftadjoint} above by performing both operations only to the homogeneous elements of $X(A)$. It then follows that each $X_i$ has a unique grading extending the one of $X_{i-1}$, giving the newly adjoined elements $x'$ satisfying $x'\cdot e_V=x$ the degree one larger than that of $x$. Consequently, the colimit inherits a unique compatible grading. The adjunction on graded objects is then a formal consequence.
\end{proof}

\begin{Cor}
	There exists a grading on $\wLt$, uniquely determined by the property that $e$ is homogeneous of degree $-1$. With this grading, $\wLt$ is initial among all graded global $2$-torsion group laws. 
\end{Cor}

\subsection{Global real oriented ring spectra} \label{sec:globalrealoriented}
In this section we turn to topology and show that the homotopy groups of a global real oriented ring spectrum equipped with the universal Euler class form a graded global $2$-torsion group law. For the study of real bordism we restrict the global category to the global family of elementary abelian $2$-groups, in the sense of \cite[Section 4.3]{Sch18}. We call the resulting spectra \emph{$\el$-global spectra}.

Let $E$ be a commutative $\el$-global homotopy ring spectrum, i.e., a commutative monoid in the $\el$-global homotopy category with respect to the smash product.

\begin{Def} A \emph{real orientation} of $E$ is an element $t\in E_{C_2}^1(S^{\sigma})=\pi_{\sigma-1}^{C_2}(E)$ which is a unit with respect to the $RO(C_2)$-graded multiplication discussed in \Cref{sec:products} and restricts to $1\in E^1(S^1)$ at the trivial group.
\end{Def}	
 Every real orientation $t$ gives rise to a unit Thom class $t_V=V^*(t)\in \pi_{V-1}^A(E)$ for any character $V\colon A\to C_2$ of an elementary abelian $2$-group $A$. Furthermore, pulling back $t$ along the inclusion of fixed points $S^0\to S^{\sigma}$ yields the universal Euler class $e\in E_{C_2}^1(S^0)= \pi_{-1}^{C_2}(E)$.
\begin{Prop} \label{prop:realgivesfgl}
	If $E$ is a real oriented $\el$-global ring spectrum, then the pair $(\upi_*E,e)$ is a graded $2$-torsion global group law.
\end{Prop}
\begin{proof} The sphere for the sign-representation $\sigma$ sits inside a $C_2$--cofiber sequence
	\[ (C_2)_+\to S^0 \to S^{\sigma}. \]
More generally, given a non-trivial character $V\colon A\to C_2$, we can pull back the $C_2$-actions via $V$, yielding a cofiber sequence of based $A$-spaces
\[ (A/\ker(V))_+\to S^0 \to S^{V}.\]
Applying the $A$-equivariant cohomology theory $E_A$ to this cofiber sequence gives a long exact sequence
\[ \hdots \to E_A^*(S^V)\to E_A^*(S^0)\to E_A^*((A/\ker(V))_+) \to \hdots .\]
The induction isomorphism from \Cref{sec:induction} implies that $E_A^*((A/\ker(V))_+)$ identifies with $E_{\ker(V)}^*(S^0)$, and the resulting map $E_A^*(S^0)\to E_{\ker(V)}^*(S^0)$ equals the restriction $\res^A_{\ker(V)}$ (\Cref{lem:restriction}). Since $E$ is a global theory and $\ker(V)\hookrightarrow A$ has a section, this restriction map also has a section and it follows that the long exact sequence collapses into a collection of (split) short exact sequences. Moreover, the real orientation of $E$ yields an identification $E_A^*(S^V)\cong E_A^{*-1}(S^0)$, and the map $E_A^{*-1}(S^0)\to E_A^*(S^0)$ becomes multiplication with the Euler class (since it is by definition the composition of the inclusion $S^0\hookrightarrow S^V$ and the Thom class). Putting everything together and replacing the cohomological notation $E_A^*(S^0)$ with the homological $\pi_*^A(E)$, we see that the sequence
\[ 0\to \pi_{*+1}^A(E)\xr{e_V\cdot} \pi_*^A(E)\xr{\res_{\ker(V)}^A} \pi_*^{\ker(V)}(E)\to 0\]
is exact, as desired.
\end{proof}
\begin{Remark} For $A=C_2$, $V=\sigma$ the sign-representation and $E=\MO$ the real equivariant bordism spectrum (see below), the above argument first appeared in \cite[Theorem 2.2]{Sin02}. Generally, the Gysin sequence relating the restriction map $\res^{C_2}_1$ to the Euler class $S^0\to S^{\sigma}$ is a standard tool in equivariant stable homotopy theory, going back at least to \cite{Bre67}.
\end{Remark}
We are particularly interested in the case where $E$ is the universal global real oriented theory $\MO$, or its periodic version $\MOP=\sum_{i\in \Z} (\MO\wedge S^i)$. These are discussed in detail in \cite[Section 6.1]{Sch18}. In particular, the real orientation of $\MOP$ is introduced in \cite[Construction 6.1.15]{Sch18}. Note that there is an isomorphism $\upi_0(\MOP)\cong \bigoplus_{n\in \Z} \upi_n(\MO)$. Hence, $\upi_0(\MOP)$ carries an ungraded global $2$-torsion group law, which we will show to be the universal one. In this sense, the difference between $\MOP$ and $\MO$ corresponds roughly to the difference between ungraded and graded group laws on the algebraic side.

\subsubsection{Relation to real oriented $A$-ring spectra} \label{sec:realAorientation}
Some words are in order on how this story relates to real orientations at a fixed group $A$. For this we recall the notion of a real orientation of an $A$-ring spectrum. Let $\U^{\R}_A$ denote a complete real $A$-universe, i.e., a sum of countable infinitely many copies of each character $V\in A^*$. Let $\R P(\U^{\R}_A)$ be its real projective space. The tautological line  bundle $\gamma$ over $\R P(\U^{\R}_A)$ is universal among all line bundles in $A$-spaces. The Thom space of $\gamma$ is again $A$-homeomorphic to $\R P(\U^{\R}_A)$. Now recall (see \cite[Definition 4.1]{CGK00} or \cite[Definition 9.5]{Gre01} for the complex case) that a real orientation of a commutative $A$-homotopy-ring spectrum $Y$ is a class $t^{(A)}\in Y^1(\R P(\U^{\R}_A))$ whose restriction to $Y^1(\R P(\epsilon \oplus V))$ for a character $V\in A^*$ is given by
\[ \begin{cases} 1 & \text{if } \ V=\epsilon \\
\text{unit} & \text{otherwise}. \end{cases} \]
Here, the condition that the restriction of $t^{(A)}$ to $Y^1(\R P(\epsilon \oplus V))$ is a unit is again meant in the $RO(A)$-graded sense, identifying $\R P(\epsilon \oplus V)$ with $S^V$ by sending $v\in V$ to $[1: v]$ and $\infty$ to $[0: 1]$. Our definition differs slightly from the one given in \cite{CGK00} in that they allow the restriction to the trivial sphere $S^1$ to be an arbitrary unit, not necessarily the element $1$.

The restriction $y(\epsilon)$ of $t^{(A)}$ to $Y^1(\R P(\U^{\R}_A)_+)$ via the $0$-section is called the \emph{universal Euler class} or \emph{coordinate} for $Y$. Given such a coordinate, the quintuple
\[ F(Y)=(Y^*(S^0),Y^*(\R P(\U^{\R}_A)_+),\Delta,\theta,y(\epsilon)) \]
is a graded $A$-equivariant formal group law, where
\[ \Delta\colon Y^*(\R P(\U^{\R}_A)_+) \to Y^*((\R P(\U^{\R}_A)\times \R P(\U^{\R}_A))_+)\cong Y^*(\R P(\U^{\R}_A)_+)\hotimes_{Y^*}Y^*(\R P(\U^{\R}_A)_+)\]
is induced from the multiplication map $\R P(\U^{\R}_A)\times \R P(\U^{\R}_A)\to \R P(\U^{\R}_A)$ classifying the tensor product of $A$-equivariant real line bundles, and
\[ \theta(V)\colon Y^*(\R P(\U^{\R}_A)_+)\to Y^*(S^0)\]
is induced from the map $S^0\to \R P(\U^{\R}_A)_+$ classifying $V$ thought of as an $A$-equivariant line bundle over a point. The proof that $F(Y)$ is indeed an $A$-equivariant formal group law is analogous to the complex version proved in \cite{CGK00}. Since the tensor product of any $A$-equivariant real line bundle with itself is trivial, it follows that $F(Y)$ is $2$-torsion.

The link from the global definition to this one is via the following observation: The unit sphere $S(\U^{\R}_A)$ is a model for the universal $C_2$-space $E_AC_2$ in $A$-spaces, where $C_2$ acts by multiplication with $-1$. Moreover, applying the unreduced $(C_2,A)$-homotopy orbit construction (see \Cref{sec:homorbits}) to the $C_2$--line bundle $\sigma \to *$ (with $A\times C_2$ acting through the projection to $C_2$) via this model yields the tautological $A$-line bundle $\gamma_A$. It follows that $(S^{\sigma})_{h_AC_2}$ is equivalent to the Thom space $\R P(\U^{\R}_A)$ of $\gamma_A$. Hence, given a real orientation $t$ of an $\el$-global ring spectrum $E$, we can apply the homotopy orbit construction $h_{A,C_2}(S^{\sigma})$ of \Cref{sec:homorbits} to obtain an element $t^{(A)}=h_{A,C_2}(S^{\sigma})(t)\in E_A^1(\R P(\U^{\R}_A))$.
\begin{Lemma} \label{lem:A-orientation}
	The element $t^{(A)}$ is a real orientation of $E_A$.
\end{Lemma}
\begin{proof} We have to show that the restriction of $t^{(A)}$ to $E_A^1(\R P(\epsilon \oplus V))$ equals $1$ when $V=\epsilon$ and an $RO(A)$-graded unit otherwise. Under the $A$-homeomorphisms $\R P(\epsilon \oplus V)\cong S^V$ and $\R P(\U^{\R}_A)\cong S(\U^{\R}_A)_+\wedge_{C_2}(S^{\sigma})$, the inclusion $\R P(\epsilon \oplus V))\to \R P(\U^{\R}_A)$ corresponds to the map $S^V=V^*(S^\sigma)\to S(\U^{\R}_A)_+\wedge_{C_2}(S^{\sigma})$ associated to a $\Gamma(V)$-fixed point $*_V=v\in S(\U^{\R}_A)$ spanning a character isomorphic to $V$. Hence, using the notation of \Cref{sec:homorbits}, the restriction of $t^{(A)}$ is given by the image of $t$ under the composition
	\[ E_{C_2}^1(S^{\sigma})\xr{\pr_{C_2}^*}E_{A\times C_2}^1(\pr_{C_2}^*S^{\sigma})\xr{h_{A,C_2}(\pr_{C_2}^*S^{\sigma})} E_A^1((S^{\sigma})_{h_AC_2})\xr{[*_V,-]^*} E_A^1(S^V). \]
By \Cref{lem:orbitrestriction}, this composition equals restriction along $V\colon A\to C_2$. The restriction map $V^*$ is compatible with the $RO(A)$-graded multiplications, hence it takes $t$ to a unit in general, and to $1$ if $V=\epsilon$ since that is required of a global orientation. This finishes the proof.
\end{proof}

Hence, for every real oriented $\el$-global ring spectrum $E$ we have a global $2$-torsion group law $(\upi_*(E),e)$, as well as an $A$-equivariant $2$-torsion formal group law $F(E_A)$ for every elementary abelian $2$-group $A$. For the rest of this subsection we explain that each $F(E_A)$ can be constructed out of $(\upi_*(E),e)$ purely algebraically.

We start by applying the argument from the proof of \Cref{lem:A-orientation} above to the base space instead of the total space of the universal line bundle. This shows that there is a commutative diagram
\begin{equation} \label{eq:comp1} \xymatrix{ \pi_*^{C_2}(E) \ar[r]^-{\pr_{C_2}^*} \ar[rd]_{V^*} & \pi_*^{A\times C_2}(E) \ar[rr]^-{h_{A,C_2}} \ar[d]^{(\id_A,V)^*} && E_A^{-*}(\R P(\U^{\R}_A)_+) \ar[d]^{\theta(V)} \\
 & \pi_*^A(E) \ar[rr]^{\cong} && E_A^{-*}      } \end{equation}
Furthermore, naturality implies that the comultiplication map on $E_A^{-*}(\R P(\U^{\R}_A)_+)$ fits into the diagram
\begin{equation} \label{eq:comp2}\xymatrix{\pi_*^{A\times C_2\times C_2}(E) \ar[rr]^-{h_{A,C_2\times C_2}} && E_A^{-*}(\R P(\U^{\R}_A)\times \R P(\U^{\R}_A)_+)\ar[r]^-{\cong} &  E_A^{-*}(\R P(\U^{\R}_A)_+)\hotimes E_A^{-*}(\R P(\U^{\R}_A)_+) \\
	\pi_*^{A\times C_2}(E) \ar[rr]^-{h_{A,C_2}} \ar[u]^{(\id_A\times m)^*} && E_A^{-*}(\R P(\U^{\R}_A)) \ar[ur]_-{\Delta} \ar[u]_{m^*} \\
	\pi_*^A(E) \ar[u]^{\pr_A^*} \ar[rr]^-{\cong} && E_A^{-*}.\ar[u] 
}\end{equation}
Now let $I_V\subseteq \pi_*^{A\times C_2}(E)$ denote the kernel of the augmentation $(\id_A,V)^*\colon \pi_*^{A\times C_2}(E)\to \pi_*^A(E)$, and denote by $(\pi_*^{A\times C_2}(E))^{\wedge}_A$ the completion at finite products of the $I_V$ with $V\in A^*$. More generally, we write $(\pi_*^{A\times B}(E))^{\wedge}_A$ for the completion at finite products of the kernels of the maps $(\id_A,\alpha)^*\colon \pi_*^{A\times B}(E)\to \pi_*^{A}(E)$ with $\alpha\colon A\to B$. Then we have the following:

\begin{Prop} \label{prop:realcompletion} Let $E$ be a real oriented $\el$-global  ring spectrum. Then the map 
	\[ h_{A,C_2}\colon \pi_*^{A\times C_2}(E)\to  E_A^{-*}(\R P(\U^{\R}_A)_+) \]
	sends an Euler class of the form $e_{(V,\id)}$ to $y(V)$ and induces an isomorphism
	\[ (\pi_*^{A\times C_2}E)_A^{\wedge}\xr{\cong} E_A^{-*}(\R P(\U^{\R}_A)_+). \]
	More generally, the homotopy orbit maps induce isomorphisms
	\[ (\pi_*^{A\times (C_2)^{\times n}}E)^{\wedge}_A \xr{\cong} E_A^{-*}(\R P(\U^{\R}_A)^{\times n}_+)\xr{\cong} E_A^{-*}(\R P(\U^{\R}_A)_+)^{\hotimes n}. \]
\end{Prop}
\begin{proof} The $0$-section $\R P(\U^{\R}_A)_+\to \R P(\U^{\R}_A)$ is obtained by applying the $(C_2,A)$-homotopy orbit construction to the $C_2$-map $S^0\to S^{\sigma}$. Hence, the universal Euler class $y(\epsilon)$ is the image of $\pr_{C_2}^*(e)=e_{(\epsilon,\id)}$. For the other characters $V$ recall that the $A^*$-action $l_V$ on $E_A^{-*}(\R P(\U^{\R}_A)_+)$ is given by the composite
	\[ E_A^{-*}(\R P(\U^{\R}_A)_+) \xr{\Delta} E^{-*}(\R P(\U^{\R}_A)_+)\hotimes E^{-*}(\R P(\U^{\R}_A)_+) \xr{\theta(V)\hotimes \id} E_A^{-*}(\R P(\U^{\R}_A)_+). \]
The commutative diagrams \eqref{eq:comp1} and \eqref{eq:comp2} hence imply that the homotopy orbit map $h_{A,C_2}$ is $A^*$-equivariant for the $A^*$-action on $\pi_*^{A\times C_2}(E)$ defined via the composite
\[ \pi_*^{A\times C_2}(E) \xr{(\id_A\times m)^*} \pi_*^{A\times C_2\times C_2}(E) \xr{((\id_A,V)\times \id_{C_2})^*} \pi_*^{A\times C_2}(E), \]
or in other words by $l_V=\begin{pmatrix} \id_A & 0\\ V & \id_{C_2}\end{pmatrix}^*\colon \pi_*^{A\times C_2}(E)\to \pi_*^{A\times C_2}(E)$. It follows that $e_{(V,\id)}=l_V(e_{(\epsilon,\id)})$ is sent to $y(V)=l_V(y(\epsilon))$.

Knowing this, the two statements about completion follow from a more general statement (\Cref{lem:realadjunction}) about completions of the form $X(A\times B)^{\wedge}_A$ for arbitrary global $2$-torsion group laws $X$, which we will prove in the next section. Indeed, \Cref{lem:realadjunction} together with the first part of the proposition guarantees that the map $\pi_*^{A\times C_2}(E)\to E_A^{-*}(\R P(\U^{\R}_A)_+)$ is an isomorphism modulo $I_{V_1}\cdots I_{V_n}$ (it sends the basis  $1,e_{(V_1,\id)},e_{(V_2,\id)}e_{(V_1,\id)},\hdots,e_{(V_{n-1},\id)}\cdots e_{(V_1,\id)}$ to the basis $1,y(V_1),y(V_2)y(V_1),\hdots, y(V_{n-1})\cdots y(V_1)$). Since we know that the target $E_A^{-*}(\R P(\U^{\R}_A)^{\times n}_+)$ is complete, the map must be a completion. Part (2) of \Cref{lem:realadjunction} then gives the statement about higher $A\times C_2^{\times n}$.
\end{proof}
\begin{Remark} Alternatively, the proposition has the following more topological proof: We can write
\[ E_A^{-*}(\R P(\U^{\R}_A)_+)\cong \lim_{\text{fin.dim.} W\subset \U^{\R}_A} E_A^{-*}(\R P(W)_+)\cong \lim_{V_1,\hdots,V_i\in A^*} E_A^{-*}(\R P(V_1\oplus\hdots \oplus V_i)_+). \]
Let $W=V_1\oplus\hdots \oplus V_i$. The projective $A$-space $\R P(W)$ is the $C_2$-orbit space of the unit sphere $S(W\otimes \sigma)$ of the $(A\times C_2)$-representation~$W\otimes \sigma$. Since the $C_2$-action on the unit sphere is free, the induction isomorphism (\Cref{sec:induction}) yields an identification
\[ E_A^{-*}(\R P(W)_+)\cong E_{A\times C_2}^{-*}(S(W\otimes \sigma)_+). \]
There is an $(A\times C_2)$-cofiber sequence
\[ S(W\otimes \sigma)_+\to S^0\to S^{W\otimes \sigma} \]
which induces a long exact sequence
\[ \hdots \to E_{A\times C_2}^{-*}(S^{W\otimes \sigma}) \to E_{A\times C_2}^{-*} \to E_{A\times C_2}^{-*}(S(W\otimes \sigma)_+)\to E_{A\times C_2}^{-*+1}(S^{W\otimes \sigma}) \to \hdots . \]
The orientation gives an isomorphism $E_{A\times C_2}^{-*}(S^{W\otimes \sigma})\cong E_{A\times C_2}^{-*-\dim W}$, under which the map $E_{A\times C_2}^{-*}(S^{W\otimes \sigma}) \to E_{A\times C_2}^{-*}$ becomes multiplication with the Euler class $e_{W\otimes \sigma}$. By the decomposition into characters $W\otimes \sigma\cong (V_1\otimes \sigma) \oplus \hdots \oplus (V_i\otimes \sigma)$, the Euler class $e_{W\otimes \sigma}$ factors as the product $e_{W\otimes \sigma}=e_{V_1\otimes \sigma}\cdots e_{V_i\otimes \sigma}$, or in our previous notation $e_{W\otimes \sigma}=e_{(V_1,\id)}\cdots e_{(V_i,\id)}$. Since all these Euler classes are regular elements, the long exact sequence above decomposes into short exact sequences, yielding an isomorphism
\[ E_A^{-*}(\R P(W)_+)\cong E_{A\times C_2}^{-*}(S(W\otimes \sigma)_+) \cong E_{A\times C_2}^{-*}/e_{W\otimes \sigma}\cong \pi^{A\times C_2}_*(E)/e_{(V_1,\id)}\cdots e_{(V_i,\id)}, \]
which upon passing to the limit over all tuples $(V_1,\hdots,V_i)$ yields the desired statement. An iteration of this argument can be used for $A\times C_2^{\times n}$.
\end{Remark}
Together with the diagrams \eqref{eq:comp1} and \eqref{eq:comp2}, \Cref{prop:realcompletion} shows that the full structure of the graded $A$-equivariant formal group $F(E_A)$ can be obtained from the global homotopy groups $\upi_*(E)$ by passing to a suitable completion. The observation that $F(Y)$ can often be described as the completion of a simpler structure if $Y$ is part of a family of equivariant spectra first appeared in \cite[Section 7]{CGK00}, predating a rigorous definition of a global spectrum.

Note moreover that given $V\in A^*$, the Euler class $e_V\in \pi_{*}^A(E)$ defined by pulling back the coordinate $e\in \pi_{*}^{C_2}(E)$ along $V$ agrees with the Euler class $e_V\in E_A^*(S^0)=\pi_*^A(E)$ associated to $F(E_A)$. Hence, the Euler classes in $\upi_*(E)$ play two different roles: Those in $\pi_*^A(E)$ correspond to the Euler classes of $F(E_A)$, while the ones of the form $e_{(V,\id)}\in \pi_*^{A\times C_2}(E)$ are taken to the coordinates $y(V)$ of $F(E_A)$ under the completion map. As we will see in the next section, this interplay of Euler classes and coordinates exists more generally for global $2$-torsion group laws and their associated $A$-equivariant formal group laws. It ultimately allows us to derive the regularity of universal Euler classes $e_V$ from the regularity of coordinates $y(V)$, which is required by the axioms of an $A$-equivariant formal group law.

\subsection{From global group laws to $A$-equivariant formal group laws} \label{sec:realadjunction1}
We now mimic these topological constructions algebraically to associate an $A$-equivariant $2$-torsion formal group law $X^{\wedge}_A$ to an arbitrary global $2$-torsion group law $X$.

Let $X$ be a global $2$-torsion group law, and $A$ an elementary abelian $2$-group. Given another $B\in \el$ and homomorphism $\alpha\colon A\to B$, we write $I_{\alpha}$ for the kernel of $(\id,\alpha)^*\colon X(A\times B)\to X(A)$. We then set $X(A\times B)^{\wedge}_A$ to be the completion of $X(A\times B)$ at finite products of the ideals~$I_{\alpha}$.

\begin{Lemma} \label{lem:realadjunction} Let $X$ be a global $2$-torsion group law and $A\in \el$.
\begin{enumerate}
\item Given $n$ characters $V_1,\hdots,V_n\in A^*$, the quotient $X(A\times C_2)/(I_{V_1}\cdot \hdots \cdot I_{V_n})$ is free of rank $n$ over $X(A)$, and a basis is given by \[ 1,e_{(V_1,\id)},e_{(V_2,\id)}e_{(V_1,\id)},\hdots,e_{(V_{n-1},\id)}\cdots e_{(V_1,\id)}. \]
\item Given $B,B'\in \el$, the map
	\[ X(A\times B)\otimes_{X(A)} X(A\times B')\xr{(\pr_{A\times B}^*,\pr_{A\times B'}^*)} X(A\times B\times B')\]
induces an isomorphism
\[ X(A\times B)^{\wedge}_A\hotimes_{X(A)}X(A\times B')^{\wedge}_A\xr{\cong} X(A\times B\times B')^{\wedge}_A.\]
\end{enumerate}
\end{Lemma}
\begin{proof} 
For Part (1), note that given $V\in A^*$ there is a short exact sequence of groups	
\[ 1\to A\xr{(\id,V)} A\times C_2\xr{(V,\id)} C_2 \to 1\]
that is split by the projection $\pr_A\colon A\times C_2\to A$. Therefore, 
\[ 0\to X(A\times C_2)\xr{e_{(V,\id)}\cdot} X(A\times C_2) \xr{(\id,V)^*}X(A) \to 0\]
is also short exact by the axioms of a global $2$-torsion group law, and split by $\pr_A^*$. It follows that every element $x\in X(A\times C_2)$ can be written uniquely as $x=\pr_A^*(x_0)+x_1\cdot e_{(V,\id)}$ with $x_0\in X(A)$. Given that $I_V$ is generated by $e_{(V,\id)}$, Part (1) follows by an easy inductive argument.

Part (2): By choosing a basis we can assume that $B'=(C_2)^{\times n}$. Via induction on $n$ we can further reduce to the case $n=1$, hence we are considering the map 
\[ X(A\times B)\otimes_{X(A)} X(A\times C_2)\xr{(\pr_{A\times B}^*,\pr_{A\times C_2}^*)} X(A\times B\times C_2).\]
Now take $(\alpha,V)\colon A\to B\times C_2$ and let $x\in I_{(\alpha,V)}\subseteq X(A\times B\times C_2)$. We consider the map $i_V\colon A\times B\to A\times B\times C_2$ whose first two components are the projections to $A$ and $B$, and whose last component equals $A\times B\to A\xr{V} C_2$. Then the kernel of $i_V^*\colon X(A\times B\times C_2)\to X(A\times B)$ is generated by the Euler class $e_{(V,\epsilon,\id)}$ associated to the map $(V,\epsilon,\id)\colon A\times B\times C_2\to C_2$. Now the element $x_0=(i_V)^*(x)\in X(A\times B)$ lies in $I_{\alpha}\subseteq X(A\times B)$, and $x-\pr_{A\times B}^*(x_0)$ is an element of the kernel of $(i_V)^*$, hence it is divisible by $e_{(V,\epsilon,\id)}$. So we see that $I_{(\alpha,V)}$ is generated by $\pr_{A\times B}^*(I_\alpha)$ and $e_{(V,\epsilon,\id)}$. Therefore we can form the completion in two steps: Completing $X(A\times B\times C_2)$ at the Euler classes $e_{(V,\epsilon,\id)}$ for $V\colon A\to C_2$ yields $X(A\times B)\hotimes_{X(A)} X(A\times C_2)^{\wedge}_{A}$, and further completing at the ideal generated by $\pr_{A\times B}^*(I_\alpha)$ gives $X(A\times B)^{\wedge}_A\hotimes_{X(A)}X(A\times C_2)^\wedge_A$, as desired.
\end{proof}
We obtain an $A$-equivariant $2$-torsion formal group law $X^\wedge_A$ as the tuple
\[ (X(A),X(A\times C_2)_A^{\wedge},\Delta,\theta,y(\epsilon)),\]
where
\begin{itemize}
	\item $\Delta$ is the induced map on completion by the composite
	\[ X(A\times C_2)\xr{(\id,m)^*} X(A\times C_2 \times C_2)\to X(A\times C_2\times C_2)_A^{\wedge}\cong X(A\times C_2)_A^{\wedge}\hotimes_{X(A)} X(A\times C_2)_A^{\wedge}.\]
	\item $\theta$ is the induced map on completion by the map $X(A\times C_2)\to X(A)^{A^*}$ whose $V$-th component is the map $(\id,V)^*\colon X(A\times C_2)\to X(A)$.
	\item $y(\epsilon)$ is the image of $e_{(\epsilon,\id)}$ under the completion map $X(A\times C_2)\to X(A\times C_2)_A^{\wedge}$.
\end{itemize}
Checking that this is indeed an $A$-equivariant $2$-torsion formal group law is straightforward: Coassociativity, cocommutativity and counitality of $\Delta$ follow directly from the corresponding properties of $m\colon C_2\times C_2\to C_2$, and compatibility of $\Delta$ and $\theta$ follows from the commutative diagram
\[ \xymatrix{ X(A\times C_2) \ar[rr]^{(\id,VW)^*} \ar[d]_-{(\id_A,m)^*}&& X(A) \\
			X(A\times C_2\times C_2). \ar[urr]_-{(\id_A,V,W)^*}
}\]
The topology on $X(A\times C_2)_A^{\wedge}$ is by definition described by the kernels of $\theta$, and regularity of $y(\epsilon)$ is an immediate consequence of Part (1) of \Cref{lem:realadjunction}, since the coefficients of $y(\epsilon)\cdot x$ with respect to a flag $F$ starting with $\epsilon$ are given by the shift of the coefficients of $x$ with respect to the flag $F'$ obtained by removing $\epsilon$ at the beginning of $F$.

Finally, the discussion in \Cref{sec:realAorientation} shows that for an $\el$-global real oriented ring spectrum $E$ with associated global $2$-torsion group law $\upi_*E$, the $A$-equivariant formal group law $(\upi_*E)^{\wedge}_A$ is isomorphic to the one associated to the induced real orientation of $E_A$.

\subsection{From $A$-equivariant $2$-torsion formal group laws to global $2$-torsion group laws} \label{sec:realadjunction2}
Having seen how to assign an $A$-equivariant formal group law to a global group law, we now study how to go the other way, as indicated in \Cref{ex:realgencomplete}. Throughout this section we fix an $A$-equivariant $2$-torsion formal group law $F=(k,R,\Delta,\theta,y(\epsilon))$, to which we want to associate a global $2$-torsion group law $F_{gl}$. The idea is simple: We want to set $F_{gl}(1)=k$, $F_{gl}(C_2)=R$, $F_{gl}(C_2\times C_2)=R\hotimes R$ etc., and use the comultiplication $\Delta$ to define the functoriality in the groups.

To make this precise, we use the following coordinate-free description. Let $\cC_{k}$ denote the category of complete commutative linear topological $k$-algebras, i.e., commutative topological $k$-algebras $S$ which have a basis $\{U_i\}_{i\in I}$ of open neighborhoods of $0$ consisting of ideals $U_i$ such that the induced map
\[ S\to \lim_{i\in I} S/U_i \]
is an isomorphism (it can be shown that this condition is independent of the chosen basis of ideals). Given two objects $S,S'\in \cC_k$, we denote by $S\hotimes S'$ their completed tensor product over $k$. Choosing ideal neighborhood bases $\{U_i\}_{i\in I}$ and $\{V_j\}_{j\in J}$ for $S$ and $S'$ respectively, $S\hotimes S'$ can be computed as
\[ S\hotimes S'= \lim_{i\in I,j\in J}(S/U_i\otimes S'/V_j)\]
equipped with the limit topology (with each $S/U_i\otimes S'/V_j$ discrete). The completed tensor product $-\hotimes -$ is the coproduct in~$\cC_k$, making the complete Hopf-algebra $R$ a commutative cogroup object. Hence, given an object $S\in \cC_{k}$, we obtain a natural abelian group structure on the set $R(S)$ of morphisms $R\to S$ via the formula
\[ f+g=(R\xrightarrow{\Delta} R\hat{\otimes}R \xr{f \hat{\otimes} g} S\hat{\otimes} S\xrightarrow{m} S),\]
with neutral element
\[ R\xr{\theta(\epsilon)} k\to S.\]
Since comultiplication by $2$ is trivial on $R$, the abelian groups $R(S)$ are in fact $\F_2$-vector spaces.

Now let $B$ be an elementary abelian $2$-group (or in other words a finite dimensional $\F_2$-vector space). Then we have the following:
\begin{Lemma}
The functor 
	\[ \Hom(B^*,R(-))\colon \cC_k\to \V \]
	is representable.
\end{Lemma}
\begin{proof} Choosing a basis of $B$, we find that there is a natural isomorphism 
	\[ \Hom(B^*,R(S))\cong R(S)^{\times n} \]
	for some $n$, which implies that $R^{\hat{\otimes}n}$ is a representing object.
\end{proof}

We let $F_{gl}(B)$ denote a choice of representing object. Since the assignment $\Hom(B^*,R(S))$ is also covariantly functorial in $B$, it follows that $F_{gl}(-)$ extends canonically to a contravariant functor from elementary abelian $2$-groups to commutative $k$-algebras, in particular to an $\el$-algebra. Note that for any $B,B'\in \el$ the canonical map 
\[ F_{gl}(B)\hotimes F_{gl}(B')\to F_{gl}(B\times B') \]
is an isomorphism, since $\Hom((B\times B')^*,R(-))$ is naturally isomorphic to $\Hom(B^*,R(-))\times \Hom((B')^*,R(-))$.

\begin{Lemma} \label{lem:realorientationregular} The element 
$y(\epsilon)\in R\cong F_{gl}(C_2)$ is a coordinate for $F_{gl}$.
\end{Lemma}
\begin{proof} Let $V\in B^*$ be a non-trivial character. Extending $V$ to a basis $V_2,\hdots,V_n$ of $B^*$, we see that $V^*(y(\epsilon))$ identifies with $y(\epsilon)\otimes 1_{R^{\hotimes (n-1)}}$ under the induced isomorphism $F_{gl}(B)\cong R^{\hotimes n}$, and restriction to $\ker(V)$ equals the map
	\[ \theta(\epsilon) \hotimes \id_{R^{\hotimes (n-1)}}\colon R^{\hotimes n}\to R^{\hotimes (n-1)}. \]
	By assumption, the sequence $0\to R\xr{y(\epsilon)\cdot} R\xr{\theta(\epsilon)} k\to 0$ is split exact, hence it remains so after tensoring with~$R^{\hotimes (n-1)}$, which finishes the proof.
\end{proof}
\begin{Remark} Note that this construction works more generally for any $2$-torsion commutative cocommutative complete linear topological Hopf-algebra $R$ over $k$ together with a regular element generating the augmentation ideal.
\end{Remark}

\begin{Remark} Like $(-)^{\wedge}_A$, the functor $(-)_{gl}$ also has a topological analog: The functor that assigns to a global spectrum an $A$-equivariant spectrum has a lax symmetric monoidal right adjoint $r_A$ (see \cite[Theorem 4.5.24]{Sch18}). A real orientation of an $A$-ring spectrum $Y$ induces a global real orientation on $r_AY$, and one can show that there is an isomorphism 
\[ \upi_*(r_AY)\cong F(Y)_{gl},\]
where $F(Y)$ is the $A$-equivariant formal group law associated to the real orientation of $Y$.	
\end{Remark}

\subsubsection{Extended functoriality} \label{sec:extendedreal}
There is one piece of structure of $F$ that we have not made use of yet: The augmentation map $\theta$ (or rather the components $\theta(V)$ with $V\neq \epsilon$). This additional structure can be interpreted as additional functoriality of $F_{gl}$, in the following way: Sending $V\in A^*$ to the map $R\xr{\theta(V)}k$ defines a group homomorphism $A^*\to R(k)$. Since $k$ is initial, this means that $R(S)$ is in fact naturally an $\F_2$-vector space under $A^*$, and we can rewrite the defining property of $F_{gl}(B)$ as
\[ \map(F_{gl}(B),S)\cong \Hom(B^*,R(S))\cong \Hom_{A^*/}(A^*\times B^*,R(S)),\]
the set of group homomorphisms from $A^*\times B^*$ to $R(S)$ under $A^*$. The upshot is that there is an induced map $F_{gl}(B)\to F_{gl}(C)$ for any map $A\times C\to A\times B$ that lies over the projections to $A$, or in other words for any map $C\times A\to B$, not only for maps $C\to B$. Using this extended functoriality, the full structure of $F$ can be reconstructed from $F_{gl}$: $k$ is given by $F_{gl}(1)$, $R$ by $F_{gl}(C_2)$, the comultiplication identifies with $(\id_A\times m)^*\colon F_{gl}(C_2)\to F_{gl}(C_2\times C_2)$, the maps $\theta(V)$ correspond to $(\id_A,V)^*\colon F_{gl}(C_2)\to F_{gl}(1)$ and the class $y(\epsilon)$ is just the coordinate of $F_{gl}$.

Note that for a global $2$-torsion group law $X$ with associated $A$-equivariant formal group law $X^{\wedge}_A$, \Cref{lem:realadjunction} shows that there is a natural isomorphism
\[ (X^{\wedge}_A)_{gl}(B)\cong X(A\times B)^{\wedge}_A. \]
Moreover, the additional functoriality of $(X^{\wedge}_A)_{gl}$ in maps over $A$ agrees with the apparent one on $X(A\times -)$ induced on completion.

\subsection{The adjunction} \label{sec:realadjunction3}

We now put the two constructions of the previous sections together and show:
\begin{Prop} \label{prop:realadjunction} The functors $(-)^{\wedge}_A$ and $(-)_{gl}$ form an adjunction 
\[(-)^{\wedge}_{A}\colon GL^{2-\tor}_{gl}\rightleftarrows FGL_{A}^{2-\tor}\colon (-)_{gl}. \]
The unit $X\to (X^{\wedge}_A)_{gl}$ is given in level $B$ by the composite
\[ X(B)\xr{\pr_B^*}X(A\times B) \to X(A\times B)^{\wedge}_A\cong (X^{\wedge}_A)_{gl}(B), \]
and the counit $(F_{gl})^{\wedge}_A\to F$ is the pair of maps $F_{gl}(A)\to F_{gl}(1)\cong k$ and $F_{gl}(A\times C_2)^{\wedge}_A\to F_{gl}(C_2)\cong R$ induced by the diagonals $\diag\colon A\to A\times A$ and $\diag\times \id_{C_2}\colon A\times C_2\to A\times A\times C_2$ respectively, using the extended functoriality of $F_{gl}$.
\end{Prop}
\begin{proof}
The adjunction is most easily described as a composite of two adjunctions of the following form:
 \[ \xymatrix{ GL^{2-\tor}_{gl}\ar@/^/[rr]^{(-)_A} \ar@/^2.2pc/[rrrr]^{(-)_A^{\wedge}} && GL^{2-\tor}_{gl/A} \ar@/^/[rr]^{(-)^{\wedge}} \ar@/^/[ll]^{u} && FGL^{2-\tor}_A \ar@/^/[ll]^{(-)_{gl/A}} \ar@/^2.2pc/[llll]_{(-)_{gl}}} \]
 
Here, the intermediate category $GL^{2-\tor}_{gl/A}$, whose objects one might call `$A$-global $2$-torsion group laws', is the category of global $2$-torsion group laws $Y$ with extended functoriality as described at the end of the previous section. This means that there are restriction maps $Y(B)\to Y(C)$ for any map $A\times C\to A\times B$ over $A$. The categories $GL^{2-\tor}_{gl/A}$ and $GL^{2-\tor}_{gl}$ are related by an adjunction, where the right adjoint $u\colon GL^{2-\tor}_{gl/A}\to GL^{2-\tor}_{gl}$ forgets the extra functoriality, and the left adjoint $(-)_A\colon GL^{2-\tor}_{gl}\to GL^{2-\tor}_{gl/A}$ is defined via $X_A(B)=X(A\times B)$. The unit $X\to u(X_A)$ is then induced by the projections $A\times B\to A$, while the counit $u(Y)_A\to Y$ is restriction along the maps $\diag\times B\colon A\times B\to A\times A\times B$ via the extended functoriality of $Y$.

As we just saw, the functor $(-)_{gl}$ naturally factors through $u$, yielding $(-)_{gl/A}\colon FGL^{2-\tor}_A\to GL^{2-\tor}_{gl/A}$. The remaining functor $(-)^{\wedge}\colon GL^{2-\tor}_{gl/A}\to FGL^{2-\tor}_A$ is the evident generalization of our definition of $(-)^{\wedge}_A$ to an arbitrary object of $GL^{2-\tor}_{gl/A}$. It sends $Y\in GL^{2-\tor}_{gl/A}$ to the quintuple
\[ (Y(1),Y(C_2)^{\wedge},\Delta,\theta,y(\epsilon)),\]
where the completion on $Y(C_2)$ is formed with respect to all kernels of augmentation ideals $(\id_A,V)^*\colon Y(C_2)\to Y(1)$ (with $(\id_A,V)\colon A\to A\times C_2$, using the extended functoriality of $Y$), and $\Delta$ is induced by $(\id_A\times m)^*$ on completions. To see that this is indeed left adjoint to $(-)_{gl/A}$, we note that \Cref{lem:realadjunction} also applies in this more general setting, showing that there is a natural isomorphism
\[ (Y^{\wedge})_{gl/A}(B)\cong Y(B)^{\wedge}_A \]
for any $B\in \el$. We can then define the unit of the adjunction $Y\to (Y^{\wedge})_{gl/A}$ to be the completion map. For the other direction, note that $F_{gl/A}(B)$ is already complete for any $B$. In particular there are natural isomorphisms $F_{gl/A}(1)\cong k$ and $F_{gl/A}(C_2)^{\wedge} \cong R$ and we find that $F^{\wedge}_{gl/A}$ is naturally isomorphic to $F$ itself. We let the counit be this natural isomorphism. The composite unit and counit then work out as stated, finishing the proof of  \Cref{prop:realadjunction}.
\end{proof}

\begin{Remark} Note that this shows that $(-)_{gl/A}$ identifies the category of $A$-equivariant $2$-torsion formal group laws with the full subcategory of $A$-global $2$-torsion group laws consisting of the complete objects. When $A$ is the trivial group, this embeds ordinary $2$-torsion formal group laws into global $2$-torsion group laws, as described in \Cref{ex:realcomplete}.
\end{Remark}

Now consider the universal global $2$-torsion group law $\wLt$, the initial object of $GL_{gl}^{2-\tor}$. Since any left adjoint preserves initial objects, it follows that the associated $A$-equivariant $2$-torsion formal group law $(\wLt)^{\wedge}_A$ is also universal. In particular, $\wLt(A)$ is canonically isomorphic to the equivariant Lazard ring $L^{2-\tor}_A$, and for any global $2$-torsion group law $X$ the map
\[ \wLt(A)\to X(A),\]
given by evaluating the unique map $\wLt\to X$ on $A$, classifies the $A$-equivariant formal group law $X^{\wedge}_A$. In fact, we have:
\begin{Prop} The adjunctions of \Cref{prop:realadjunction} induce an isomorphism $\bL^{2-\tor}\cong \wLt$ of $\el$-algebras with coordinate $e\in \wLt$ corresponding to the universal Euler class $e_{\sigma}\in \bL^{2-\tor}$. In particular,~$\bL^{2-\tor}$ equipped with the universal Euler class $e_{\sigma}$ is a global $2$-torsion group law and hence for every elementary abelian $2$-group $A$ all Euler classes of non-trivial characters are regular elements in $L^{2-\tor}_A$.
\end{Prop}
\begin{proof}
The relation between the coordinate and the universal Euler class is clear by construction: In the $C_2$-equivariant formal group law $X^{\wedge}_{C_2}$ associated to any $X\in GL^{2-\tor}_{gl}$, the Euler class $e_{\sigma}\in X(C_2)$ is given as the pullback of $e_{(\epsilon,\id)}\in X(C_2\times C_2)$ along the diagonal $C_2\to C_2\times C_2$, which gives back the coordinate $e\in X(C_2)$. To show that the isomorphisms $(\wLt)(A)\cong L^{2-\tor}_A$ are compatible with restriction maps, let $\alpha\colon B\to A$ be a homomorphism. For every $X\in GL^{2-\tor}_{gl}$ we obtain a map of $A$-equivariant formal group laws $(X)^{\wedge}_A\to \alpha_*(X)^{\wedge}_B$ (where $\alpha_*$ denotes the push-forward along $\alpha$ as in \Cref{sec:globallazard}) via
\[ X(A)\xr{\alpha^*} X(B) \]
and
\[  X(A\times C_2)^{\wedge}_A\xr{(\alpha\times \id_{C_2})^*}  \alpha_*X(B\times C_2)^{\wedge}_{B}.  \]
For $X=\wLt$, this must give the unique map of $A$-equivariant formal group laws $(\wLt)^{\wedge}_A\to \alpha_* (\wLt)^{\wedge}_B$. In particular, the underlying map $L^{2-\tor}_A\cong \wLt(A)\to \wLt(B)\cong L^{2-\tor}_B$ classifies the push-forward along $\alpha$ of the universal $B$-equivariant $2$-torsion formal group law. This equals by definition the functoriality in $\bL^{2-\tor}$, so we find that the levelwise isomorphisms between $\bL^{2-\tor}$ and $\wLt$ are indeed compatible with the structure maps.
\end{proof}

\begin{Remark} We stress that the regularity of Euler classes in Lazard rings is a priori very unclear. The Euler classes of arbitrary $A$-equivariant formal group laws are typically not regular, so this is a special feature of the universal ones. Historically, getting control of the Euler torsion in Lazard rings has been the main obstacle to proving Theorems \ref{thm:A} and \ref{thm:B}, see \Cref{sec:outline} in the introduction. It is a major benefit of working globally that the regularity follows essentially formally.
\end{Remark}

\begin{Cor}
All $2$-torsion Lazard rings $L_A^{2-\tor}$ are integral domains.
\end{Cor}
\begin{proof} The non-equivariant $2$-torsion Lazard ring $L^{2-\tor}$ is a polynomial algebra over $\F_2$, hence an integral domain. By \Cref{prop:geomlazard}, the ring $L_A^{2-\tor}[e_V^{-1}\ |\ V\in A^*-\{\epsilon\}]$ is obtained from $L^{2-\tor}$ by adjoining polynomial and Laurent ring generators, in particular it is again a domain. But since we just saw that the Euler classes are regular elements of $L_A^{2-\tor}$, this means that $L_A^{2-\tor}$ injects into this localization and hence must be a domain itself.
\end{proof}

\subsection{Geometric fixed points and proof of the main theorem} \label{sec:realmain}

Before we turn to the proof of Theorems \ref{thm:B} and \ref{thm:D}, we first generalize the construction of inverting Euler classes to arbitrary global $2$-torsion group laws.

\begin{Def}[Geometric fixed points] Let $X$ be a global $2$-torsion group law, and $A$ an elementary abelian $2$-group. Then we define
	\[ \Phi^A(X)=X(A)[e_V^{-1}\ |\ V\in A^*-\{\epsilon\}] \]
and call it the \emph{$A$-geometric fixed points of $X$}.
\end{Def}
Since any other coordinate only differs from $e$ by a unit in $X(C_2)$, the definition of $\Phi^A(X)$ is independent of the choice of coordinate. The collection of the $\Phi^A(X)$ for varying $A$ does not assemble to an $\el$-algebra, since the pullback along an arbitrary homomorphism $\alpha\colon B\to A$ can send a non-trivial $V\in A^*$ to the trivial class in $B^*$. This does not happen if $\alpha$ is surjective, so the $\Phi^A(X)$ still form a contravariant functor from the category of elementary abelian $2$-groups and surjective group homomorphisms to commutative rings.

The name `geometric fixed points' is motivated from topology: Given a genuine $A$-spectrum $Y$, the $A$-geometric fixed points $\Phi^A(Y)$ are the non-equivariant spectrum defined as the genuine fixed points $(Y\wedge \widetilde{E}P_A)^A$, where $\widetilde{E}P_A$ is a based $A$-CW complex uniquely characterized by the fact that its $A$-fixed points are equivalent to $S^0$ and the fixed points at all other subgroups are contractible. If $E$ is a global spectrum, $\Phi^A(E)$ is defined as $\Phi^A(E_A)$.

\begin{Lemma} For a real oriented global ring spectrum $E$ the geometric fixed point maps $\pi^A_*(E)\to \pi_*(\Phi^A(E))$ induce a natural isomorphism
\[ \Phi^A(\upi_*(E))\cong \pi_*(\Phi^A(E))=\Phi^A_*(E). \]
\end{Lemma}
\begin{proof} This follows from the fact that the sphere
	\[ \hocolim_{V\subseteq \U_A; V^A=0} S^V \]
is a model for $\widetilde{E}P_A$, where the colimit is taken along all finite dimensional subrepresentations $V\subseteq \U_A$ with trivial fixed points. Hence there is an equivalence
\[ (E_A\wedge \widetilde{E}P_A)^A\simeq \hocolim_{V\subseteq \U_A; V^A=0} (E_A\wedge S^V)^A\simeq (E_A\wedge S^{\dim(V)})^A,\]
where the last equivalence comes from the real orientation of $E_A$. Decomposing each $V$ into a direct sum of irreducible representations shows that $\Phi^A_*(E_A)$ is obtained from $\pi_*((E_A)^A)=\pi_*^A(E_A)$ via a filtered colimit of iterated multiplications by Euler classes, each appearing infinitely often, which proves the claim.
\end{proof}

Now we come to the proof of Theorems \ref{thm:B} and \ref{thm:D}. By \Cref{prop:realgivesfgl} applied to the real orientation of $\MO$, the coefficients $\upi_*(\MO)$ form a global $2$-torsion group law. Hence there exists a unique map
\[ \bL^{2-\tor}\to \upi_*(\MO)\]
of global $2$-torsion group laws, and we want to show that this map is an isomorphism. With the structural understanding of $\bL^{2-\tor}$ and $\upi_*(\MO)$ we gathered in the previous sections, the only input from topology we require is a comparison of their geometric fixed points.

\begin{Prop} \label{prop:realgeomiso} The map $\bL^{2-\tor}\to \upi_*(\MO)$ induces an isomorphism on geometric fixed points $\Phi^A$ for every elementary abelian $2$-group $A$.
\end{Prop}
\begin{proof} By \Cref{prop:realgeomlazard} we understand the geometric fixed points of the Lazard ring $L^{2-\tor}_A$. The computation of the geometric fixed points of $MO_A$ goes back to tom Dieck (see \cite{tD70} for the complex analog) and is also described in \cite[Proposition 3.2]{Fir13}. From these descriptions it follows that $\Phi^A(\bL^{2-\tor})$ and $\Phi_*^A (\MO)$ are abstractly isomorphic, assuming Quillen's non-equivariant isomorphism $L^{2-\tor}\cong MO_*$ \cite{Qui69}. The proof that the map $\Phi^A(\bL^{2-\tor})\to \Phi_*^A (\MO)$ realizes such an isomorphism can then be given analogously to the one in the complex case due to Greenlees (see \cite[Proposition 13.2]{Gre01} and \cite[Section 11]{Gre}).
	
For the convenience of the reader, we sketch the argument. Let $Y$ be a real oriented $A$-ring spectrum. The geometric fixed point functor $\Phi^A$ defines a bijection
\begin{equation}  \label{eq:split} [\RPU_+,Y\wedge \widetilde{E}P_A]_*^A\cong [\RPU^A_+, \Phi^A(Y)]_*\cong \prod_{A^*} [\R P^\infty_+,\Phi^A(Y)]_*,\end{equation}
using that $\RPU^A$ decomposes into a disjoint union of copies of $\R P^{\infty}$ indexed over all $V\in A^*$ (see \cite[Section 11.A]{Gre01}, for example). The ring $[\R P^\infty_+,\Phi^A(Y)]_*=\Phi^A(Y)^*(\R P^\infty_+)$ is the associated non-equivariant formal group law over $\Phi_*^A(Y)$, and the splitting \eqref{eq:split} corresponds to the splitting of Euler-invertible formal group laws described in the proof of \Cref{prop:eulerinvert}.

Now we consider the case $Y=MO_A$. There is an equivalence
\begin{equation} \label{eq:equiv}  \Phi^A(MO_A)\simeq S[\widetilde{A}^*]\wedge MO\wedge (\prod_{V\in \widetilde{A}^*}BO)_+,\end{equation}
(see \cite[Proposition 3.2]{Fir13}, or \cite[Theorem 4.9]{Sin01} for the complex analog) where $\widetilde{A}^*$ denotes the set of non-trivial characters and $S[\widetilde{A}^*]$ is the homotopy ring spectrum $\bigvee_{W\in \Z[V]_{V\in \widetilde{A}^*}} S^{|W|}$. In the latter, $\Z[V]_{V\in \widetilde{A}^*}$ denotes the free abelian group on $\widetilde{A}^*$, and $|W|$ denotes the sum of the coefficients of $W$. The Euler class $e_V\in \Phi^A_*(MO_A)$ corresponds to the inclusion $S^{-1}\to S[\widetilde{A}^*]$ associated to $-V$.

Now we recall from non-equivariant homotopy theory that there is an isomorphism $MO_*(BO_+)= MO_*[b_1,b_2,\hdots]$, where $b_i$ is the image of the generator in $MO_i(\R P^{\infty}_+)$ dual to the $i$-th power of the standard coordinate $x$ of $MO$. Recall moreover that in terms of the pushforward of $x\in MO^1(\R P^{\infty}_+)$ under $MO\to MO\wedge BO_+$, the map $\R P^{\infty}_+\to BO_+\to MO\wedge BO_+$ is given by the power series $1+b_1x+b_2x^2 + b_3x^3 + \hdots $. For $V\in \widetilde{A}^*$, we write $b_i^V\in $ for the image of $b_i$ under the inclusion $MO\wedge BO_+\xr{MO\wedge i_V} \Phi^A(MO_A)$ associated to $V$. Then we find that there is an isomorphism
\[ \Phi^A_*(MO_A)= MO_*[b_i^V,e_V^{\pm 1}\ |\ V\in \widetilde{A}^*,i\in \N_{>0}].\]
In \Cref{prop:eulerinvert} we saw that $\Phi^A(\bL^{2-\tor})$ is isomorphic to $L^{2-\tor}[\gamma^V_i,e_V^{\pm 1}\ |\ V\in \widetilde{A}^*, i\in \N_{>0}]$, where the $\gamma^V_i$ are determined by the fact that under the Euler-invertible splitting, the $V$th component of the coordinate $y(\epsilon)$ equals $e_V+\gamma_1^V x+\gamma_2^V  x^2 + \hdots $, where $x$ is the coordinate of the associated non-equivariant formal group law.

Hence, we need to consider the orientation class in the $A$-equivariant formal group law over $\Phi^A_*(MO_A)$. As described above, this $A$-equivariant formal group law is given by
\[ \prod_{A^*} [\R P^\infty_+,\Phi^A(MO_A)]_*\cong \prod_{A^*} [\R P^\infty_+,S[\widetilde{A}^*]\wedge MO\wedge (\prod_{V\in \widetilde{A}^*}BO)_+]_*. \]
By taking fixed points, the orientation $(\RPU_+)\wedge S^{-1}\to MO_A\wedge \widetilde{E}P_A$ then corresponds to a map
\[ (\bigsqcup_{A^*}\R P^{\infty})_+\wedge S^{-1}\to S[\widetilde{A}^*]\wedge MO\wedge (\prod_{V\in \widetilde{A}^*}BO)_+. \]
Inspection of the equivalence \eqref{eq:equiv} shows that this map is given by the wedge of the maps
\[ \lambda_V\colon \R P^{\infty}_+\wedge S^{-1}\to BO_+\wedge S^{-1} \xr{i_V\wedge e_V} \Phi^A(MO_A)\]
for non-trivial characters $V$, and the usual orientation
\[ t\colon \R P^{\infty}_+\wedge S^{-1}\to MO \to \Phi^A(MO_A) \]
for the trivial character.

By what we recalled above, the map $\lambda_V$ can be expressed in terms of the coordinate $x$ as
\[ \lambda_V = e_V(1+b_1^V x + b_2^V x^2 + \hdots)= e_V+e_Vb_1^Vx+ e_Vb^V_2 x^2+\hdots. \]
Hence, $\gamma^V_i$ is sent to the element $e_Vb^V_{i}$, a unit multiple of $b^V_{i}$. Since we know that the map $L^{2-\tor}\to MO_*$ is an isomorphism by \cite{Qui69}, it follows that so is $\Phi^A(\bL^{2-\tor})\to \Phi^A_*(MO_A)$, as desired.
\end{proof}

\begin{Remark} \label{rem:hankeclasses}
In terms of the classes described in the proof, the classes $Y_{d,V}\in \Phi^A_*(MO_A)$ that feature in the pullback square of \cite{Fir13} relating geometric and homotopical bordism are defined via 
\[ Y_{d,V}=e_V^{-1} b_{d-1}^V, \]
for $d\geq 2$ (see \cite[Definition 3.5]{Fir13}), or in other words the image of the element $e_V^{-2}\gamma^V_{d-1}\in \Phi^A(\bL^{2-\tor})$. This is the real analog of Hanke's definition in the complex version \cite{Han05}.
\end{Remark}
Theorems \ref{thm:B} and \ref{thm:D} are then a direct consequence of the following `Whitehead theorem for global $2$-torsion group laws'. 

\begin{Prop} \label{prop:whitehead} Let $f:X\to Y$ be a morphism of global $2$-torsion group laws such that $\Phi^A(f)$ is an isomorphism for every elementary abelian $2$-group $A$. Then $f$ is an isomorphism.
\end{Prop}
\begin{proof} We fix $A\in \el$ and show that the map $f(A)\colon X(A)\to Y(A)$ is an isomorphism. Since all non-trivial Euler classes are regular in $X$ and $Y$, it follows that $X(A)$ and $Y(A)$ embed into $\Phi^A(X)$ and $\Phi^A(Y)$ respectively. By assumption, $\Phi^A(f)$ is an isomorphism, thus $f(A)$ is injective. Furthermore, surjectivity of $\Phi^A(f)$ implies that for every $y\in Y(A)$ there exist non-trivial $V_1,\hdots,V_n\in A^*$ and $x\in X(A)$ such that $f(x)=e_{V_1}\cdot \hdots e_{V_n}\cdot y$. If $n=0$ and hence $f(x)=y$, we are done. If $n\geq 1$,  note that 
\[ f(\res^A_{\ker(V_n)}(x))=\res^A_{\ker(V_n)}(f(x))=0, \]
since $f(x)$ is a multiple of $e_{V_n}$. But $f(\ker(V_n))$ is also injective, thus $\res^A_{\ker(V_n)}(x)=0$. As $X$ is a global $2$-torsion group law, this implies that $x$ can be written as $x=e_{V_n}\cdot x'$. It follows that
	\[ e_{V_n}\cdot f(x')=f(e_{V_n}\cdot x')=f(x)= e_{V_1}\cdot \hdots \cdot e_{V_n}\cdot y,\]
which means that
	\[ f(x')=e_{V_1}\cdot \hdots \cdot e_{V_{n-1}}\cdot y\]
by regularity of $e_{V_n}$. Repeating this process for $V_{n-1},\hdots,V_1$ produces the desired preimage of $y$, which finishes the proof of the proposition and Theorems \ref{thm:B} and \ref{thm:D}.
\end{proof}

\begin{Remark} One can show that more generally, $\upi_*(\MO^{\wedge n})$ is universal among $\el$-algebras equipped with a so-called strict $n$-tuple of coordinates. We carry this out in the complex case in \Cref{sec:cooperations}.
\end{Remark}

\section{Complex bordism} \label{sec:complex}

We now move on to global complex bordism $\MU$ and its relation to equivariant formal group laws via Theorems \ref{thm:A}, \ref{thm:C} and \ref{thm:E}. The general strategy is similar to the one in \Cref{sec:real}, but it is complicated by the facts that (1) not every abelian compact Lie group is a torus and especially (2) not every non-trivial character $V\colon \T^r\to \T$ is split. Proceeding analogously to the $2$-torsion case in \Cref{sec:real} will show that the Euler classes $e_V\in L_{\T^r}$ for split characters $V\colon \T^r\to \T$ are regular elements. This is not sufficient to reduce to geometric fixed points since these are formed by inverting the Euler classes for all non-trivial characters. The step from split surjective characters to all non-trivial ones requires an additional argument, which is given in \Cref{sec:lazardregular} and relies on studying the effect of the $p$-th power map $\T^r\xr{(-)^p} \T^r$ on Lazard rings.

In the following we leave out details where they are analogous to those in \Cref{sec:real} and focus on the differences.

\subsection{Global group laws} \label{sec:globalfgl}
We begin with the definition of a global group law. For many of the constructions it is technically convenient to first restrict to the family of tori and later extend formally to all abelian compact Lie groups, in the way explained in \Cref{sec:fromtoritoall} below. See the discussion in \Cref{sec:complextopological} for the reasons behind this.

\begin{Def} A \emph{global group law} is a functor
	\[ X\colon \text{(tori)}^{op}\to \text{commutative rings}\]
together with a class $e\in X(\T)$, called a \emph{coordinate}, such that for every torus $A$ and every split surjective character~$V\colon A \to \T$ the sequence
\begin{equation} \label{eq:defglo} 0\to X(A)\xr{e_V\cdot} X(A)\xr{\res^{A}_{\ker(V)}} X(\ker(V))\to 0\end{equation}
is exact, where we write $e_V$ for the pullback $V^*e$. A morphism of global group laws $(X,e)\to (Y,e')$ is a natural transformation $X\to Y$ sending $e$ to~$e'$. We denote the category of global group laws by $GL_{gl}$.
\end{Def}
\begin{Remark} Since the opposite of the category of tori is equivalent to the category of finitely generated free abelian groups via Pontryagin duality, the more algebraically minded reader might prefer to phrase global group laws as functors
\[ \text{finitely generated free abelian groups} \to \text{commutative rings}.\]
This point of view suggests a relationship between global group laws and abelian cogroup objects in commutative rings, i.e., commutative cocommutative Hopf algebras. We say more about this in \Cref{ex:hopf} and \Cref{sec:complexadjunction}.
\end{Remark}

Again, we usually simply write $X$ for the tuple $(X,e)$. We  also make use of the weaker notion of a \emph{$\Ze$-algebra}, which is just a contravariant functor $X$ from tori to commutative rings equipped with an element $e\in X(\T)$ which restricts to $0$ at the trivial group. Every $\Ze$-algebra can be formally turned into a global group law:

\begin{Prop} \label{prop:reflexiveglobal}	The category of global group laws is a reflective subcategory of the category of $\Ze$-algebras which is closed under limits. In particular, it has all limits and colimits. 
\end{Prop}
\begin{proof} The surjectivity of the restriction maps in \eqref{eq:defglo} is implied by the existence of a section of $V$, and hence the global group laws are exactly the $\Ze$-algebras that are local with respect to the morphisms
\[ f_{A,V}\colon \Ze[x]/(\res_{\ker(V)}^A(x)) \to \Ze[y] \]
sending $x$ to $y\cdot e_V$. Here, $(A,V)$ ranges through all pairs of a torus $A$ and a split character $V\in A^*$, and $x$ and $y$ are formal variables at the group $A$. The proposition then follows from \Cref{lem:categorical}. 
\end{proof}

In particular there exists an initial global group law $\wL$ (which we will later prove to be isomorphic to the global Lazard ring $\bL$ defined in \Cref{sec:globallazard}).

\begin{Def} A \emph{graded} global group law is a contravariant functor from tori to graded commutative rings equipped with a coordinate $e$ of homogeneous degree $-2$ such that the sequences \eqref{eq:defglo} are exact.
\end{Def}
We obtain the following, cf. \Cref{sec:gradings}.
\begin{Cor} \label{cor:complexgrading}
	There exists a unique grading on $\wL$, which is concentrated in even degrees. With this grading, $\wL$ is initial among all graded formal group laws. 
\end{Cor}

\subsubsection{From tori to all abelian compact Lie groups} \label{sec:leftkan} \label{sec:fromtoritoall} It is useful to extend global group laws to functors defined on all abelian compact Lie groups (i.e., to an $Ab$-algebra) via left Kan extension. This process has a more explicit  description in terms of quotienting by Euler classes, which we now describe. We first need a preparatory lemma.

\begin{Lemma} \label{lem:euleragree}
	Let $T_1,T_2$ be tori and $f,g:T_1\to T_2$ two group homomorphisms which agree on a closed subgroup $B$ of $T_1$. Then for every global group law $X$, the maps $f^*$ and $g^*\colon X(T_2)\to X(T_1)$ agree modulo the ideal generated by all Euler classes $e_V$ for those $V\in T_1^*$ which vanish on $B$.
\end{Lemma}
\begin{proof}
Let $g'=f^{-1}\cdot g$ be the pointwise product of $f^{-1}$ and $g$. Then $g'$ vanishes on $B$, and $g$ factors as
\[ T_1\xr{(f,g')} T_2\times T_2\xr{m} T_2.\]
Given $x\in X(T_2)$, the difference $m^*(x)-\pr_1^*(x)\in X(T_2\times T_2)$ lies in the kernel of the restriction to the first $T_2$-factor. Hence we can write $m^*(x)$ as 
\[ m^*(x)=\pr_1^*(x) + x',\]
with $x'$ in the ideal generated by all $e_V$ for characters $V\in \im(\pr_2^*\colon T_2^*\to (T_2\times T_2)^*)$ (cf. \Cref{lem:elementary} for the real analog). So we have $g^*(x)=f^*(x)+(f,g')^*(x')$, and $(f,g')^*(x')$ lies in the ideal generated by Euler classes $e_V$ with $V\in \im((g')^*\colon T_2^*\to T_1^*)$. Since any $V$ in $\im((g')^*)$ becomes trivial when restricted to $B$ (because $g'$ vanishes on $B$), this yields the desired statement.
\end{proof}

We can now extend a global group law $X$ as follows: For any abelian compact Lie group $A$ we choose an embedding $i_A\colon A\hookrightarrow T_A$ into a torus. For simplicity we choose $i_A$ to be the identity if $A$ already is a torus. Then we set
\[ X(A)=X(T_A)/(e_V\ |\ V\in T_A^*, i_A^*V=\epsilon).\]
Note that it is enough to quotient out by  $e_V$ for $V$ ranging through a set of generators of $\ker(i_A^*)$. For a group homomorphism $\alpha\colon B\to A$, we choose a dotted arrow in the diagram
\[  \xymatrix{ T_B \ar@{-->}[r]^-{\widetilde{\alpha}} & T_A \\
				B \ar[u]^{i_B} \ar[r]^-{\alpha}& A \ar[u]_{i_A}},  \]
which exists since $T_A$ is an injective object in the category of abelian compact Lie groups. We then set $\alpha^*\colon X(A)\to X(B)$ to be the unique map making the diagram
\[ \xymatrix{ X(T_B) \ar[d] & X(T_A) \ar[l]_-{\widetilde{\alpha}^*}\ar[d] \\
				X(B) & X(A) \ar[l]^-{\alpha^*}} \]
commute. This is well-defined: If $V\in (T_A)^*$ restricts to the trivial element in $A^*$, then, by commutativity of the square, $\widetilde{\alpha}^*(V)$ restricts to the trivial element in $B^*$, and hence the relevant Euler classes are indeed sent to $0$ in $X(B)$. Furthermore, the definition is independent of the choice of $\widetilde{\alpha}$ by \Cref{lem:euleragree}. It then follows easily that this construction preserves identities and composition, that its restriction to the full subcategory of tori agrees with the original $X$, and that $i_A^*\colon X(T_A)\to X(A)$ is given by the projection map. Finally, this extension is a left Kan extension by the following two observations:
\begin{enumerate}
	\item By definition, every class of $X(A)$ is the image of a class in $X(T)$ for some torus $T$ and group homomorphism $A\to T$. In fact, the chosen embedding $i_A$ does the job.
	\item Let $X'$ be any extension of $X$ to a functor defined on all abelian compact Lie groups. Then, if $V\in T_A^*$ restricts to $\epsilon$ in $A^*$, its Euler class $e_V$ must restrict to $0$ in $X'(A)$ since the composite
	\[     A\to T_A\xr{V} \T\]
	factors through the trivial group, where the unique Euler class is trivial.
\end{enumerate}

Since left Kan extensions are unique up to canonical isomorphism, this also implies that the description of $X(A)$ is independent of the chosen embedding $i_A$. From now on we will not distinguish between a global group law defined only on tori and its extension to all abelian compact Lie groups obtained in this way.

\subsection{Regular global group laws}  \label{sec:regularfgl} Given a global group law, all Euler classes $e_V$ for split characters $V$ are regular elements by definition, while no condition is required for the other Euler classes.

\begin{Lemma} \label{lem:k-regular} Given a global group law $X$ and number $k\in \N$, the following two conditions are equivalent:

\begin{enumerate}
\item For every torus $T$, natural number $l\leq k$ and linearly independent $l$-tuple $(V_1,\hdots,V_l)$ of characters of $T$, the sequence
\[ (e_{V_1},\hdots,e_{V_l})\]
is regular in $X(T)$.
\item For every abelian compact Lie group $A$ with $\pi_0(A)$ a product of at most $(k-1)$ cyclic groups, and every surjective character $V\colon A\to \T$, the sequence
\begin{equation*} \label{eq:exact} 0\to X(A)\xr{e_V\cdot} X(A)\xr{\res^A_{\ker(V)}} X(\ker(V))\to 0\end{equation*}
is exact.
\end{enumerate}

\end{Lemma}
\begin{Def} If the two equivalent conditions are satisfied, we say that $X$ is a \emph{$k$-regular global group law}. If $X$ is $k$-regular for all $k$, we say that $X$ is \emph{regular}.
\end{Def}

\begin{proof}[Proof of \Cref{lem:k-regular}] (1) $\Rightarrow$ (2): Let $(A,V)$ be as in (2) above. We choose an embedding $i_A\colon A\hookrightarrow T_A$ and a basis $V_1,\hdots,V_l$ of $\ker(i_A^*)$. We can arrange that $l\leq k-1$ since $\pi_0(A)$ is a product of at most $(k-1)$ cyclic groups by assumption. Furthermore, we lift $V$ to a character $\widetilde{V}$ of $T_A$. Note that $X(A)$ identifies with $X(T_A)/(e_{V_1},\hdots,e_{V_l})$, and $X(\ker(V))$ identifies with $X(T_A)/(e_{V_1},\hdots,e_{V_l},e_{\widetilde{V}})$. It follows that the map $X(A)\to X(\ker(V))$ is surjective with kernel generated by $e_V$. Moreover, the surjectivity of $V$ implies that $(V_1,\hdots,V_k,\widetilde{V})$ is a linearly independent tuple, and hence $e_V=i_A^*(e_{\widetilde{V}})$ is a regular element of
\[X(T_A)/(e_{V_1},\hdots,e_{V_l})\cong X(A), \]
so the statement follows.

(2) $\Rightarrow$ (1): By induction on $l\leq k$ it suffices to show that $e_{V_l}$ is a regular element in 
\[ X(T)/(e_{V_1},\hdots,e_{V_{l-1}}),\]
which identifies with $X(A)$ for $A=\ker(V_1)\cap \hdots \cap \ker(V_{l-1})$. Since the sequence is regular, the restriction of $V_l$ to $A$ is a surjective character, so by assumption its Euler class must be regular, which finishes the proof.
\end{proof}

Hence, if $X$ is a regular global group law, then the sequence
\[ 0\to X(A)\xr{e_V\cdot} X(A)\xr{\res^A_{\ker(V)}} X(\ker(V))\to 0\]
is exact for every abelian compact Lie group $A$ and surjective character $V$. We will later see that the initial global group law $\wL$ is regular, but this is unclear at this point.

\subsection{Examples} \label{sec:examples}
We list some examples of global group laws.

\begin{Example}[Additive global group law] Let $k$ be a commutative ring. Then the \emph{additive global group law} $(\Ga)_k$ over $k$ is defined as
	\[ (\Ga)_k(A)=k[e_V, V\in A^*]/(e_{VW}-e_V-e_W\ |\ V,W\in A^*), \]
with coordinate $e=e_{\tau}\in (\Ga)_k(\T)$ for the tautological $\T$-character $\tau$. Choosing a basis $W_1,\hdots,W_r$ of $(\T^r)^*$ gives an isomorphism
\[ (\Ga)_k(\T^r)\cong k[e_{W_1},\hdots,e_{W_r}]. \]
It follows analogously to \Cref{ex:torsionadd} that $(\Ga)_k$ is indeed a global group law. Its regularity depends on the ground ring $k$: Given an $l$-tuple of linearly independent characters $V_1,\hdots,V_l\in (\T^r)^*$, we can choose a basis $W_1,\hdots,W_r$ of $(\T^r)^*$ such that each $V_i$ is a non-trivial multiple of $W_i$ modulo the subgroup generated by $W_1,\hdots,W_{i-1}$ (and hence $e_{V_i}$ is the same multiple of $e_{W_i}$ modulo $(e_{W_1},\hdots,e_{W_{i-1}})$). We then have an isomorphism $(\Ga)_k(\T^r)\cong k[e_{W_1},\hdots,e_{W_r}]$. It follows that if $k=\Q$, the sequence $e_{V_1},\hdots,e_{V_l}$ is regular in $(\Ga)_k(\T^r)$, since by induction each $e_{V_i}$ is a unit multiple of $e_{W_i}$ modulo $(e_{V_1},\hdots,e_{V_{i-1}})=(e_{W_1},\hdots,e_{W_{i-1}})$. So $(\Ga)_{\Q}$ is a regular global group law. If $k=\F_p$, then $e_{V_i}=0$ if $p$ divides $V_i$, so $(\Ga)_{\F_p}$ is not even $1$-regular. When $k=\Z$, then each single Euler class $e_{V}$ with $V\in (\T^r)^*-\{\epsilon\}$ is regular, but for example any sequence $(e_{V_1},e_{V_2})$ with some prime $p$ dividing both $V_1$ and $V_2$ is not regular as
\[ e_{V_2}\cdot e_{V_1/p}=p \cdot e_{V_2/p}\cdot e_{V_1/p}=e_{V_2/p}\cdot e_{V_1} \]
is trivial modulo $e_{V_1}$, while $e_{V_1/p}$ is not. Therefore, $(\Ga)_{\Z}$ is $1$-regular but not $2$-regular.
\end{Example}
\begin{Example}[Multiplicative global group law] \label{ex:mult} The \emph{multiplicative global group law} over a commutative ring $k$ is defined via
	\[ (\Gm)_k(A)=kA^*\cong RU(A)\otimes_{\Z}k, \]
the group ring of $A^*$ over $k$, or in other words the complex representation ring $RU(A)$ tensored with~$k$. The coordinate $e$ is the difference $[\tau] - 1\in RU(\T)\otimes k$. A basis $W_1,\hdots,W_r$ of $(\T^r)^*$ gives an isomorphism to the Laurent polynomial ring
\[ (\Gm)_k(\T^r)\cong k[[W_1]^{\pm 1},\hdots, [W_r]^{\pm 1}]. \]
Given a linearly independent tuple $(V_1,\hdots,V_l)$ in $(\T^r)^*$, we can again choose the basis in such a way that each $V_i$ is a non-trivial multiple $n_iW_i$ of $W_i$ modulo the subgroup generated by $W_1,\hdots,W_{i-1}$. Then the quotient $(\Gm)_k(\T^r)/(e_{V_1},\hdots,e_{V_{i-1}})$ takes the form 
\[ (k[[W_1]^{\pm 1},\hdots,[W_{i-1}]^{\pm 1}]/(e_{V_1},\hdots,e_{V_{i-1}}))[[W_i]^{\pm 1},\hdots, [W_r]^{\pm 1}], \]
with $e_{V_i}$ being of the form $x+ [W_i]^{n_i}$ for some $x\in k[[W_1]^{\pm 1},\hdots,[W_{i-1}]^{\pm 1}]/(e_{V_1},\hdots,e_{V_{i-1}})$. Since any element of this form is regular, we conclude that $(\Gm)_k$ is a regular global group law for all ground rings $k$.

The multiplicative global group law over $k=\mathbb{Z}$ is realized as $\upi_0$ of the periodic global $K$-theory spectrum~$\mathbf{KU}$, see \cite{Joa04} and \cite[Section 6.4]{Sch18}.
\end{Example}
\begin{Example}[Hopf algebras/1-dimensional linear algebraic groups] \label{ex:hopf}
	The category of global group laws for which all maps
	\begin{equation} \label{eq:tensoriso} (\pr_A^*\otimes \pr_B^*)\colon X(A)\otimes_{X(1)} X(B)\to X(A\times B) \end{equation}
	are isomorphisms is equivalent to the category of triples $(k,R,y(\epsilon))$ of a commutative ring $k$, a cocommutative commutative Hopf algebra $R$ over $k$ and a regular element $e\in R$ generating the augmentation ideal. The correspondence sends a global group law $X$ with the above property to the triple $(X(1),X(\T),e)$, where the comultiplication on $X(\T)$ is given by
	\[ X(\T)\xr{m^*} X(\T\times \T)\xleftarrow{\cong} X(\T)\otimes_{X(1)}X(\T). \]
	The previous two examples are special cases of this construction, for the Hopf algebras given by the polynomial ring $k[x]$ (the additive group) and the Laurent ring $k[x^{\pm 1}]$ (the multiplicative group). A further specific example is the circle $S^1=\R[x,y]/(x^2+y^2-1)$ over $\R$, with standard comultiplication
	\[ \Delta\colon \R[x,y]/(x^2+y^2-1) \to  \R[x_1,y_1]/(x_1^2+y_1^2-1)\otimes_\R \R[x_2,y_2]/(x_2^2+y_2^2-1) \]
	sending $[x]$ to $[x_1\otimes  y_1 - x_2 \otimes y_2]$ and $[y]$ to $[x_1\otimes y_2 + x_2\otimes y_1]$. A coordinate is given by $[x]-1$. 
\end{Example}
\begin{Example}[Complete global group laws] \label{ex:complete}
	We say that a global group law is \emph{complete} if each $X(\T^r)$ is complete with respect to the augmentation ideal $\ker(X(\T^r)\to X(1))$. Then $X(\T)$ is isomorphic to the power series ring over $X(1)$ on the coordinate $e$, and each $X(\T^r)$ becomes a power series ring on $r$ generators. The subcategory of complete global group laws is equivalent to the category of ordinary formal group laws, see \Cref{sec:complexadjunction}. Topologically, the complete global group laws correspond to the global Borel theories  of ordinary complex oriented ring spectra (see \cite[Example 4.5.19]{Sch18}).
	
	When $X$ is complete as above, then $X(C_n)$ is isomorphic to $X(1)[|e|]/([n]_F(e))$, where $[n]_F(e)$ denotes the $n$-series of the associated formal group law $F$. This ring classifies the $n$-torsion in $F$. If $F$ is a divisible formal group law over a complete local ring $X(1)$ whose residue field is of characteristic $p$ (`divisible' meaning that the multiplication-by-$p$ map $X(1)[|e|]\xrightarrow{[p]_F}X(1)[|e|]$ is an isogeny), then the sequence
	\[ X(1) \leftarrow X(C_p) \leftarrow X(C_{p^2}) \leftarrow \hdots  \leftarrow X(C_{p^n})\leftarrow \hdots\]
	is the $p$-divisible group associated to $F$, see \cite{Tat67}.
\end{Example}
\begin{Example}[Complete Hopf algebras] \label{ex:gencomplete}
	More generally, complete topological Hopf algebras with a regular generator of the augmentation ideal give rise to global group laws, see \Cref{rem:completegeneral}. In particular, every $A$-equivariant formal group law $F$ over any abelian compact Lie group $A$ defines a global group law $F_{gl}$ (\Cref{sec:complexadjunction}).
\end{Example}
\begin{Example}[Base change] \label{ex:basechange} If $X$ is a global group law, $k$ is a commutative ring and $X(1)\to k$ is a ring map, then one obtains a new global group law $X_k$ by tensoring each degree $X(A)$ with $k$ over $X(1)$. To see that this construction preserves the relevant short exact sequences, note that these are split and hence preserved under any additive functor.
	
Base change can be combined with the previous examples to obtain new global group laws. In particular, we can start with a complete global group law $F_{gl}$ associated to an ordinary formal group law $F$ and base change to another ground ring $k$. Then $(F_{gl})_k$ is in general \emph{not} complete. From the viewpoint of the divisible case described at the end of \Cref{ex:complete}, the base change on global group laws computes the base change of the associated $p$-divisible group, which in particular preserves the height. This shows that global group laws can encode $p$-divisible groups that are not connected, unlike ordinary formal group laws.
\end{Example}
\begin{Example}[Universal constructions] \label{ex:universal} We already saw that every $\Ze$-algebra can be universally turned into a global group law. For example this gives rise to the initial global group law $\widetilde{\bL}$, or a universal global group law equipped with a choice of $n$ coordinates (see \Cref{sec:cooperations}), or one which is universal with the property that $e_{V^{-1}}=-e_V$ for all $V$ (i.e., the universal `odd global group law', cf. \cite{BM14}).
\end{Example}

\subsection{Global complex oriented ring spectra} \label{sec:complextopological}
We now come to the relation between global group laws and complex orientated global ring spectra. Here, again, `global' is meant with respect to the family of abelian compact Lie groups. Let $E$ be a commutative global homotopy ring spectrum, i.e., a commutative monoid with respect to the smash product on the global homotopy category \cite[Corollary 4.3.26]{Sch18}. We from now on simply write `global ring spectrum' for this notion. Let $\tau$ denote the tautological $\T$-character.

\begin{Def} A \emph{complex orientation} of $E$ is an $RO(\T)$-graded unit $t\in \pi_{2-\tau}^{\T}E=E_{\T}^2(S^{\tau})$ which restricts to $1$ at the trivial group.
\end{Def}
Given a complex orientation $t$ of $E$, we obtain the universal Euler class $e\in E_{\T}^2(S^0)=\pi_{-2}^{\T}(E)$ by pulling back $t$ along the inclusion $S^0\hookrightarrow S^{\tau}$.

\begin{Prop} Let $E$ be a complex oriented global ring spectrum. Then the pair $(\upi_*(E),e)$ is a graded global group law. In fact, the sequence
\begin{equation} \label{eq:complexexact} 0\to \pi_*^A(E)\xr{e_V\cdot} \pi_*^A(E)\xr{\res^A_{\ker(V)}} \pi_*^{\ker(V)}(E)\to 0 \end{equation}
is exact for every abelian compact Lie group $A$ (not necessarily a torus) and split character $V\colon A\to \T$.
\end{Prop}
\begin{proof} We consider the cofiber sequence
	\begin{equation} \label{eq:cofibertau} \T_+ \to S^0\to S^{\tau} \end{equation}
of based $\T$-spaces. Pulling back along a surjective character $V\colon A\to \T$ yields a cofiber sequence
\[ A/\ker(V)_+ \to S^0 \to S^V \]
of based $A$-spaces, which is taken to a long exact sequence
\[ \hdots \to E_A^*(S^V) \to E_A^*(S^0)\to E_A^*(A/\ker(V)_+)\to \hdots . \]
Via the isomorphism $E_A^*(A/\ker(V)_+)\cong E_{\ker(V)}^*(S^0)$, the second map agrees with restriction	to $\ker(V)$. If $V$ is split, the fact that $E$ is a global theory implies that this restriction has a section and the long exact sequence decomposes into short exact sequences. Finally, the complex orientation gives an isomorphism $E_A^*(S^V)\cong E_A^{*-2}$, under which the first map becomes multiplication with the Euler class $e_V$, so the statement follows.
\end{proof}
\begin{Remark} The line of argument in the above proof can be found in \cite[Theorem 1.2]{Sin01}, in the universal case $E=\MU$. One difference is that Sinha makes use of the fact that $\upi_*(\MU)$ is concentrated in even degrees to deduce the surjectivity of restriction maps $\pi_*^A(MU_A)\to \pi_*^B(MU_B)$ for all inclusions $B\hookrightarrow A$ of abelian groups. This produces short exact sequences of the above form also for non-split surjective characters $V\in A^*$. For split characters $V\in A^*$ the surjectivity of $\pi_*^A(MU_A)\to \pi_*^{\ker(V)}(MU_{\ker(V)})$ follows directly from the global structure and hence generalizes to all complex oriented global ring spectra $E$, while the short exact sequences for non-split characters require additional assumptions on $E$, see \Cref{lem:complexregular} below.
\end{Remark}
\textbf{Warning:} When $E$ is a global complex oriented ring spectrum, it is in general not true that $\upi_*(E)$ is left Kan extended from tori. An example is given by the Borel theory $bH\F_p$ for ordinary cohomology with $\F_p$ coefficients (cf. \cite[Example 4.5.19]{Sch18}), since the restriction map on group cohomology $H^*(B\T,\F_p)\to H^*(BC_p,\F_p)$ is not surjective.

Therefore, the proposition should be read as `the restriction of $\upi_*(E)$ to tori is a global group law', while the generalized exact sequences \eqref{eq:complexexact} refer to the actual values of $\pi_*^A(E)$ and $\pi_*^{\ker(V)}(E)$, not the left Kan extension of its restriction to tori. To make this distinction clear, we write $\upi_*^{\toral}E$ for the left Kan extension.

In light of this, one might wonder why we have not defined a global group law to be a contravariant functor from all abelian compact Lie groups to commutative rings for which the analogs of \eqref{eq:complexexact} are exact. The reason to restrict to tori instead is of technical nature. With this more general definition one runs into problems when defining the adjunctions of \Cref{sec:complexadjunction}, for example it is unclear how to compute the completion $X(C_2)^{\wedge}$ at the augmentation ideal when the restriction map $X(\T)\to X(C_2)$ is not surjective.

Given these issues, it is useful to know when $\upi_*(E)$ does agree with the left Kan extension from tori:
\begin{Lemma} \label{lem:complexregular} Let $E$ be a complex oriented global ring spectrum. \begin{enumerate}
		\item The $Ab$-algebra $\upi_*(E)$ is left induced from tori if and only if $\upi_*^{\toral}(E)$ is regular as a global group law.
		\item A sufficient condition for the equivalent conditions of (1) to hold is that $\upi_*E$ be concentrated in even degrees.
\end{enumerate}
\end{Lemma}
\begin{proof} Let $A$ be an abelian compact Lie group and $V\colon A\to \T$ a character that is surjective but not necessarily split. Again we can pullback the cofiber sequence \eqref{eq:cofibertau} and obtain a long exact sequence
	\begin{equation}  \label{eq:longexact} \hdots \to \pi_*^A(E)\xr{e_V\cdot} \pi_*^A(E)\xr{\res^A_{\ker(V)}} \pi_*^{\ker(V)}(E)\to \hdots .  \end{equation}
Now, since $V$ does not split, it is not generally true that $\res^A_{\ker(V)}$ is surjective, so the long exact sequence does not split up into short exact ones. In fact, we see that the restriction map is surjective if and only if $e_V$ is a regular element. Note that if $\upi_*(E)$ is concentrated in even degrees, then the boundary maps are necessarily trivial, which proves Part (2).

If $\upi_*(E)$ is left induced from tori (i.e., $\upi_*^{\toral}(E)=\upi_*(E)$), then all restriction maps are surjective. Hence, all Euler classes $e_V$ for surjective characters are regular and $\upi_*^{\toral}(E)$ is regular, as desired. We now conversely assume that $\upi_*^{\toral}(E)$ is a regular global group law. We show by induction on the rank of $\pi_0(A)$ that the natural map $\upi_*^{\toral}(E)(A)\to \pi_*^A(E)$ is an isomorphism for every abelian compact Lie group~$A$. For this we choose an embedding $i_A\colon A\hookrightarrow T_A$ into a torus and a basis $V_1,\hdots,V_k$ for the kernel of $i_A^*\colon T_A^*\to A^*$. We can assume that $k$ equals the rank of $\pi_0(A)$. Let $B=\ker(V_1)\cap\hdots\cap \ker(V_{k-1})$. Then by the induction hypothesis the natural map $\upi_*^{\toral}(E)(B)\to \pi_*^B(E)$ is an isomorphism. Furthermore, the restriction of $V_k$ to $B^*$ is a surjective character and hence $e_{\res_B^{T_A}(V_k)}\in \upi_*^{\toral}(E)(B)\cong \pi_*^B(E)$ is a regular element. The long exact sequence \eqref{eq:longexact} associated to $\res_B^{T_A}(V_k)$ together with the formula for the left Kan extension from \Cref{sec:leftkan} then implies that we have an isomorphism
\[ \upi_*^{\toral}(E)(A)\xrightarrow{\cong} \upi_*^{\toral}(E)(B)/(e_{\res_B^{T_A}(V_k)}) \xrightarrow{\cong} \pi_*^B(E)/(e_{\res_B^{T_A}(V_k)})\xrightarrow{\cong} \pi_*^A(E), \]
which finishes the proof.
\end{proof}
Hence, the relationship between topology and algebra is tightest if one restricts to regular global group laws. For general complex oriented theories one either needs to somehow implement the long exact sequences \eqref{eq:longexact} in the definition of a global group law, making the theory very cumbersome to work with, or to disregard the classes that are not restrictions of toral ones. Since the global Lazard ring \emph{is} left Kan extended from tori, all the equivariant formal group law data lies in the restrictions of toral classes. So the latter simplification is quite harmless.

We will particularly be interested in global complex bordism $\MU$ and its periodic version $\MUP=\bigvee_{i\in \Z} \MU\wedge S^{2i}$, both described in  \cite[Example 6.1.53]{Sch18}.

\begin{Cor} \label{cor:muregular} The global group laws of $\MU$ and $\MUP$ are regular.
\end{Cor}
\begin{proof} The coefficients of $MU_A$ are concentrated in even degrees for every abelian compact Lie group $A$, which was annonounced by L\"{o}ffler in \cite{Lof73}. A proof by Comeza{\~n}a can be found in \cite[XXVIII Theorem 5.3]{Com96}, Hence the coefficients of $MUP_A$ are also concentrated in even degree. Therefore, Part (2) of \Cref{lem:complexregular} gives the statement.
\end{proof}

\subsubsection{Relation to complex oriented $A$-ring spectra} We discuss the relation to complex orientations for a fixed abelian compact Lie group $A$, in analogy to \Cref{sec:realAorientation} in the real case. Let $Y$ be a commutative $A$-homotopy ring spectrum (short: $A$-ring spectrum). We write $\CPU$ for a complete complex $A$-universe, i.e., a countably infinite dimensional complex $A$-representation into which any finite dimensional $A$-representation embeds. Given a character $V\in A^*$, there is an $A$-homeomorphism $\C P(\epsilon\oplus V)\cong S^V$ sending $v$ to $[1:v]$.

\begin{Def} A complex orientation of $Y$ is an element $t^{(A)}\in Y^2(\CPU)$ which restricts to $1$ in $Y^*(\C P(\epsilon \oplus \epsilon))$ and to an $RO(A)$-graded unit in $Y^*(\C P(\epsilon\oplus V))$ for all other characters $V\in A^*$.
\end{Def}
Again, we differ slightly from the definition in \cite{CGK00} in that they do not require $t^{(A)}$ to restrict to $1$ at the trivial group but merely to a unit. Pulling back $t^{(A)}$ along the $0$-section of the universal $A$-equivariant line bundle gives the universal Euler class or coordinate $y(\epsilon)\in Y^2(\CPU_+)$. Cole--Greenlees--Kriz \cite{CGK00} then show that the quintuple
\[ (Y^*,Y^*(\CPU_+),\Delta,\theta,y(\epsilon)) \]
is a graded $A$-equivariant formal group law, where $\Delta$ is induced by the multiplication map $m\colon (\CPU\times \CPU)_+ \to \CPU_+$ classifying tensor products of line bundles via
\[ Y^*(\CPU_+)\xr{m^*} Y^*((\CPU\times \CPU)_+)\cong Y^*(\CPU_+)\hotimes_{Y^*} Y^*(\CPU_+). \]
The map $\theta(V)\colon Y^*(\CPU_+)\to Y^*$ is induced by the map $S^0\to \CPU_+$ classifying $V$.

When $E$ is a complex oriented global ring spectrum, we again make use of the homotopy orbit map
\[ E_{\T}^2(S^{\tau})\xr{\pr_{C_2}^*} E_{A\times \T}^2(S^{\tau})\xr{h_{A,\T}(S^{\tau})} E_A^2((S^{\tau})_{h_{\T}A})\cong E_A^2(\CPU) \]
to push the global orientation $t$ forward to an orientation $t^{(A)}$ for $E_A$. This uses that, similarly to the real case, the universal $A$-equivariant line bundle can be constructed by applying $(\T,A)$-homotopy orbits to the $\T$-bundle $\tau\to *$, and that its Thom space is again $\T$-homeomorphic to $\CPU$. That $t^{(A)}$ indeed defines an orientation for $E_A$ follows from the fact that the composite
\[ E_{\T}^2(S^{\tau})\to E_A^2((S^{\tau})_{h_{\T}A})\to E_A^2(S^V) \]
equals restriction along $V\colon A\to \T$, as a consequence of \Cref{lem:orbitrestriction}.

Furthermore, the homotopy orbit maps give rise to commutative diagrams
\[ \xymatrix{ \pi_*^{\T}(E) \ar[r]^-{\pr_{\T}^*} \ar[rd]_{V^*} & \pi_*^{A\times \T}(E) \ar[r] \ar[d]^{(\id_A,V)^*} & E_A^{-*}(\CPU_+) \ar[d]^{\theta(V)} \\
	& \pi_*^A(E) \ar[r]^{\cong} & E_A^{-*}      } \]
and
\[ \xymatrix{\pi_*^{A\times \T\times \T}(E) \ar[r] & E^{-*}(\CPU\times \CPU_+)\ar[r]^-{\cong} &  E^{-*}(\CPU_+)\hotimes E^{-*}(\CPU_+) \\
	\pi_*^{A\times \T}(E) \ar[r] \ar[u]^{(\id_A\times m)^*} & E_A^{-*}(\CPU_+) \ar[ur]_-{\Delta} \ar[u]_{m^*} \\
	\pi_*^A(E) \ar[u]^{\pr_A^*} \ar[r]_-{\cong} & E_A^{-*}.\ar[u] 
}\]
We also have:
\begin{Prop} \label{prop:complexcompletion} The homotopy orbit maps
	\[ h_{A,\T^r}(S^0)\colon \pi^{A\times \T^r}_*(E) \to E_A^{-*}(\CPU^{\times r}_+)\cong E_A^{-*}(\CPU_+)^{\hotimes r} \]
are completion with respect to finite products of the augmentation ideals \[ I_{\alpha}=\ker((\id_A,\alpha)^*\colon \pi_*^{A\times \T^r}(E)\to \pi_*^A(E)) \]
for homomorphisms $\alpha\colon A\to \T^r$.
\end{Prop}
\begin{proof} The proof is analogous to that of \Cref{prop:realcompletion}.
\end{proof}

Moreover, $h_{A,\T}(S^0)$ sends the Euler class $e_{(\epsilon,\id)}\in \pi_{-2}^{A\times \T}(E)$ to $y(\epsilon)\in E^2_\T(\CPU_+)$. Hence, the $A$-equivariant formal group law associated to $E_A$ is encoded in the functoriality of the coefficients $\upi_*(E)$ and the universal Euler class.

\subsection{The adjunction between global group laws and $A$-equivariant formal group laws} \label{sec:complexadjunction}

We now describe the process of constructing an $A$-equivariant formal group law out of an arbitrary global group law not necessarily induced by a complex oriented global ring spectrum. For the reasons described earlier, we first restrict to the case where $A$ is a torus. The constructions are very similar to those in Sections \ref{sec:realadjunction1}, \ref{sec:realadjunction2} and \ref{sec:realadjunction3}, so we will be brief.

Let $X$ be a global group law and $A,B$ tori. We denote by $X(A\times B)^{\wedge}_A$ the completion at finite products of ideals $I_{\alpha}$ for maps $\alpha\colon A\to B$ (as in \Cref{prop:complexcompletion}). Then we have:

\begin{Lemma}
	\begin{enumerate}
		\item Given another torus $B'$, the projection maps induce an isomorphism
		\[ X(A\times B)^{\wedge}_A\hotimes_{X(A)} X(A\times B')^{\wedge}_A\xr{\cong} X(A\times B\times B')^{\wedge}_A.\]
		\item The quintuple
		\[ (X(A),X(A\times \T)^{\wedge}_A,\Delta,\theta,y(\epsilon)) \]
		is an $A$-equivariant formal group law, where $\Delta$ is given by
		\[ (\id_A\times m)^*\colon X(A\times \T)^{\wedge}_A\to X(A\times \T\times \T)^{\wedge}_A\cong X(A\times \T)^{\wedge}_A\hotimes_{X(A)}X(A\times \T)^{\wedge}_A, \]
		$\theta(V)$ is induced by $(\id_A,V)^*\colon X(A\times \T)\to X(A)$ and $y(\epsilon)$ is the image of $e_{(\epsilon,\id)}\in X(A\times \T)$ under the completion map. We denote this $A$-equivariant formal group law by $X^{\wedge}_A$.
	\end{enumerate}
\end{Lemma}
\begin{proof} See \Cref{sec:realadjunction1} for the proof in the $2$-torsion case.
\end{proof}
Here, we make use of $A$ being a torus, because otherwise the element $e_{(\epsilon,\id)}\in X(A\times \T)$ would in general not be regular. In \Cref{sec:generaladjunction} we describe a way to get around this problem and associate an $A$-equivariant formal group law $X^{\wedge}_A$ to $X$ also when $A$ is not a torus.

To define the functor $(-)_{gl}$ in the other direction we do not need to restrict to tori. Let $A$ be an arbitrary abelian compact Lie group, and let $F=(k,R,\Delta,\theta,y(\epsilon))$ be an $A$-equivariant formal group law. We recall from \Cref{sec:realadjunction2} that $\cC_k$ denotes the category of complete commutative linear topological $k$-algebras. Given $S\in \cC_k$, let $R(S)$ denote the set of morphisms $R\to S$ in $\cC_k$ equipped with the abelian group structure defined via
\[ f+g=(R\xr{\Delta}R\hotimes R\xr{f\hotimes g}S\hotimes S\xr{m} S). \]
The existence of inverses for this addition follows from $R$ having a coinverse, see \Cref{rem:antipode}. Furthermore, the augmentation $\theta$ induces a group homomorphism $A^*\to R(k)$, implying that the functor $R(-)$ takes values in abelian groups under $A^*$. We have:
\begin{Lemma} The functor
	\[ \Hom_{A^*/}(A^*\times B^*,R(-)) \]
is representable for every torus $B$.
\end{Lemma}
\begin{proof} Choosing a basis $V_1,\hdots,V_n$ for $B^*$, we find that $R^{\hotimes n}$ is a representing object.
\end{proof}
We choose one such representing object $F_{gl}(B)$ for every $B$. Then the assignment $B\mapsto F_{gl}(B)$ has a canonical extended functoriality, as in \Cref{sec:extendedreal}. This means that every map $A\times B\to A\times B'$ which lies over the projections to $A$ induces a restriction map $F_{gl}(B')\to F_{gl}(B)$. If we only remember the restriction maps along $A\times B \xr{\id_A\times f}A\times B'$ for homomorphisms $f\colon B\to B'$, we obtain a functor
\[ F_{gl}\colon (\text{tori})^{op}\to \text{commutative rings}. \]
Furthermore we have:
\begin{Lemma}
	The class $y(\epsilon)\in R\cong F_{gl}(\T)$ is a coordinate for $F_{gl}$.
\end{Lemma}
\begin{proof} See the proof of \Cref{lem:realorientationregular}.
\end{proof}

\begin{Remark} \label{rem:completegeneral} The construction $(-)_{gl}$ works more generally for any cocommutative commutative complete linear topological algebra over $k$ which is equipped with a regular generator of its augmentation ideal and which allows an antipode. In this case one needs to set $A$ to be the trivial group in the above.
\end{Remark}

\begin{Prop} \label{prop:complexadjunction} Let $A$ be a torus. Then:
\begin{enumerate} 
	\item There is an adjunction
	\[ (-)^{\wedge}_A\colon GL_{gl} \rightleftarrows FGL_A \colon (-)_{gl}. \]
	The unit $X\to (X^{\wedge}_A)_{gl}$ is given in degree $B$ by the composite
	\[ X(B)\xr{\pr_B^*}X(A\times B) \to X(A\times B)^{\wedge}_A\cong (X^{\wedge}_A)_{gl}(B). \]
	The counit $(F_{gl})^{\wedge}_A\to F$ is the pair of maps $F_{gl}(A)\to F_{gl}(1)\cong k$ and $F_{gl}(A\times \T)^{\wedge}_A\to F_{gl}(\T)\cong R$ induced by the diagonals $\diag\colon A\to A\times A$ and $\diag\times \id_{\T}\colon A\times \T\to A\times A\times \T$ via the extended functoriality of $F_{gl}$.
	\item Given a complex oriented global ring spectrum $E$, there is a natural isomorphism 
	\[ (\upi_*(E))^{\wedge}_A \cong F(E_A), \]
	where $F(E_A)$ is the $A$-equivariant formal group law associated to the induced complex orientation of~$E_A$.
\end{enumerate}	
\end{Prop}
\begin{proof} As in the proof of \Cref{prop:realadjunction}, the adjunction factors naturally through the category $GL_{gl/A}$ of global group laws with extended functoriality. The functor from $GL_{gl/A}$ forgets the extended functoriality, with left adjoint sending $X$ to $X(A\times -)$. For the second adjunction, the functor $(-)_{gl/A}$ embeds $A$-equivariant formal group laws inside $GL_{gl/A}$ as the complete objects, and the adjunction unit is the completion map. We refer to the proof of \Cref{prop:realadjunction} for more details.
	
	Part (2) follows from \Cref{prop:complexcompletion} and the diagrams preceding it.
\end{proof}

\begin{Cor} \label{cor:lazardisuniversal} There is an isomorphism of $Ab$-algebras $\bL\cong \widetilde{\bL}$ sending the universal Euler class $e=e_{\tau}\in L_\T$ to the universal coordinate $e\in \widetilde{\bL}(\T)$.
	
Hence, the pair $(\bL,e)$ is the universal global group law.
\end{Cor}
\begin{proof} Since any left adjoint preserves initial objects, the $A$-equivariant formal group law $\widetilde{\bL}^ {\wedge}_A$ must be the universal one for any torus $A$. In particular, $\widetilde{\bL}(A)$ is canonically isomorphic to $L_A$. Moreover, given any global group law $X$, the Euler class $e_{\tau}\in X(\T)$ of the associated $\T$-equivariant formal group law is given as the pullback of $e_{(\epsilon,\id)}\in X(\T\times \T)$ along the diagonal $\T\to \T\times \T$, which gives back the orientation class $e\in X(\T)$. The same arguments as in the $2$-torsion case (see \Cref{sec:realadjunction3}) then apply to show that these degreewise isomorphisms assemble to a natural isomorphism of functors $(\text{tori})^ {op}\to (\text{commutative rings})$. Now the universal global group law $\widetilde{\bL}$ is by definition left Kan extended from tori, and \Cref{cor:lazardeulerquotient} shows that the same is true for $\bL$, hence we obtain a canonical isomorphism $\bL\cong \widetilde{\bL}$ of $Ab$-algebras, as desired.
\end{proof}

In particular we have the following regularity of Euler classes in Lazard rings, which we will later strengthen to all Euler classes for surjective characters  in \Cref{cor:lazardregular}:
\begin{Cor}
If $A$ is a torus and $V\colon A\to \T$ is split, then the Euler class $e_V\in L_A$ is regular.
\end{Cor}

\subsubsection{Adjunction for non-tori} \label{sec:generaladjunction}
\Cref{cor:lazardisuniversal} can be used to define an associated $A$-equivariant formal group law $X^{\wedge}_A$ for general abelian compact Lie groups $A$, in the following way. Evaluating the unique map $\bL\to X$ of global group laws at $A$ yields a map $L_A\to X(A)$. We define $X^{\wedge}_A$ to be the $A$-equivariant formal group law classified by this map. Unravelling the definitions, $X^ {\wedge}_A$ can be computed as follows: Choose an embedding $i_A\colon A\hookrightarrow T_A$ into a torus, and consider the pushforward of the $T_A$-equivariant formal group law $X^{\wedge}_{T_A}$ along $(i_A)^*\colon X(T_A)\to X(A)$. This defines a $T_A$-equivariant formal group law over $X(A)$ for which an Euler class $e_V$ vanishes if $V$ restricts to $\epsilon$ in $A^*$ (simply because $i_A^*$ sends these classes to $\epsilon$, see \Cref{sec:leftkan}). By \Cref{lem:fullsubcat}, this is the same datum as an $A$-equivariant formal group law; one checks that it agrees with~$X^ {\wedge}_A$. Note that the functor $(-)_{gl}\colon FGL_A\to GL_{gl}$ factors as $FGL_A\xr{(i_A)_*} FGL_{T_A}\xr{(-)_{gl}} GL_{gl}$. Since $(i_A)_*$ is right adjoint to the functor quotienting out all Euler classes $e_V$ for $V$ in the kernel of $i_A^ *$, it follows formally that there is an adjunction
\[ (-)^{\wedge}_A\colon GL_{gl} \rightleftarrows FGL_A \colon (-)_{gl}. \]
for general $A$.

If $E$ is a complex oriented global ring spectrum, then it is in general not quite true that $\upi^{\toral}_*(E)^ {\wedge}_A$ agrees with the $A$-equivariant formal group law $F(E_A)$ associated to $E_A$, since $\pi^A_*(E_A)$ does not necessarily agree with $\upi_*^{\toral}(E)(A)$. However, there is a natural map $\upi_*^{\toral}(E)(A)\to \pi^A_*(E_A)$ and one checks that $F(E_A)$ is the pushforward of $\upi^{\toral}_*(E)^ {\wedge}_A$ along this map. If $\upi^{\toral}_*(E)^ {\wedge}_A$ is regular, these problems disappear by \Cref{lem:complexregular} and $\upi^{\toral}_*(E)^ {\wedge}_A$ is canonically isomorphic to $F(E_A)$.

\subsection{The global Lazard ring is $2$-regular} \label{sec:lazardregular}
Our strategy for proving Theorems \ref{thm:A} and \ref{thm:C} is again to reduce to geometric fixed points. For this we need to extend our understanding of the regularity of Euler classes in the global Lazard ring beyond split characters. As $\upi_*(\MU)$ is concentrated in even degrees, we know that it is a regular global group law (\Cref{cor:muregular}). Hence every $e_V\in \pi_*^A(\MU)$ with $V\in A^*$ surjective is regular, for any abelian compact Lie group $A$. We want to prove that $\bL$ is isomorphic to $\upi_*(\MU)$, so the same must be true for $\bL$.  The goal of this subsection is to first show something weaker:
\begin{Prop} \label{prop:lazard2regular} The global Lazard ring $\bL$ is $2$-regular in the sense of \Cref{sec:regularfgl}: For every torus $A$ and every pair of linearly independent characters $V_1,V_2\in A^*$ the sequence $(e_{V_1},e_{V_2})$ is regular in $L_A$, i.e., $e_{V_1}$ is a regular element of $L_A$ and $e_{V_2}$ is a regular element of $L_A/(e_{V_1})\cong L_{\ker(V_1)}$.
\end{Prop}
This turns out to be enough to prove Theorems \ref{thm:A} and \ref{thm:C}. Knowing that $\upi_*(\MU)$ is regular, we can then deduce the regularity of $\bL$ from its isomorphism to $\upi_*(\MU)$. I know of no purely algebraic proof that $\bL$ is fully regular.

We now turn to the proof of \Cref{prop:lazard2regular}. Let $A$ be a torus. By now we only know that the Euler classes for split characters are regular in $L_A$. If we want to apply our previous method for showing Euler class regularity once more, we run into the following problem: It is in general not the case that the global group law $F_{gl}$ associated to an $A$-equivariant formal group law $F$ is $2$-regular. For example,  the Euler class $e_p\in F_{gl}(\T)$ for the $p$-th power map on the circle is given by the $p$-series $[p]_F(y(\epsilon))$. Taking $A$ to be the trivial group and $F$ the additive formal group law over $\F_p$, this $p$-series is trivial and hence far from being a regular element. This means that the adjunction of \Cref{prop:complexadjunction} does not restrict to an adjunction between $2$-regular global group laws and all $A$-equivariant formal group laws. Instead one needs to restrict to a subclass of $A$-equivariant formal group laws satisfying extra conditions. This is problematic because we do not know a priori  that the universal $A$-equivariant formal group law satisfies these conditions. We overcome this problem in two steps:
\begin{enumerate}
	\item[Step 1:] Localizing at a prime $p$ gives us control over all sequences $(e_{V_1},e_{V_2})$ for linearly independent characters $V_1,V_2\in A^*$ which remain linearly independent in the quotient $A^*/pA^*$.
	\item[Step 2:] The Euler classes for the remaining linearly independent pairs of characters are then studied via the effect of the $p$-th power map on $L_A$ and $L_{\ker(V_1)}$, using \Cref{prop:eulerinvert} on the interplay of localized Lazard rings and short exact sequences of abelian compact Lie groups.
\end{enumerate}
We start with \textbf{Step 1}: Let $p$ be a prime number.
\begin{Def} We say that a global group law $X$ is \emph {$(p,2)$-regular} if
\begin{enumerate}
	\item $X(1)$ is a $\Z_{(p)}$-algebra (and hence so is $X(A)$ for any $A$), and
	\item for all tori $A$ and characters $V_1,V_2\in A^*$ which are linearly independent modulo $p$, i.e., their images in the $\F_p$-vector space $A^*/pA^*$ are linearly independent, the Euler classes $(e_{V_1}, e_{V_2})$ form a regular sequence in $X(A)$.
\end{enumerate}
\end{Def}

\begin{Lemma} \label{lem:p2regular} Let $k$ be a $\Z_{(p)}$-algebra and $B$ an abelian compact Lie group. Then for every $B$-equivariant formal group law $F$ with ground ring~$k$, the global group law $F_{gl}$ is $(p,2)$-regular.
\end{Lemma}
\begin{proof} Let $A$ be a torus and $V_1,V_2\in A^*$ be linearly independent modulo $p$. We first make a few preliminary reductions on $V_1$ and $V_2$ which apply to any global group law $X$. To start, we can choose coordinates on $A\cong \T^r$ and hence $A^*\cong \Z^r$ so that $V_1$ and $V_2$ are given by the vectors $(n_1,0,\hdots,0)$ and $(a,n_2,0,\hdots,0)$ respectively, where $n_1$ and $n_2$ are natural numbers coprime to $p$.
We now argue that for showing that $(e_{V_1},e_{V_2})$ form a regular sequence in $X(A)$ for all such pairs $V_1,V_2$, we can reduce to the case where $a=0$ and hence the coefficient vector for $V_2$ is of the form $(0,n_2,0,\hdots,0)$ with $n_2$ coprime to $p$. For this we consider the character $V_2^{n_1}$, the $n_1$-fold power of $V_2$, whose coefficient vector is $(a\cdot n_1,n_2\cdot n_1,0,\hdots,0)$. Since the Euler class $e_{V_2}$ divides the Euler class $e_{V_2^{n_1}}$ (this is a consequence of the universal Euler class $e_{\tau}=e\in X(\T)$ dividing the Euler class $e_{\tau^n}$, which follows from the fact that $e_{\tau^n}$ is sent to $0$ under restriction to the trivial group), it suffices to show that $e_{V_2^{n_1}}$ is a regular element in the quotient $X(A)/(e_{V_1})$. But modulo $e_{V_1}$, the Euler class $e_{V_2^{n_1}}$ agrees with the Euler class associated to the vector $(0,n_2\cdot n_1,0,\hdots,0)$, since the character with coefficient vector $(a\cdot n_1,0,\hdots,0)$ vanishes in the quotient $A^*/(V_1)$ (cf. \Cref{sec:leftkan}). This proves the claim.

We now specialize to the case $F_{gl}$ for a $B$-equivariant formal group law $B$. By the above, we have to show that for every pair of natural numbers $n_1,n_2$ coprime to $p$ and $r\geq 2$, the Euler classes $e_{V_1},e_{V_2}$ associated to the vectors $V_1=(n_1,0,\hdots,0)$ and $V_2=(0,n_2,0,\hdots,0)$ form a regular sequence in $F(\T^r)$. We have an isomorphism
\[  F_{gl}(A)\cong R\hotimes R\hotimes R^{\hotimes r-2}, \]
with Euler classes $e_{V_1}$ and $e_{V_2}$ given by $[n_1]_F(y(\epsilon))\otimes 1\otimes 1$ and $1\otimes [n_2]_F(y(\epsilon))\otimes 1$. Modulo $y(\epsilon)^2$, the $n_i$-series $[n_i]_F(y(\epsilon))$ equals $n_i y(\epsilon)$. By assumption, $k$ is a $\Z_{(p)}$-algebra, and hence $n_1$ and $n_2$ are units in $k$. It follows that both $[n_i]_F(y(\epsilon))$ are regular elements in $R$ (note that in order to be a regular element in $R$, it is sufficient that the leading term with respect to any choice of flag is a regular element in~$k$), which directly implies that $(e_{V_1},e_{V_2})$ is a regular sequence in $F_{gl}(A)$.
\end{proof}

\begin{Remark} The proof shows that the $p$-regularity of global group laws associated to $p$-local $B$-equivariant formal group laws is not specific to the number two. More generally, let $A$ be a torus and $B$ an abelian compact Lie group. Further, let $F$ be a $B$-equivariant formal group law defined over a $p$-local ring $k$ and $V_1,\hdots,V_n\in A^*$ an $n$-tuple of characters which is linearly independent modulo $p$. Then $(e_{V_1},\hdots,e_{V_n})$ is a regular sequence in  $F_{gl}(A)$.
\end{Remark}
The lemma implies that the adjunction between global group laws and $B$-equivariant formal group laws restricts to an adjunction
\[   (-)_B^{\wedge}\colon GL_{gl}^{(p,2)-\reg}\rightleftarrows FGL^{(p)}_B\colon (-)_{gl},\]
where $FGL^{(p)}_B\subseteq FGL_B$ denotes the category of $B$-equivariant formal group laws over $p$-local rings.

Once again we want to apply \Cref{lem:categorical} (this time to the category of $(\Ze)_{(p)}$-algebras) to see that there exists a universal $(p,2)$-regular global group law. For this we need to see that $(p,2)$-regular global group laws are the local objects for some set of morphisms. Note that we can rephrase the conditions for a $(\Ze)_{(p)}$-algebra $X$ to be a $(p,2)$-regular global group law in the following way, where $A$ ranges through all tori:
\begin{enumerate}
	\item If $V\in A^*$ is a split character, then
	\[ 0\to X(A)\xr{e_V\cdot} X(A)\xr{\res^ {A}_{\ker(V)}}X(\ker(V))\to 0 \]
	is exact.
	\item \label{eq:p2regular} Let $V_1,V_2\in A^*$ be characters which are linearly independent modulo $p$. If two elements $x,y\in X(A)$ satisfy
	\[ x\cdot e_{V_1}=y\cdot e_{V_2},\]
	then $y$ is uniquely divisible by $e_{V_1}$.
\end{enumerate}
To see the equivalence, note that in the second item, divisibility of $y$ by $e_{V_1}$ guarantees that $e_{V_2}$ is a regular element in $X(A)/(e_{V_1})$, and the uniqueness implies that $e_{V_1}$ is a regular element in $X(A)$ (set $x=y=0$). We have already seen that properties of type (1) are equivalent to being local with respect to the morphisms $(f_{A,V})_{(p)}$, see \Cref{prop:reflexiveglobal}. For the second type we can take
\[ g_{A,V_1,V_2}\colon (\Ze)_{(p)}[x,y]/(x\cdot e_{V_1}-y\cdot e_{V_2})\to (\Ze)_{(p)}[x,y,y']/(x\cdot e_{V_1}-y\cdot e_{V_2}, e_{V_1}\cdot y'-y). \]
Hence $(p,2)$-regular global group laws are the local objects with respect to this set of morphisms, and we obtain a universal $(p,2)$-regular global group law $\bL^{(p,2)}$.
\begin{Remark}
The existence of a universal object is less clear for $(p,n)$-regularity with $n>2$ (defined in the obvious way). This is the reason why in this section we have to restrict to showing $2$-regularity of $\bL$, and later use the comparison with $\upi_*(\MU)$ to deduce the full regularity.
\end{Remark}
It follows that for any torus $A$, the left adjoint functor $(-)^{\wedge}_A$ must take $\bL^{(p,2)}$ to the universal $A$-equivariant formal group law over $p$-local rings. By the same arguments as in Sections \ref{sec:realadjunction3} and \ref{sec:complexadjunction} we deduce that $\bL^{(p,2)}$ is canonically isomorphic to $\bL_{(p)}$, the $p$-localization of the global Lazard ring. In particular we have shown that $\bL_{(p)}$ is $(p,2)$-regular, finishing the proof of Step 1. We note that this in particular implies the following:
\begin{Cor} \label{cor:regcoprimep} For every torus $A$ and character $V\in A^*$ which is not divisible by $p$, the Euler class $e_V\in (L_A)_{(p)}$ is a regular element.
\end{Cor}
\begin{proof} We consider the larger torus $\T\times A$ and the pair of characters $(V_1,q^*V)$, where $V_1$ is the projection to the first component and $q\colon \T \times A\to A$ is the projection. Then $(V_1,q^*V)$ are linearly independent modulo $p$, and hence the reduction of $e_{q^*V}$ to the quotient $(L_{\T \times A})_{(p)}/(e_{V_1})$ is a regular element. Under the isomorphism $(L_{\T \times A})_{(p)}/(e_{V_1})\cong L_A$, this reduction agrees with $e_V$, proving the claim.
\end{proof}

We now come to \textbf{Step 2}: Let $A$ be a torus. Continuing to work $p$-locally, it remains to consider Euler classes of pairs of characters $V_1,V_2\in A^*$ which are linearly independent, but not necessarily linearly independent modulo $p$. We first discuss the case where $V_1$ is not divisible by $p$ but $V_2$ is arbitrary. Hence we know by \Cref{cor:regcoprimep} that $e_{V_1}\in (L_A)_{(p)}$ is regular, and our goal is to show that $\res^{A}_{\ker(V_1)} (e_{V_2})$ is regular in $(L_{\ker(V_1)})_{(p)}$. For this we choose a homomorphism $\beta_p\colon \ker(V_1)\to \ker(V_1)$ which induces the identity on $\pi_0(\ker(V_1))$ and equals the $p$-th power map on the identity component. We have the following lemmata:
\begin{Lemma} \label{lem:eWregular1} Let $V\in \ker(V_1)^*$ be a character not in the image of $\beta_p^*\colon \ker(V_1)^*\to \ker(V_1)^*$. Then the associated Euler class $e_V\in (L_{\ker(V_1)})_{(p)}$ is a regular element.
\end{Lemma}
\begin{proof} A character $V\in \ker(V_1)^*$ lies in the image of $\beta_p^*$ if and only if its restriction to the identity component $(\ker(V_1))_1^*$ is divisible by $p$. Hence, if $V$ is not in the image, it is (1) itself not divisible by $p$ and (2) a surjective character (i.e., of infinite order in $\ker(V_1)^*$), since any torsion character is trivial on the identity component and hence in the image of $\beta_p^*$. Let $W$ be a lift of $V$ to a character for $A$. Then combining (1) and (2) with the assumption that $V_1$ is not divisible by $p$ shows that the pair $(V_1,W)$ is linearly independent modulo $p$. So Step 1 implies that $e_V\in (L_{\ker(V_1)})_{(p)}$ is regular, as desired.
\end{proof}
\begin{Lemma} \label{lem:eWregular}Let $W\in \ker(V_1)^*$ be a character and assume that $e_W\in (L_{\ker(V_1)})_{(p)}$ is regular. Then $e_{\beta_p^*(W)}\in (L_{\ker(V_1)})_{(p)}$ is also regular.
\end{Lemma}
\begin{proof}
Since $\beta_p$ is surjective, we obtain a short exact sequence
	\[ 1\to B\to \ker(V_1) \xr{\beta_p} \ker(V_1) \to 1. \]
Applying \Cref{prop:eulerinvert}, we see that the composition
\[ (L_{\ker(V_1)})_{(p)}\xr{(\beta_p)^*}(L_{\ker(V_1)})_{(p)}\to (L_{\ker(V_1)})_{(p)}[e_V^{-1}\ |\ V\notin \im(\beta_p^*)]\]
is a flat map. This means that if $e_W$ is regular in $(L_{\ker(V_1)})_{(p)}$, then its image under $(\beta_p)^*$, i.e., $e_{\beta_p^*(W)}$, is a regular element in $(L_{\ker(V_1)})_{(p)}[e_V^{-1}\ |\ V\notin \im(\beta_p^*)]$. By \Cref{lem:eWregular1}, each $e_V\in (L_{\ker(V_1)})_{(p)}$ for $V\notin \im(\beta_p^*)$ is a regular element. Hence the map
\[ (L_{\ker(V_1)})_{(p)}\to (L_{\ker(V_1)})_{(p)}[e_V^{-1}\ |\ V\notin \im(\beta_p^*)] \]
is injective and therefore $e_{\beta_p^*(W)}$ must also be regular in $(L_{\ker(V_1)})_{(p)}$, proving the lemma.
\end{proof}
Now, since $\res^{A}_{\ker(V_1)} (V_2)\in \ker(V_1)^*$ is surjective, it can be written as $\res^{A}_{\ker(V_1)} (V_2)=(\beta_p^*)^m(W)$ for some $m\in \N$ and surjective character $W\in \ker(V_1)^*$ whose restriction to the component of the identity is not divisible by~$p$, i.e., which is not in the image of $\beta_p^*$. By \Cref{lem:eWregular1}, the associated Euler class $e_W$ of $W$ is then regular in $(L_{\ker(V_1)})_{(p)}$. Iterated application of \Cref{lem:eWregular} shows that $e_{\res^{A}_{\ker(V_1)}(V_2)}$ is also regular in $(L_{\ker(V_1)})_{(p)}$. This finishes the proof of the case where $V_1$ is not divisible by $p$ and $V_2$ is arbitrary. Again, this in particular implies the following:
\begin{Cor} Let $A$ be a torus and $V\in A^*$ a non-trivial character. Then $e_V\in (L_A)_{(p)}$ is regular.
\end{Cor}
\begin{proof} The proof is analogous to that of \Cref{cor:regcoprimep}: Let $V_1\in (\T \times A)^*$ be the projection to the first component and $q\colon \T \times A\to A$ the projection to the second component. Then the pair of characters $(V_1,q^*V)\in (\T\times A)^*$ is linearly independent with $V_1$ not divisible by $p$. Hence by what we just saw the Euler class $\res^{\T\times A}_A(e_{q^*V})=e_V\in (L_A)_{(p)}$ is regular.
\end{proof}
Knowing this, we see that for tori $A$ and non-trivial characters $V_1,V_2\in A^*$, the pair $(e_{V_1},e_{V_2})$ is a regular sequence if and only if the pair $(e_{V_2},e_{V_1})$ is a regular sequence. Indeed, if $x\cdot e_{V_1}=y\cdot e_{V_2}$ and we assume that $(e_{V_1},e_{V_2})$ is regular, the element $y$ must be divisible by $e_{V_1}$, say $y=y'\cdot e_{V_1}$. This yields that $x\cdot e_{V_1}=y'\cdot e_{V_2} \cdot e_{V_1}$, which by regularity of $e_{V_1}$ implies that $x$ is divisible by $e_{V_2}$. Since we already know that $e_{V_2}$ is regular, this proves that the sequence $(e_{V_2},e_{V_1})$ is regular. 

Hence, we also know that $(e_{V_1},e_{V_2})$ is a regular sequence in $(L_A)_{(p)}$ if $(V_1,V_2)$ are linearly independent and $V_2$ is not divisible by $p$, with no further condition on $V_1$. But now we can choose $\beta_p\colon \ker(V_1)\to \ker(V_1)$ as above and repeat the exact same argument to deduce the case where $V_2$ is also arbitrary. This finishes the proof that $\bL_{(p)}$ is $2$-regular for all $p$. Since regularity can be checked locally, \Cref{prop:lazard2regular} follows.

\begin{Cor} \label{cor:lazarddomain}
	If $A$ is a torus, the Lazard ring $L_A$ is an integral domain.
\end{Cor}
\begin{proof} By \Cref{prop:geomlazard} we know that for any abelian compact Lie group $A$ the Lazard ring $L_A$ becomes an integral domain after inverting all Euler classes of non-trivial characters, since the non-equivariant Lazard ring $L$ is a polynomial algebra over the integers \cite{Laz55}. If $A$ is a torus we just saw that these Euler classes are regular elements, so it follows that $L_A$ itself is a domain.
\end{proof}

\subsection{Geometric fixed points and proof of the main theorem}

Let $X$ be a global group law. As in \Cref{sec:realmain}, we define the geometric fixed points $\Phi^A(X)$ at any abelian compact Lie group $A$ by \[ \Phi^A(X)=X(A)[e_V^{-1} \ |\ V \in A^*-\{\epsilon\}].\]
If $X=\upi_*(E)$ is a regular global group law associated to a global complex oriented ring spectrum, then there is a natural isomorphism
\[ \Phi^A(\upi_*(E))\cong \pi_*(\Phi^A(E))=\Phi^A_*(E).\]
For general complex oriented $E$, such an isomorphism only holds at tori.

We now move towards the proof of Theorems \ref{thm:A} and \ref{thm:C} and consider the unique map of global group laws
\[ \alpha \colon \bL \to \upi_*(\MU).\]
Again, the only input from topology that we need is the following:
\begin{Prop}[Greenlees] \label{prop:mulazardgeom}
	The morphism $\alpha\colon \bL \to \upi_*(\MU)$ of global group laws induces an isomorphism on geometric fixed points $\Phi^A$ for every abelian compact Lie group $A$.
\end{Prop}
\begin{proof} The computation of $\Phi^A(MU_A)$ is due to tom Dieck \cite{tD70} and also appears in \cite[Theorem 4.9]{Sin01}. In \Cref{prop:geomlazard} we saw a formula for $\Phi^A(\bL)$ (originally also due to Greenlees), showing that the rings are abstractly isomorphic. The proof that $\varphi_A$ is indeed an isomorphism is given in \cite[Section 11]{Gre} for finite~$A$, but applies without change to infinite $A$ (it is also stated in this generality in \cite[Section 13.B.]{Gre01}). We refer to \Cref{prop:realgeomiso} for a proof of the real analog, which straightforwardly translates to the complex case.
\end{proof}

For real bordism there was a simple Whitehead theorem (\Cref{prop:whitehead}) that allowed us to conclude directly from the geometric fixed point statement that the map $\bL^{2-\tor}\to \upi_*(\MO)$ is an isomorphism. In the complex case we have to work a little harder. We first show the following structural property of $L_{\T}$, where we denote by $e_n\in L_{\T}$ the Euler class associated to the $n$-th power map on $\T$:

\begin{Prop} \label{prop:psi} There exists a family of elements $\{\psi_n\}_{n\in \N_{>0}}$ of $\bL(\T)$ uniquely characterized by the equations
	\begin{equation} \label{eq:eulerproduct} e_n=\prod_{m|n}\psi_m \end{equation}
for all $n\in \N_{>0}$. Moreover, each $\psi_n$ is a prime element and generates the kernel of the composite
\[ \varphi^{\T}_{C_n}\colon \bL(\T)\xr{\res_{C_n}^{\T}} \bL(C_n)\to \Phi^{C_n}(\bL).\]
\end{Prop}
\begin{Example}[Cyclotomic polynomials] \label{ex:psimult} Once the proposition is proved, it defines elements $\psi_n\in X(\T)$ for any global group law $X$. To understand the nature of these elements, it is helpful to consider the example of the multiplicative global group law over $\Z$ (\Cref{ex:mult}). Its value at the circle group is the Laurent polynomial ring $\Z[t^{\pm 1}]$ with universal Euler class $t-1$. The element $\psi_n$ is then the $n$-th cyclotomic polynomial, and equation \eqref{eq:eulerproduct} becomes the standard decomposition of $t^n-1$ into a product of cyclotomic polynomials. To see why the $\psi_n$ are as claimed in this case, note that the multiplicative global group law is regular and hence the $\psi_n\in \Z[t^{\pm 1}]$ are uniquely characterized by the equations \eqref{eq:eulerproduct}. Since $\psi_1=e_1=t-1$, the claim follows.
\end{Example}

\begin{proof}[Proof of \Cref{prop:psi}] 	
We prove the proposition by induction on $n$. By definition, $\psi_1$ is given by $e_1$. We now fix an $n>1$ and assume we have proved the existence and uniqueness of $\psi_m$ for all $m$ properly dividing~$n$, and that these classes satisfy the properties stated in the proposition. Recall from \Cref{cor:lazarddomain} that $\bL(\T)$ is a domain, so there can be at most one $\psi_n$ satisfying \eqref{eq:eulerproduct}.
	
It remains to prove existence. Note that $e_n$ maps to $0$ in $\bL(C_m)$ for all $m$ dividing $n$. Hence $e_n$ is uniquely divisible by $e_m$ for all these $m$. In particular, each individual such $\psi_m$ divides $e_n$. We need to show that their product also divides $e_n$, as then we can (and have to) define $\psi_n$ as the unique element satisfying equation \eqref{eq:eulerproduct}. Since the $\psi_m$ are prime elements, this will follow once we know that no two $\psi_m$ and $\psi_{m'}$ with $m\neq m'$ agree up to a unit. To see that this is indeed not the case it suffices to find a single example of a global group law for which (the push-forwards of) $\psi_m$ and $\psi_{m'}$ do not differ by a unit. The multiplicative global group law described in \Cref{ex:psimult} does the job, since it is clear that no two cyclotomic polynomials differ by a unit in $\Z[t^{\pm 1}]$ as the roots of $\psi_m$ in $\C$ are given by the primitive $m$-th roots of unity. This proves the existence and uniqueness of $\psi_n$. The properties of the multiplicative global group law also imply that $\psi_n$ does not divide any of the $\psi_m$ with $m<n$ and vice versa, since this is not the case for cyclotomic polynomials.

It remains to see that $\psi_n$ generates the kernel of the map $\varphi^{\T}_{C_n}$. First note that the product $\psi_n\cdot e_1\cdot \hdots \cdot e_{n-1}$ is divisible by $e_n$, since each $\psi_m$ divides $e_m$. Each $e_i$ with $0<i<n$ is by definition sent to a unit under $\varphi^{\T}_{C_n}$, while $e_n$ is sent to $0$. It follows that $\psi_n$ lies in the kernel of $\varphi^{\T}_{C_n}$. For the other direction, assume that $x\in \bL(\T)$ is sent to $0$ under $\varphi^{\T}_{C_n}$. This means that its restriction to $\bL(C_n)$ is Euler class torsion. In other words there exist numbers $m_1,\hdots,m_{n-1}$ such that $x\cdot e_1^{m_1}\cdots e_{n-1}^ {m_{n-1}}$ is divisible by $e_n$. This uses that the kernel of the restriction map is generated by $e_n$ and that $e_1,\hdots,e_{n-1}$ restrict to all the non-trivial Euler classes of $C_n$. Since $\psi_n$ divides $e_n$, we can hence write
\[ x\cdot e_1^{m_1}\cdots e_{n-1}^ {m_{n-1}}=x'\cdot \psi_n \]
for some $x'\in \bL(\T)$. Now each $e_i$ is a product of the regular prime elements $\psi_j$ with $j<n$. As we argued above, none of the $\psi_j$ divide $\psi_n$, so it follows that the whole product $e_1^{m_1}\cdots e_{n-1}^ {m_{n-1}}$ must divide $x'$. We can cancel by $e_1^{m_1}\cdots e_{n-1}^ {m_{n-1}}$ on both sides, showing that $x$ is divisible by $\psi_n$, as desired. Finally, that $\psi_n$ is again prime follows from the fact that $\varphi^{\T}_{C_n}$ defines an embedding of $\bL(\T)/(\psi_n)$ into $\Phi^{C_n}(\bL)$, which is an integral domain by \Cref{prop:eulerinvert}. This finishes the proof.
\end{proof}

In fact, more is true:
\begin{Lemma} \label{lem:psi} Let $A$ be a torus, $V\colon A \to \T$ be a split character and $n\in \N_{>0}$. Then $V^*{\psi_n}\in \bL(A)$ is a prime element and generates the kernel of the composition
\[ \varphi^{A}_{\ker(V^n)}\colon \bL(A)\xr{\res^{A}_{\ker(V^n)}} \bL(\ker(V^n))\to \Phi^{\ker(V^n)}\bL.\]
\end{Lemma}
\begin{proof} Since the global Lazard ring $\bL$ is $2$-regular (\Cref{prop:lazard2regular}), all Euler classes $e_{W}$ for surjective characters $W\in (\ker(V^n))^*$ are regular elements of $\bL(\ker(V^n))$. Hence, in order to determine the kernel of $\bL(\ker(V^n))\to \Phi^{\ker(V^n)}\bL$, it suffices to consider the Euler-torsion for the non-surjective, non-trivial characters of $\ker(V^n)$. These are given precisely by the restrictions of the characters $V,V^2,\hdots,V^ {n-1}\in A^*$ to $\ker(V^n)^*$. It then follows that an element $x\in \bL(A)$ lies in the kernel of $\varphi^{A}_{\ker(V^n)}$ if and only if $x\cdot e_V^{m_1}\cdots e_{V^{n-1}}^ {m_{n-1}}$ is divisible by $e_{V^n}$ for some numbers $m_1,\hdots,m_{n-1}$. Furthermore, no $V^*\psi_m$ with $m<n$ can divide $V^*\psi_n$ as this would imply that also $\psi_m$ divides $\psi_n$ by pulling back along a section $\T\to A$ of $V$. Knowing this, the proof proceeds as the one above for \Cref{prop:psi}.
\end{proof}

\Cref{lem:psi} further shows that for any global group law $X$ and torus $A$, the geometric fixed points $\Phi^{A}X$ can also be obtained from $X(A)$ by inverting $V^*\psi_n$ for all split characters $V\colon A\to \T$ and $n\in \N_{>0}$. This is because any non-trivial character $W$ can be written as $V^n$ for some split character $V$, and $e_{V^n}$ decomposes as
\[ e_{V^n}=V^*e_n=\prod_{m|n}V^* \psi_m .\]

\begin{proof}[Proof of Theorems \ref{thm:A} and \ref{thm:C}] We wish to show that the canonical map
	\[ \alpha: \bL \to \upi_*(\MU)\]
is an isomorphism, and we have already seen in \Cref{prop:mulazardgeom} that it induces an isomorphism on geometric fixed points for all abelian compact Lie groups. Since both $\bL$ and $\upi_*(\MU)$ are left induced from tori, it is sufficient to show that $\alpha$ induces an isomorphism on tori. Let $A$ be a torus. All Euler classes $e_V\in \bL(A)$ with $V\neq \epsilon\in A^*$ are regular elements, so it follows that the maps $\bL(A)\to \Phi^{A}(\bL)$ are injective. Since $\Phi^A(\alpha)$ is an isomorphism, this implies that $\alpha(A)$ is injective. It remains to prove that $\alpha(A)$ is also surjective.	Again using that $\Phi^{A}(\alpha)$ is an isomorphism, it is sufficient to show that if $V^*\psi_n\cdot y\in \pi^{A}_*(\MU)$ is in the image for some split character $V$ and $n\in \N_{>0}$, then so is  $y$. For this, let $x\in \bL(A)$ be a preimage of $V^*\psi_n\cdot y$. Since $\varphi^{A}_{\ker(V^n)}(V^*\psi_n)=0\in \Phi^A_*(\MU)$, it follows that $\varphi^{A}_{\ker(V^n)}(x)$ lies in the kernel of $\Phi^{\ker(V^n)}(\alpha)$. But $\Phi^{\ker(V^n)}(\alpha)$ is an isomorphism, so $x$ lies in the kernel of $\varphi^{A}_{\ker(V^n)}$. Hence \Cref{lem:psi} implies that $x$ is divisible by $V^*\psi_n$, yielding an equation of the form
\[ V^*\psi_n\cdot y = \alpha(x)=\alpha(V^*\psi_n\cdot x')=V^*\psi_n \cdot \alpha(x') \in \pi_*^ {A}(\MU) \]
for some $x'\in \bL(A)$. Since $\pi_*^ {A}(\MU)$ is an integral domain and $V^*\psi_n$ is non-trivial (both statements follow from the embedding $\pi_*^ {A}(\MU) \hookrightarrow \Phi_*^ {A}(\MU)$), we can conclude that $y=\alpha(x')$. This finishes the proof of Theorems \ref{thm:A} and \ref{thm:C}.
\end{proof}

\begin{Cor} \label{cor:lazardregular} The global Lazard ring $\bL$ is a regular global group law. Moreover, for every abelian compact Lie group $A$ the ring $L_A$ is free as a module over the Lazard ring $L$.
\end{Cor}
\begin{proof} We know that $\upi_*(\MU)$ is regular, since it is concentrated in even degrees. Moreover, Comeza{\~n}a \cite{Com96} showed that $\pi_*^A(\MU_A)$ is free as a module over $\pi_*(MU)$. Hence, by Theorems \ref{thm:A} and \ref{thm:C}, the analogous statements hold for $\bL$.	
\end{proof}

\begin{Remark} \label{rem:generalize} It is useful to abstract what properties of $\bL$ and $\upi_*(\MU)$ we used to conclude that the map $\bL\to \upi_*(\MU)$ is an isomorphism from only knowing that it is an isomorphism on geometric fixed points.  Let $f\colon X\to Y$ be a morphism of global group laws such that
\begin{enumerate}	
	\item $X$ is $2$-regular, all geometric fixed points $\Phi^A(X)$ are integral domains, and the classes $\psi_n,\psi_m\in X(\T)$ for $n\neq m$ are not multiples of one another. 
	\item $Y$ is $1$-regular.
\end{enumerate}
Note that the condition on the $\psi_i$ follows if $X$ maps to any global group law with this property, for example the multiplicative one. Then if $\Phi^Af$ is an isomorphism for all $A$ (in fact it suffices to consider $A$ with $\pi_0(A)$ cyclic), the map $f$ is already an isomorphism. Indeed, the conditions on $X$ guarantee that the classes $\psi_n$ satisfy the analogs of \Cref{prop:psi} and \Cref{lem:psi}, and the $1$-regularity of $Y$ implies that $Y(A)$ is a domain for all tori $A$, so the same proof applies.
\end{Remark}

\subsection{Cooperations} \label{sec:cooperations} Several results in chromatic homotopy theory rely not only on understanding the coefficients $MU_*$ but also the cooperations $MU_*MU$ and higher analogs $\pi_*(MU^{\wedge n})$. For example, these play an essential role in the Adams-Novikov spectral sequence. The goal of this final section is to give a description of the cooperations for global equivariant bordism.

Let $X$ be an $Ab$-algebra equipped with two coordinates $e^{(1)}$ and $e^{(2)}$. Then $e^{(2)}=\lambda e^{(1)}$ for a unit $\lambda\in X(\T)$. We say that the tuple $(e^ {(1)},e^{(2)})$ is \emph{strict} if the unit $\lambda$ restricts to $1$ at $X(1)$. More generally we say that an $n$-tuple $(e^{(1)},\hdots,e^{(n)})$ of coordinates is strict if every pair $(e^{(i)},e^{(j)})$ is strict.	

Now assume given $n$ complex oriented global ring spectra $E_1,\hdots,E_n$. We obtain $n$ orientations $t_1,\hdots,t_n$ in $(E_1\wedge \hdots \wedge E_n)^2_\T(S^\tau)$, differing from one another by units in $\pi_0^{\T}(E_1\wedge \hdots \wedge E_n)$. Since each of the orientation classes restrict to $1$ at the trivial group, the same must be true for the units relating them. So it follows that the $n$-tuple of coordinates $(e^{(1)},\hdots,e^{(n)}))$ associated to these orientations is strict.

Our goal is to show the following:
\begin{Theorem}[{\Cref{thm:E} in the introduction}]  \label{thm:coops} The tuple $(\upi_*(\MU^{\wedge n}),e^{(1)},\hdots,e^{(n)})$ is universal among $Ab$-algebras equipped with a strict $n$-tuple of coordinates.	
\end{Theorem}

We can obtain a universal $Ab$-algebra with $n$ coordinates $\bL^{(n)}$ by applying the left adjoint from \Cref{prop:reflexiveglobal} to the $\Ze$-algebra \[ \Ze[\lambda_2^{\pm 1},\hdots,\lambda_n^{\pm 1}]/(\res^{\T}_1(\lambda_i)-1), \]
where $\lambda_i$ are free unit variables at the circle group $\T$. The associated strict $n$-tuple is then $(e,\lambda_2e,\lambda_3e,\hdots,\lambda_ne)$.
We hence obtain a map
	\[ \alpha^{(n)}\colon \bL^{(n)} \to \upi_*(\MU^{\wedge n}). \]
To show that $\alpha^{(n)}$ is an isomorphism we again need to describe the values $\bL^{(n)}(A)$ in terms of $A$-equivariant formal group laws. By the adjunction of \Cref{prop:complexadjunction}, the ring $\bL^{(n)}(A)$ is initial among $A$-equivariant formal group laws equipped with a strict $n$-tuple of coordinates. This shows that $\bL^{(n)}(A)$ is flat over $\bL(A)=L_A$ (where we view $\bL^{(n)}$ as an algebra over $\bL$ via the first coordinate), since once the $A$-equivariant formal group law and the first coordinate are specified, it remains to give the units $\lambda_2,\hdots,\lambda_n\in R^\times$ which augment to $1$. Choosing a flag $F$ starting with the trivial character, this amounts to giving the coefficients $a_j^F(\lambda_i)$ for $j\geq 1$ (as $a_0^F(\lambda_i)$ is required to be $1$), such that the associated Euler classes, which can be expressed as polynomials in the coefficients $a_j^F(\lambda_i)$, are units (see \Cref{lem:unit} and the paragraph preceding it). Hence, the ring classifying such data is a localization of a polynomial ring over $L_A$. This also shows that $\bL^{(n)}(\T^r)$ is again an integral domain for all $r$, and that $\Phi^A(\bL^{(n)})$ is an integral domain for all abelian compact Lie groups~$A$.

We can express the universal property of $\bL^{(n)}(A)$ a little differently: It is initial among rings $k$ equipped with an $n$-tuple of $A$-equivariant formal group laws $(F_1,\hdots,F_n)$ together with strict isomorphisms $F_i\xr{\cong}F_{i+1}$ for $i=1,\hdots,n-1$. By the latter we mean isomorphisms $R_i\to R_{i+1}$ over $k$ that are compatible with the maps $\Delta$ and $\theta$ and send the coordinate $y_i(\epsilon)$ to $\lambda y_{i+1}(\epsilon)$ for a unit $\lambda\in R_{i+1}^\times$ which augments to $1\in k$.

Note that every strict isomorphism $\varphi\colon F\xr{\cong}F'$ gives rise to a strict isomorphism $\widetilde{\varphi}\colon \widetilde{F}\xr{\cong} \widetilde{F}'$ between the associated non-equivariant formal group laws by completion.
\begin{Lemma} \label{lem:strictiso} Let $A$ be an abelian compact Lie group, and $F$ and $F'$ be $A$-equivariant formal group laws over $k$, all of whose Euler classes $e_V$ with $V\neq \epsilon$ are invertible. Then the assignment $\varphi\mapsto \widetilde{\varphi}$ is a bijection from the set of strict isomorphisms $F\xr{\cong} F'$ to the set of strict isomorphisms $\widetilde{F}\xr{\cong} \widetilde{F}'$.
\end{Lemma}
\begin{proof} Let $R$ and $R'$ be the complete Hopf algebras of $F$ and $F'$, respectively, and $\widetilde{R}$ and $\widetilde{R}'$ their completions at the augmentation ideal. Then we saw in the proof of \Cref{prop:eulerinvert} that $R$ is canonically isomorphic to $\map(A^*,\widetilde{R})$ and likewise for $R'$, with $A^*$-action permuting the factors. Hence any isomorphism $F\cong F'$ is determined by the composite $R\to R'\to \widetilde{R}'$ under the $A^*$-action. This composite factors as $R\to \widetilde{R}$ followed by the induced isomorphism on completions $\widetilde{R}\xr{\cong} \widetilde{R}'$. Hence it is determined by the latter. Since $\varphi$ is strict if and only if $\widetilde{\varphi}$ is strict (the units $\lambda$ and $\widetilde{\lambda}$ have the same augmentation), the claim follows.
\end{proof}
\begin{Cor} The ring $\Phi^A(\bL^{(n)})$ is universal among rings equipped with an $n$-tuple of Euler-invertible $A$-equivariant formal groups $(F_1,\hdots,F_n)$ plus strict isomorphisms $\widetilde{F}_i\xr{\cong} \widetilde{F}_{i+1}$ between their associated non-equivariant formal group laws.
\end{Cor}

Using this, we show:
\begin{Prop} The map $\alpha^{(n)}\colon \bL^{(n)}\to \upi_*(\MU^{\wedge n})$ induces an isomorphism on all geometric fixed points.	
\end{Prop}
\begin{proof} Geometric fixed points are strong symmetric monoidal, so the maps $i_1,\hdots,i_n\colon \MU\to \MU^{\wedge n}$ induce an equivalence
	\[ \Phi^A(\MU^{\wedge n})\cong (\Phi^A(\MU))^{\wedge n}.\]
The spectrum $\Phi^A(\MU)$ is a wedge of shifted copies of $MU$ (see \cite[Theorem 4.9]{Sin01}), hence in particular it is Landweber exact. Now recall that given two Landweber exact theories $E$ and $E'$, the ring $(E\wedge E')_*$ is universal among rings $R$ equipped with two maps $E_*\to R$ and $E'_*\to  R$ and a strict isomorphism between the pushforwards of the given formal group laws over $E$ and $E'$. This follows from the case $(MU\wedge MU)_*$ by base-changing one copy of $MU$ to $E_*$ and the other to $E'_*$. Hence, \Cref{prop:mulazardgeom} shows that the coefficients of $(\Phi^A(\MU))^{\wedge n}$ are universal among rings equipped with an $n$-tuple of Euler-invertible $A$-equivariant formal group laws, together with strict isomorphisms between their associated non-equivariant formal group laws. We just saw that this is exactly the universal property of $\Phi^A(\bL^{(n)})$, so the claim follows.
\end{proof}

We now apply the same proof as in the previous section to conclude that $\bL^{(n)}\to \upi_*(\MU^{\wedge n})$ is an isomorphism, following \Cref{rem:generalize}. We argued above that the first coordinate map $\bL(A)\to \bL^{(n)}(A)$ is flat for all $A$, which proves that $\bL^{(n)}$ is also a regular global group law. Furthermore, we saw that the geometric fixed points $\Phi^A(\bL^{(n)})$ are integral domains for all $A$. Since $\bL$ is a retract of $\bL^{(n)}$, the different $\psi_i$ cannot be multiples of one another. Finally, $\upi_*(\MU^{\wedge n})$ is again concentrated in even degrees as a consequence of the results of \cite{CGK02}, hence it is regular. So we see that all the conditions of \Cref{rem:generalize} are satisfied, which concludes the proof of \Cref{thm:coops}.

\newcommand{\etalchar}[1]{$^{#1}$}

\end{document}